\documentclass[a4paper,11pt]{article}

\usepackage{scalerel}

\usepackage[T1]{fontenc}
\usepackage[utf8]{inputenc}
\usepackage{lipsum}
\usepackage{amsmath, amssymb, mathtools, amsthm}
\usepackage{comment}

\usepackage[usenames,dvipsnames]{xcolor}
\usepackage{graphicx}
\usepackage{minitoc}
\usepackage{tikz}
\usepackage{enumerate}
\usepackage{filecontents}
\usepackage[margin=0.96 in]{geometry}
\usepackage{bbm}

\usepackage[numbers,sort&compress]{natbib}
\usepackage[colorlinks=true]{hyperref}
\hypersetup{urlcolor=blue, citecolor=red}
\usepackage{mathtools}
\usepackage{float}
\usepackage{xspace}
\usepackage{wrapfig}

 \usepackage{epsfig}
 \usepackage{subcaption}
 \usepackage{multirow}
 \usepackage{lineno}
 \usepackage{fullpage}
 \usepackage[normalem]{ulem} 
 \usepackage{makeidx}

\usepackage[title]{appendix}
\usepackage{dsfont}

\DeclarePairedDelimiter{\prt}{(}{)}

\newcommand \commentout[1] {}

\newcommand{\partialt}[1]{\dfrac{\partial#1}{\partial t}}
\newcommand{\partialx}[1]{\dfrac{\partial#1}{\partial x}}

\DeclareMathAlphabet{\mathup}{OT1}{\familydefault}{m}{n}
\newcommand{\dx}[1]{\mathop{}\!\mathup{d} #1}

\theoremstyle{plain}
\newtheorem{thm}{Theorem}[section]

\theoremstyle{remark}

\newcommand{\ie}{\emph{i.e.}}
\newcommand{\cf}{\emph{cf.}\;}

\newcommand{\sign}{\mathrm{sign}}

\newcommand{\grad}{\nabla}
\renewcommand{\div}{\nabla\cdot}

\newcommand{\R}{\mathbb{R}}

\title{
An asymptotic preserving scheme for a tumor growth model of porous medium type}
\author{Noemi David\thanks{Sorbonne Universit{\'e}, Inria, CNRS, Universit\'{e} de Paris, Laboratoire Jacques-Louis Lions UMR7598, F-75005 Paris. Email: noemi.david@sorbonne-universite.fr.} \thanks{Dipartimento di Matematica, Universit\'a di Bologna, 40126 Bologna, Italy.}
\and
Xinran Ruan\thanks{DISMA
Dipartimento di Scienze Matematiche
"Giuseppe Luigi Lagrange", Politecnico di Torino, 10129 Torino, Italy. Email: xinran.ruan@polito.it.}}

\begin{document}
\maketitle
\begin{abstract}
Mechanical models of tumor growth based on a porous medium approach have been attracting a lot of interest both analytically and numerically. In this paper, we study the stability properties of a finite difference scheme for a model where the density evolves down pressure gradients and the growth rate depends on the pressure and possibly nutrients. Based on the stability results, we prove the scheme to be asymptotic preserving (AP) in the incompressible limit.

Numerical simulations are performed in order to investigate the regularity of the pressure. We study the sharpness of the $L^4$-uniform bound of the gradient, the limiting case being a solution whose support contains a bubble which closes-up in finite time generating a singularity, the so-called focusing solution.
\end{abstract}\vskip .4cm
\begin{flushleft}
    \noindent{\makebox[1in]\hrulefill}
\end{flushleft}
\begin{flushright}
    \noindent{\makebox[1in]\hrulefill}
\end{flushright}
\vskip 0.4cm
\section{Introduction}
We consider a model of tumor growth describing the evolution of the cell population density $n(x,t)$ through a porous medium equation with a source,
 \begin{equation}\label{eq: n}
             \partialt n - \div(n \nabla p) = n G(p), \quad \qquad x \in \R^d, t>0,
    \end{equation}
    where $p$ is the internal pressure of the tumor, defined by the law of state
  \begin{equation}\label{eq: pressure law}
    p= n^\gamma, \qquad  \gamma > 1.
\end{equation}   
The non-linearity and degeneracy of the diffusion term bring several difficulties to the numerical analysis of the model, and many numerical schemes have been proposed in the literature, \cf \cite{LTWZ, LTWZ19, Monsaingeon16}. 
In this paper, we investigate the properties of the solutions to Eq. \eqref{eq: n}, using the following upwind scheme
\begin{equation*}
    \frac{d}{dt}n_i=\frac{n_{i+1/2} q_{i+1/2} -n_{i-1/2} q_{i-1/2}}{\Delta x} +n_i G(p_i), \qquad
  \text{ with }\quad  q_{i+1/2}= \frac{p_{i+1}-p_i}{\Delta x}.
\end{equation*}
On the one hand, the simplicity of the scheme allows us to prove analytical properties which do not apply to more complex ones. We prove stability results and the asymptotics preserving (AP) property of the scheme as $\gamma\rightarrow\infty$. On the other hand, despite its simplicity, we perform numerical tests that show the good efficiency of the scheme for different reaction terms $G$ as well as for $\gamma\gg 1$.

We are also interested in analysing numerically the regularity of the so-called \textit{focusing solution} of Eq. \eqref{eq: n}, whose support is initially contained outside of a compact set, see for instance \cite{AG}. Due to the degeneracy of the diffusion, the inner hole closes up in finite time and singularities occur due to this topological change. In particular, we perform numerical tests to study the blow-up of the $L^p$-norms of the pressure gradient, which are uniformly (with respect to $\gamma$) bounded for $p\le 4$, as recently proved in \cite{DP}. This regularity is actually optimal, and the focusing solution represents the limiting case since the $L^p$-norms of its gradient blow up for $p>4$ as $\gamma\rightarrow\infty$. Our aim is to obtain a numerical verification of the study of the limiting exponents from \cite{DP}. 

\paragraph{Motivations.}
Models as Eq. \eqref{eq: n}, possibly including advection terms or coupled with a second equation, have been largely applied to the description of tissue and tumor growth. They are based on the  mechanical aspects that drive the cell motion and proliferation. Describing the fact that the cells move down pressure gradients, the flow velocity in Eq. \eqref{eq: n} is given by Darcy's law, namely $\Vec{v}=-\nabla p$. 

Besides driving the cells movement, the pressure also controls the cell proliferation through an inhibitory effect, since the division rate is lower at higher pressure values. Therefore, we make the following assumption on the growth rate $G$: there exist positive constants $\alpha$ and $p_H$ such that
\begin{equation}
    \label{eq: assumptions G}
   G'(p)\leq - \alpha, \qquad G(p_H)=0,
\end{equation}
where $p_H$ represents the so-called \textit{homeostatic pressure}, namely the lowest level of pressure that prevents cell multiplication due to contact inhibition. 

Later in the paper, we also consider an extension of the model where $G$ depends both on the pressure and the concentration of a nutrient (for instance, oxygen or glucose), denoted by $c(x,t)$. In this case, Eq. \eqref{eq: n} would be coupled with an equation on $c$ that depends both on the environmental conditions (\textit{in vitro} or \textit{in vivo}) and on the stage of the tumor development (\textit{avascular} or \textit{vascular}). We refer the reader to \cite{PTV} for the formulation of the Hele-Shaw problem with nutrient and its traveling wave solutions.
 
As mentioned above, the density actually satisfies a porous medium type equation, which can be directly recovered combining the pressure law, Eq. \eqref{eq: pressure law}, and Eq. \eqref{eq: n}, namely
\begin{equation*}
    \partialt n = \frac{\gamma}{\gamma+1} \Delta n^{\gamma+1} + n G(p).
\end{equation*}
As the solution of the classical Porous Medium Equation (PME), $n$ evolves with finite speed of propagation, since the diffusion term degenerates when $n=0$. Thus, if the initial data has compact support, the solution remains compactly supported at any time and exhibits a moving front, which is the interface that separates $\{n>0\}$ and $\{n=0\}$.

As shown in \cite{PQV}, as $\gamma \rightarrow \infty$ the solution of Eq. \eqref{eq: n} converges to a solution of the Hele-Shaw free boundary problem defined on the set $\Omega(t):=\{x, \ p_\infty(x,t)>0\}$, in which the limit pressure satisfies an elliptic equation.
The so-called \textit{incompressible limit} of Eq. \eqref{eq: n} has attracted a lot of interest in the last decades and a vast literature on the topic is now available, \cf \cite{PQV, DHV, KP17}. The Hele-Shaw limit has also been studied for several extension of the model at hand, we refer the reader to \cite{DP, activemotion, PV, BPPS, DebiecEtAl2021} for models including nutrients, viscosity, active motion or a second species of cells. 

The complete proof can be found in \cite{PQV, KP17}, while here we present a formal argument to explain the link between the compressible model and the free boundary formulation.
Upon multiplying Eq. \eqref{eq: n} by $\gamma n^{\gamma-1}$, we recover the equation satisfied by the pressure, which reads
\begin{equation}\label{eq: p} 
    \partialt p = \gamma p (\Delta p + G(p)) + |\nabla p|^2.
\end{equation}
Then passing formally to the limit $\gamma\rightarrow \infty$ we find the \textit{complementarity relation}
\begin{equation*}
    p_\infty (\Delta p_\infty + G(p_\infty))=0.
\end{equation*}
This implies that the limit pressure has to satisfy the elliptic equation $-\Delta p_\infty =G(p_\infty)$ in the tumor region $\Omega(t)$.

\paragraph{Our contribution.}
\mbox{}
\vskip0.2cm
\noindent{$\circ$ \textit{Asymptotic preserving property.}}
In this paper, we show that, as $\gamma\rightarrow\infty$, the aforementioned scheme is asymptotic preserving and the solution converges to a solution of the following finite difference equation
\begin{equation*}
       p_{i} (\delta_x^2p_{i} +G(p_{i}))=0.
\end{equation*} 
\vskip0.2cm
\noindent{$\circ$ \textit{Aronson-B\'enilan estimate.}}
The derivation of the complementarity relation in the continuous case is deeply related to a lower bound on the quantity $w:=\Delta p + G(p)$, namely $w \gtrsim -\frac{C}{\gamma t}$, \cf \cite{PQV}. This bound is an adaptation of the Aronson-B\'enilan (AB in short) estimate, which is a well-known and powerful tool in the theory of porous medium equations. 

It is our aim to recover a discrete version of this lower bound for our scheme. This purpose has been already addressed in the literature, in particular we refer the reader to \cite{Monsaingeon16} for a tracking front scheme for which the author proves the Aronson-B\'enilan estimate for the classical porous medium equation (namely, with no reaction terms), and for any $\gamma>1$. Unlike \cite{Monsaingeon16}, we keep a fixed grid and show that the AB estimate holds also for a restricted class of pressure-penalized growth rates $G=G(p)$, only in the cases $\gamma=1$ and $\gamma \approx \infty$ which is our interest for the Hele-Shaw limit. To the best of our knowledge, we are the first to prove the discrete version of the AB estimate for a nontrivial pressure-dependent reaction term in the porous medium equation. It is not the main goal of this paper to prove the convergence of the scheme as $\Delta x \rightarrow 0$, nevertheless, we want to point out that this estimate could be extremely useful in this direction.
\vskip0.2cm 
\noindent{$\circ$ \textit{Focusing solution.}}
The solutions of Eq. \eqref{eq: n} exhibit different kind of singularities in the incompressible limit $\gamma\rightarrow\infty$. For instance, the limit density $n_\infty$ shows jump discontinuities across the boundary of the tumor region $\partial \Omega(t)$, while the pressure $p_\infty$ can develop singularities in time. In fact, when a new saturated region is generated outside $\Omega(t)$, \textit{i.e.} when $n_\infty(\cdot,s)$ becomes 1 in a set of positive measure contained outside the original tumor region, for some $s>t$, the pressure instantaneously becomes positive in the same set, 
according to the relation $p_\infty(1-n_\infty)=0$.
Moreover, time discontinuities can also appear when the set $\Omega(t)$ undergoes certain topological changes, for instance when the support contains a hole which closes up at time $t=T^*$, which is called \textit{focusing time}. This particular solution is referred to as  \textit{focusing solution}. The hole filling problem has attracted a lot of attention since it represents the limiting case for several regularity results. For instance, in \cite{AG}, Aronson and Graveleau use the focusing solution to show that the H\"older continuity of the pressure gradient is optimal, for dimension $d\geq 2$. In fact, the pressure gradient blows up at the focusing time $T^*$.

In \cite{DP}, the authors prove that the $L^4$-norm of $\nabla p$ is uniformly bounded with respect to $\gamma$. Then, they show that this uniform estimate is optimal using the focusing solution as a counter-example. Through an asymptotic argument on a radial solution, they compute that 4 is the highest possible order of integrability for the gradient of the pressure of the Hele-Shaw problem. 
One of the main interests of this paper is to numerically investigate and confirm this property of the focusing solution. To this end, we perform 2-dimensional simulations with initial data given by the characteristic function of a spherical shell. The results obtained by computing the $L^p$-norms of the pressure gradient clearly show its singularity at the focusing time and confirm the worsening of the blow-up as the exponents becomes greater than 4. 
 At the best of our knowledge, there are no numerical inspections of this sharpness result in the literature, although the focusing solution has been deeply studied both analytically and numerically \cite{AG, AGV, ARONSON2016}.

\paragraph{Previous works.}  
The numerical simulation of the tumor growth model \eqref{eq: n} is challenging in two aspects, the lack of regularity of solutions near the free boundary, which is a common difficulty of porous medium equations, and the stiffness appearing in the Hele-Shaw limit $\gamma\to\infty$.

The numerical study of the porous medium equations lasts for decades and a variety of algorithms have been proposed. 
An early study of the finite difference method can be found in \cite{GraveleauJamet}. 
Further studies on the finite difference method include the interface tracking algorithm, \cf \cite{BenedettoHoff, Monsaingeon16}, which works perfectly in 1D by separating the computation of the free boundaries and the solutions inside the support, 
 a WENO scheme, \cite{LiuShuZhang}, which eliminates the oscillations around the free boundaries,  
 and so on. 
There is also an extensive study on the finite volume method, \cf \cite{BessemoulinFilbet, EymardGalloutHerbinMichel} and various finite element methods, including 
an early study of the convergence analysis, \cite{Rose}, 
the locally discontinuous Galerkin method, \cite{ZhangWu}, and the adaptive mesh, \cite{BainesHubbardJimackJones, BainesHubbardJimack, NgoHuang}. 
The relaxation scheme, which is originally designed for conservation laws, \cite{JinXin}, can be extended to porous medium equations successfully as well, \cite{CavalliNaldiPuppoSemplice, NaldiPareschiToscani}.
Besides the methods on Euler coordinates, there is an increasing interest in designing Lagrangian methods, see for example \cite{BuddCollinsHuangRussell, CarrilloRanetbauerWolfram, Carrillo2, Carrillo3, LiuWang}.  
Despite the extensive study of the numerical methods for porous medium equations, the algorithm preserving the free boundary limit is rarely studied. 
A fully implicit solver is generally needed. 
A recent work shows that one way to avoid a fully implicit scheme is to construct a semi-implicit scheme by combining the relaxation scheme with the prediction-correction formulation, \cf \cite{LTWZ}. 

\paragraph{Contents of the paper.} 
The semi-discrete scheme and the analysis of its properties are presented in Section \ref{sec: semi-discrete scheme}. We prove stability providing a priori estimates on the main quantities and their derivatives, Subsec.~\ref{sec: stability}. Let us point out that these estimates are uniform with respect to $\gamma$, and therefore stability holds for any $\gamma>1$. Then, we prove the asymptotic preserving property of the scheme, Subsec.~\ref{sec: AP} and recover a discrete version of the Aronson-B\'enilan estimate for a nontrivial reaction term,
Subsec.~\ref{sec: AB}. We introduce the implicit scheme in Section~\ref{sec: implicit}, and we extend the uniform a priori estimates previously derived on the semi-discrete scheme. The solvability of the scheme is proven in detail in Appendix~\ref{appendix_existence}. We report the results of several numerical simulations in Section~\ref{sec: numerical simulations}. We test the accuracy of the scheme using the explicit Barenblatt profile, and we compare the numerical solutions with $\gamma$ large to the exact solutions of the \textit{in vitro} and \textit{in vivo} model with nutrients. Moreover, we apply our scheme to a two-species model of tumor growth, where both populations evolve under a porous medium mechanics.
Finally, we report the results of the 2-dimensional simulations on the focusing solution which confirm the sharpness of the $L^4$-uniform bound of $\nabla p$. 
\section{The semi-discrete scheme}\label{sec: semi-discrete scheme} 
To better focus on the analysis of the upwind discretization in space, we start from the semi-discrete scheme. 
For simplicity, only the one dimensional problem is considered. The scheme for the multi-dimensional problem can be analyzed similarly.

We suppose the domain is a closed interval $\Omega = [-X, X]$.  We choose a uniform mesh with mesh size $\Delta x = \frac{2X}{M_x}$, where $M_x$ is the number of sub-intervals. 
Denote $n_i(t)$ and $p_i(t)$ to be the numerical approximations of $n(t, x_i)$ and $p(t,x_i)$, where $x_i=i\Delta x$ for $i\in I =\{-M_x, -M_x+1, \dots, M_x\}$.
Then the semi-discrete finite difference scheme for Eq.~\eqref{eq: n} is 
\begin{equation}\label{eq: dtni}
    \frac{d}{dt}n_i=\frac{n_{i+1/2} q_{i+1/2} -n_{i-1/2} q_{i-1/2}}{\Delta x} +n_i G_i, \quad i\in I,
\end{equation}
with 
\begin{equation}\label{eq: q term}
    q_{i+1/2}= \frac{p_{i+1}-p_i}{\Delta x}, \quad
    G_i=G(p_i).
\end{equation}
The Neumann boundary condition is applied so that $n_{\scaleto{-M_x-1}{5pt}} = n_{\scaleto{-M_x+1}{5pt}}$ and $n_{\scaleto{M_x+1}{5pt}} = n_{\scaleto{M_x-1}{5pt}}$.
We define $n_{i+\frac{1}{2}}$ in the upwind manner
\begin{equation}\label{eq: nionehalf}
    n_{i+1/2} =\begin{cases}
    n_i, \quad &\text{ if } \; q_{i+1/2}<0,\\
    n_{i+1}, \quad &\text{ if } \; q_{i+1/2}>0.
    \end{cases}
\end{equation}
Multiplying Eq.~\eqref{eq: dtni} by $\gamma n_i^{\gamma-1}$  we recover the finite difference equation on the pressure
\begin{equation*} 
    \frac{d}{dt}p_i=\gamma n_i^{\gamma -1}\left(\frac{n_{i+1/2}-n_i}{\Delta x}q_{i+1/2}+\frac{n_i-n_{i-1/2}}{\Delta x}q_{i-1/2}\right)+\gamma p_i \left(\delta_x^2p_i +G_i\right),
\end{equation*}
where
\begin{equation*}
    \delta_x^2 p_i := \frac{q_{i+1/2}-q_{i-1/2}}{\Delta x}.
\end{equation*}

\paragraph{Assumptions.}
We assume that there exists positive constants $C$ and $p_H$ (homeostatic pressure) such that
\begin{equation}\label{eq: assumptions initial data}\begin{split}
 0\leq p_i^0\leq p_H, \qquad \Delta x\sum_i |n_i^0|  \leq C \qquad \Delta x\sum_i |p_i^0|  \leq C,\\
 \Delta x\sum_i |n^0_{i+1}-n^0_i| \leq C, \qquad \Delta x\sum_i \left|\left(\frac{d}{dt}n_i\right)^0\right|\leq C.
 \end{split}
\end{equation}
\subsection{Stability results}\label{sec: stability}
Now we prove the positivity preserving property of the semi-discrete Scheme~\eqref{eq: dtni}, and the \textit{a priori} estimates that imply stability for any $\gamma>1$.
\begin{thm}[A priori estimates] \label{thm: a priori}
Let $T>0$ and $n_H:=p_H^{1/\gamma}$. Then, for all $0\leq t\leq T$ it holds
\begin{itemize}
    \item[(i)]  $0\leq n_i(t)\leq n_H, \; 0\leq p_i(t)\leq p_H$, $\forall i$,
    \item[(ii)] $
   \Delta x  \sum_i |n_i(t)|\leq C(T),$ $\;
    \Delta x  \sum_i |p_i(t)| \leq C(T)$,
    \item[(iii)]  $\Delta x \sum_i  |n_{i+1}(t)-n_i(t)|\leq C(T),$
    \item[(iv)]   $\Delta x  \sum_i \left|\frac{d}{dt}n_i(t)\right| \leq C(T)$, \; $\int_0^T \Delta x  \sum_i \left|\frac{d}{dt}p_i\right| \dx{t}\leq C(T)$,
    \item[(v)] $\int_0^T  \Delta x \sum_i\left|\frac{p_{i+1}-p_i}{\Delta x}\right|^2\dx{t} \leq C(T).$
\end{itemize}
for some positive constants $C(T)$ depending on $T$ and independent of $\gamma$.
\end{thm}
\begin{proof}
 
\textbf{$\boldsymbol{L^\infty}$ estimates.}
Combining Eq.~\eqref{eq: q term} and Eq.~\eqref{eq: nionehalf} we recover
\begin{equation*}
 \gamma n_i^{\gamma-1}\frac{n_{i+1/2}-n_i}{\Delta x}q_{i+1/2}   =\begin{cases}
0 \quad &\text{ if } q_{i+1/2}<0,\\
\gamma n_i^{\gamma -1}(n_{i+1}-n_i)(p_{i+1}-p_i) \quad&\text{ if } q_{i+1/2}>0,
\end{cases}
\end{equation*}
and
\begin{equation*}
 \gamma n_i^{\gamma-1}\frac{n_i-n_{i-1/2}}{\Delta x}q_{i-1/2} =\begin{cases}
\gamma n_i^{\gamma -1}(n_i-n_{i-1})(p_i-p_{i-1}) \quad &\text{ if } q_{i-1/2}<0,\\
0  \quad&\text{ if } q_{i-1/2}>0.
\end{cases}
\end{equation*}
Then, the following inequality holds
\begin{equation}\label{eq: ineq pressure pi}
    \frac{d}{dt}p_i \leq \left|q_{i+1/2}\right|^2_+ +\left|q_{i-1/2}\right|^2_-+\gamma p_i (\delta_x^2p_i +G_i).
\end{equation}
Since $0\leq p_i^0\leq p_H \; \forall i$, then $0\leq p_i(t)\leq p_H$ for every $t >0$ and consequently $0\leq n_i(t)\leq p_H^{1/\gamma}$. Thus, point \textit{(i)} is proven.
\\
\textbf{$\boldsymbol{L^1}$ estimate.}
To prove estimates \textit{(ii)}, we compute the sum of Eq.~\eqref{eq: dtni} for all $i$, and we find successively
\begin{equation*}
    \frac{d}{dt}\prt*{\Delta x\sum_i n_i} = \sum_i (n_{i+1/2}q_{i+1/2} -n_{i-1/2}q_{i-1/2}) +\Delta x\sum_i n_i G(p_i)  =\Delta x\sum_i n_i G(p_i),
\end{equation*}
\begin{equation*}
    \frac{d}{dt} \left(\Delta x\sum_i n_i\right) \leq G(0) \Delta x\sum_i n_i,
\end{equation*}
where in the last inequality we use the assumptions on the growth term, \cf Eq.~\eqref{eq: assumptions G}.
By Gronwall's lemma and Eq.~\eqref{eq: assumptions initial data}, we have
\begin{equation*}
    \Delta x\sum_i |n_i(t)| \leq e^{G(0)t} \Delta x\sum_i |n^0_i|\leq C(T), \text{ for } 0\leq t\leq T.
\end{equation*}
Upon using the $L^\infty$-bound of the pressure, we finally obtain
\begin{equation*}
   \Delta x \sum_i |p_i(t)| \leq  p_H^{(\gamma-1)/\gamma} \Delta x \sum_i |n_i(t)|\leq C(T).
\end{equation*}
\\
\textbf{$\boldsymbol{BV}$ space estimate.}
We now subtract the equation for $n_i$ from the equation for $n_{i+1}$ and multiply by $\sign(n_{i+1}-n_i)$
\begin{align*}
   \Delta x \frac{d}{dt} |n_{i+1}-n_i|  \le &(n_{i+3/2} |q_{i+3/2}|-2n_{i+1/2} |q_{i+1/2}| + n_{i-1/2} |q_{i-1/2}|)\\
    &+ \Delta x(n_{i+1} G(p_{i+1})-n_i G(p_i))\sign(n_{i+1}-n_i)).
\end{align*}
We sum over $i$ to obtain
\begin{align*}
    \frac{d}{dt} \prt*{\Delta x  \sum_i  |n_{i+1}-n_i|}\leq &\sum_i (n_{i+3/2} |q_{i+3/2}|-2n_{i+1/2} |q_{i+1/2}| + n_{i-1/2} |q_{i-1/2}|)\\
    &+ \Delta x\sum_i (|n_{i+1}-n_i|G(p_i)+n_{i+1} (G(p_{i+1})-G(p_i))\sign(n_{i+1}-n_i))\\
    \leq& \Delta x\sum_i |n_{i+1}-n_i|G(p_i),
\end{align*}
where in the last inequality we use the monotonicity of $G$, Eq.~\eqref{eq: assumptions G}.
Finally, we get
\begin{equation*}
     \frac{d}{dt} \Delta x\sum_i  |n_{i+1}-n_i|  \leq G(0) \Delta x\sum_i |n_{i+1}-n_i|,
\end{equation*}
and thus we recover \textit{(iii)} thanks to Gronwall's lemma and the assumptions on the initial data, Eq.~\eqref{eq: assumptions initial data},
\begin{equation*}
  \Delta x \sum_i |n_{i+1}(t)-n_i(t)| \leq e^{G(0)t} \Delta x\sum_i |n^0_{i+1}-n^0_i|\leq C(T), \text{ for } 0\leq t\leq T.
\end{equation*}
\\
\textbf{Estimates on the time derivatives.} Now we give the proof of the boundedness of the time derivatives, \textit{(iv)}.
Deriving Eq.~\eqref{eq: dtni} with respect to time, we obtain
\begin{equation*}
   \frac{d}{dt}\left(  \frac{d}{dt}n_i\right)  \Delta x=\frac{d}{dt}\left(n_{i+1/2} q_{i+1/2}-n_{i-1/2}q_{i-1/2})+n_i G(p_i) \Delta x\right).
\end{equation*}
We multiply by $\sign\left(\frac{d}{dt}n_i\right)$ 
\begin{equation} \label{nt}
\begin{split}
 \frac{d}{dt}\left( \left|\frac{d}{dt}n_i\right|\right)  \Delta x =& \underbrace{\frac{d}{dt}\left(n_{i+1/2}q_{i+1/2}\right)\sign\left(\frac{d}{dt}n_i\right)}_{A_i}\underbrace{-\frac{d}{dt}\left(n_{i-1/2}q_{i-1/2}\right)\sign\left(\frac{d}{dt}n_i\right)}_{B_i}\\
    &+\left(G(p_i) \left|\frac{d}{dt}n_i\right| + n_i G'(p_i) \left|\frac{d}{dt}p_i\right|\right) \Delta x.
    \end{split}
\end{equation}
We now compute $A_i$ and $B_i$
\begin{align*}
    A_i&=\left[\frac{d}{dt}n_{i+1/2}\,|q_{i+1/2}|_+ - \frac{d}{dt}n_{i+1/2}\, |q_{i+1/2}|_- + n_{i+1/2} \frac{d}{dt}q_{i+1/2}\right]\sign\left(\frac{d}{dt}n_i\right)\\
    &\leq \left|\frac{d}{dt}n_{i+1}\right||q_{i+1/2}|_+ - \left|\frac{d}{dt}n_i\right||q_{i+1/2}|_- +\frac{n_{i+1/2}}{\Delta x} \left|\frac{d}{dt}p_{i+1}\right|-\frac{n_{i+1/2}}{\Delta x} \left|\frac{d}{dt}p_{i}\right|\\
    B_i&=\left[-\frac{d}{dt}n_{i-1/2}\,|q_{i-1/2}|_+ + \frac{d}{dt}n_{i-1/2}\, |q_{i-1/2}|_- - n_{i-1/2} \frac{d}{dt}q_{i-1/2}\right]\sign\left(\frac{d}{dt}n_i\right)\\
    &\leq -\left|\frac{d}{dt}n_{i}\right||q_{i-1/2}|_+ + \left|\frac{d}{dt}n_{i-1}\right||q_{i-1/2}|_- -\frac{n_{i-1/2}}{\Delta x} \left|\frac{d}{dt}p_{i}\right|+\frac{n_{i-1/2}}{\Delta x} \left|\frac{d}{dt}p_{i-1}\right|.
\end{align*}
Upon summing over $i$, we find
\begin{equation*}
    \sum_i (A_i+B_i)\leq 0,
\end{equation*}
and then, from Eq.~\eqref{nt}, we have
\begin{align*}
    \frac{d}{dt}\left( \Delta x\sum_i \left|\frac{d}{dt}n_i\right|\right)\leq \Delta x\sum_i G(p_i) \left|\frac{d}{dt}n_i\right|,
\end{align*}
since $G'$ is negative.
Hence, we obtain
\begin{equation}\label{bvtn}
     \Delta x\sum_i \left|\frac{d}{dt}n_i\right| \leq e^{G(0)t}\Delta x\sum_i \left|\left(\frac{d}{dt}n_i\right)^0\right|\leq C(T), \text{ for } 0\leq t\leq T. 
\end{equation}
It remains to prove the estimate on the time derivative of the pressure.
We compute
\begin{equation}\label{bvp}
	\int_0^T \Delta x\sum_i \left|\frac{d}{dt}p_i\right| \dx{t}\leq \int_0^T \Delta x \sum_i \gamma n_i^{\gamma-1}\left|\frac{d}{dt}n_i\right|\mathds{1}_{\{n_i\leq 1/2\}}\dx{t}  + 2\int_0^T \Delta x\sum_i n_i\left|\frac{d}{dt}p_i\right|\mathds{1}_{\{n_i\geq 1/2\}}\dx{t}.
\end{equation}
Thanks to Eq.~\eqref{bvtn} the first term in the right-hand side is bounded. 

Let us denote $\beta:=\min_i |G'(p_i)|$. We sum Eq.~\eqref{nt} over $i$ and we integrate in time to obtain
\begin{equation*}
	\Delta x\sum_i \left|\frac{d}{dt}n_i\right| + \beta \int_0^T \Delta x\sum_i n_i \left|\frac{d}{dt}p_i\right|\dx{t} \leq G(0)\int_0^T \Delta x\sum_i\left|\frac{d}{dt}n_i\right|\dx{t} + \Delta x\sum_i \left|\left(\frac{d}{dt}n_i\right)^0\right|\leq C(T),
\end{equation*}
where the last inequality comes from Eq.~\eqref{bvtn}.
Thanks to this bound, we know that
\begin{equation*}
	\int_0^T \Delta x\sum_i n_i \left|\frac{d}{dt}p_i\right| \dx{t}\leq C(T),
\end{equation*}
and from Eq.~\eqref{bvp} we finally find
\begin{equation*}
		\int_0^T \Delta x\sum_i \left|\frac{d}{dt}p_i\right|\dx{t}\leq C(T).
\end{equation*}
\paragraph{$\boldsymbol{L^2}$ estimate on the pressure gradient.}
We sum for all $i$ the inequality satisfied by the pressure, Eq.~\eqref{eq: ineq pressure pi}, namely 
\begin{align*}
    \sum_i \frac{d}{dt} p_i &\leq \sum_i \left|\frac{p_{i+1}-p_i}{\Delta x}\right|_+^2 +\sum_i \left|\frac{p_i-p_{i-1}}{\Delta x}\right|_-^2 +\gamma \sum_i p_i (\delta_x^2 p_i +G_i)\\
    &\leq \sum_i\left|\frac{p_{i+1}-p_i}{\Delta x}\right|^2 +\gamma \sum_i \frac{p_i p_{i+1}-2p_i^2+p_{i-1}p_i}{|\Delta x|^2} +\gamma \sum_i p_i G_i \\
     &=\sum_i\left|\frac{p_{i+1}-p_i}{\Delta x}\right|^2 -\gamma \sum_i\left|\frac{p_{i+1}-p_i}{\Delta x}\right|^2 +\gamma \sum_i p_i G_i.
\end{align*}
Hence, we have
\begin{equation*}
  (\gamma -1)  \sum_i\left|\frac{p_{i+1}-p_i}{\Delta x}\right|^2 \leq - \sum_i \frac{d}{dt} p_i +\gamma \sum_i p_i G_i,  
\end{equation*}
and, upon integrating in time, we recover
\begin{equation*}
 \int_0^T  \Delta x \sum_i\left|\frac{p_{i+1}-p_i}{\Delta x}\right|^2 \dx{t}  \leq   \frac{1}{\gamma -1}  \Delta x \sum_i p^0_i   -  \Delta x \sum_i p_i(T)  +\frac{\gamma}{\gamma-1} \int_0^T  \Delta x \sum_i p_i G_i \dx{t}.  
\end{equation*}
Thus \textit{(v)} follows from the assumptions on $G$ and $p_i^0$, \cf Eqs.~(\ref{eq: assumptions G}, \ref{eq: assumptions initial data}), and the estimates \textit{(ii)} proven above.

\end{proof}
\subsection{The asymptotic-preserving property}\label{sec: AP}
As mentioned in the introduction, it is well-known that when $\gamma\rightarrow\infty$ the porous medium-type equation~\eqref{eq: n} turns out to be a free boundary problem of Hele-Shaw type. In particular, passing to the limit in the equation of the pressure
\begin{equation*}
    \partialt p = \gamma p (\Delta p + G(p)) + |\nabla p|^2,
\end{equation*}
allows to recover the  \textit{complementarity relation}, namely
\begin{equation*}
    p_\infty(\Delta p_\infty + G(p_\infty))=0,
\end{equation*}
in the sense of distributions.

We show that the semi-discrete scheme~\eqref{eq: dtni} satisfies the same property and thus is asymptotic preserving (AP) as $\gamma\rightarrow\infty$.
First of all, let us prove the following convergence result (where we point out the dependence of the solution on $\gamma$ in the notation).
\begin{thm}[Convergence result]\label{thm: convergence}
Given $n_{\gamma,i}, p_{\gamma,i}$ a solution of scheme~\eqref{eq: dtni} with $\gamma>1$. Then, for all $i$, we have
\begin{align*}
    n_{\gamma,i} &\xrightarrow[]{\gamma\rightarrow\infty} n_{\infty,i}, \quad \text{ in } L^p(0,T), \; \text{for all} \;\; 1\leq p <\infty, \\[0.5em]
    p_{\gamma,i} &\xrightarrow[]{\gamma\rightarrow\infty} p_{\infty,i}, \quad \text{ in } L^p(0,T), \; \text{for all} \;\; 1\le p<\infty, \\[0.5em]
    q_{\gamma,i+\frac 1 2} & \xrightharpoonup[]{\gamma\rightarrow\infty} q_{\infty,i+\frac{1}{2}},\;\;  \text{ weakly in } L^2(0,T).
\end{align*}
\end{thm}
\begin{proof}
Thanks to the uniform bounds \textit{(ii), (iv)} stated in Theorem~\ref{thm: a priori}, by standard compactness arguments we infer the convergence of $n_{\gamma,i}$ and $p_{\gamma,i}$ in $L^1(0,T)$. Since both the density and the pressure are bounded uniformly in $L^\infty(0,T)$, they converge strongly, up to a subsequence, in any $L^p(0,T)$, with $1\leq p <\infty$. 

Finally, the a priori bound \textit{(v)} of Theorem~\ref{thm: a priori} yields the weak convergence of $ q_{\gamma,i+\frac 1 2}$ in $L^2(0,T)$.

\end{proof}

Now we prove the asymptotic preserving property of the scheme. 
First of all, let us recall the equation satisfied by the pressure
\begin{equation}\label{eq: pressure scheme}
    \frac{d}{dt}p_i - \gamma n_i^{\gamma -1}\left(\frac{n_{i+1/2}-n_i}{\Delta x}q_{i+1/2}+\frac{n_i-n_{i-1/2}}{\Delta x}q_{i-1/2}\right) = \gamma p_i \left(\delta_x^2p_i +G_i\right).
\end{equation}
Since
\begin{equation*}
       \left|\gamma n_i^{\gamma -1}\left(\frac{n_{i+1/2}-n_i}{\Delta x}q_{i+1/2}+\frac{n_i-n_{i-1/2}}{\Delta x}q_{i-1/2}\right)\right| \leq \left|q_{i+1/2}\right|^2_+ +\left|q_{i-1/2}\right|^2_-,
\end{equation*}
thanks to Theorem~\ref{thm: a priori} we know that the left-hand side of Eq.~\eqref{eq: pressure scheme} is uniformly bounded in $L^1(0,T)$.
Testing Eq.~\eqref{eq: pressure scheme} against a function $\varphi\in C_{c}^1(0,T)$, we obtain
\begin{align*}
 & \int_0^T\gamma p_i \left(\delta_x^2p_i +G(p_i)\right)\varphi \dx{t}\\
&=-\frac 1 \gamma \left(\int_0^T p_i \varphi' \dx{t} -\int_0^T \gamma n_i^{\gamma -1}\left(\frac{n_{i+1/2}-n_i}{\Delta x}q_{i+1/2}  +\frac{n_i-n_{i-1/2}}{\Delta x}q_{i-1/2} \right)\varphi \dx{t}\right).
\end{align*}
Hence, passing to the limit $\gamma\rightarrow\infty$ using Theorem~\ref{thm: convergence}, we recover
\begin{equation*}
       p_{\infty,i} \ (\delta_x^2 p_{\infty,i} +G(p_{\infty,i}))=0,
\end{equation*} 
which is the discrete formulation of the \textit{complementarity relation}.

We now pass to the limit also in the equation for the density, which reads
\begin{equation*}    \frac{d}{dt}n_i=\frac{n_{i+1/2} q_{i+1/2} -n_{i-1/2} q_{i-1/2}}{\Delta x} +n_i G_i.
\end{equation*}
Multiplying by a test function, we obtain
\begin{equation*}
    -\int_0^T n_i \varphi' \dx{t} =\int_0^T \frac{n_{i+1/2} q_{i+1/2} -n_{i-1/2} q_{i-1/2}}{\Delta x} \varphi\dx{t} + \int_0^T n_i G(p_i) \varphi \dx{t},
\end{equation*}
hence, thanks to Theorem~\ref{thm: convergence}, we find (in the weak sense)
\begin{equation*}
    \frac{d}{dt}n_{\infty,i}=\frac{n_{\infty,i+1/2} q_{\infty,i+1/2} -n_{\infty,i-1/2} q_{\infty,i-1/2}}{\Delta x} +n_{\infty,i} G(p_{\infty,i}).
\end{equation*} 
\subsection{Stronger estimate on the pressure - The Aronson-B\'enilan estimate}\label{sec: AB}
In \cite{PQV}, Perthame, Quir\'os and V\'azquez recover the compactness needed to pass to the limit in Eq.~\eqref{eq: p} relying on a lower bound on the Laplacian of the pressure. In fact, they extend the celebrated Aronson-B\'enilan estimate of the PME to the case of non-trivial reaction term, \ie \ $G\neq 0$, proving the following bound
\begin{equation}\label{eq: AB}
    \Delta p + G(p) \gtrsim -\frac{C}{\gamma t}.
\end{equation}
It is our interest to investigate whether this lower bound on the second derivatives still holds for Eq.~\eqref{eq: dtni}, in order to obtain a discrete counterpart of a fundamental property of porous medium-type equations.

We are able to prove the discrete version of the Aronson-B\'enilan estimate, Eq.~\eqref{eq: AB}, for $\gamma=1$ and $\gamma \approx \infty$ and for a pressure-dependent growth term of the form $G(p)=\alpha (p_H-p)$. It remains an open question how to recover the discrete AB estimate for $\gamma>1$ and for a general reaction term $G$. The discrete version of the AB property for non-trivial reaction terms could be extremely useful in order to pass to the limit as $\Delta x$ vanishes and therefore to prove the convergence of the scheme. 
\begin{thm}[Aronson-B\'enilan estimate]
Let $G(p)=\alpha (p_H-p)$, with $\alpha\geq 0$. We set 
\begin{equation*}
   w_i:=
     \delta_x^2 p_i + G(p_i)= \frac{p_{i+1}-2p_i + p_{i-1}}{(\Delta x)^2}+ G(p_i), \qquad \forall i.
\end{equation*}
Then, for $\gamma = 1$ and $\gamma \approx \infty$, Scheme~\eqref{eq: dtni} satisfies the Aronson-B\'enilan estimate, i.e. 
\begin{equation*}
    w_i \geq -\frac{1}{\gamma t}.
\end{equation*}
\end{thm}
 \begin{proof}
 As in \cite{PQV}, it is sufficient to prove 
 \begin{equation*}
     \frac{d \underline{w}}{dt} \ge \gamma(\underline{w})^2, \qquad \text{with }   \underline{w} := \min_i\{w_i\}.
 \end{equation*}
\noindent{$\bullet$ \textbf{Case} $\boldsymbol{\gamma=1}$.}

We have $p_i = n_i$ and thus Scheme~\eqref{eq: dtni} can be reformulated as 
 \begin{equation*}
 	\frac{d p_i}{dt} = p_i w_i + (q_{i+\frac12}^+)^2 +(q_{i-\frac12}^-)^2, 
 \end{equation*}	
where $q_{i+\frac12}^+ = \max\{q_{i+\frac12}, 0\}$ and $q_{i-\frac12}^- = \max\{-q_{i-\frac12}, 0\}$,
and it further implies that
 \begin{align}\label{eq: wi}
 	\frac{d w_i}{dt} &= \delta_x^2 (p_i w_i) + \delta_x^2[(q_{i+\frac12}^+)^2] + \delta_x^2[(q_{i-\frac12}^-)^2]-\alpha p_i w_i -\alpha(q_{i+\frac12}^+)^2 -\alpha(q_{i-\frac12}^-)^2.
 \end{align}
In order to consider the evolution of the minimal $w_i$ we denote 
 \begin{equation*}
 	w_j := \min_i w_i .
 \end{equation*} 
On the one hand, it is easy to see that \begin{equation}\label{eq: inequality for AB 1}
 	\delta_x^2 (p_j w_j) \ge w_j \delta_x^2 p_j. 
 \end{equation}
On the other hand, by definition $$\ w_j = \dfrac{q_{j+\frac12} - q_{j-\frac12}}{\Delta x} + \alpha (p_H - p_j),$$ and the inequality $w_j \le  w_{j+1}$ indicates that
 \begin{equation*}
 	  q_{j+\frac32} + q_{j-\frac12} \ge q_{j+\frac12}(2 + \alpha |\Delta x|^2).
 \end{equation*}
 As a result, 
  \begin{align*}
 	  q_{j+\frac32}^+ + q_{j-\frac12}^+ &\ge \max\{q_{j+\frac32} + q_{j-\frac12}, \, 0\} \\
	  	& \ge \max\{ q_{j+\frac12}(2 + \alpha |\Delta x|^2), \, 0\}\\
	  	& =  q_{j+\frac12}^+ (2 + \alpha |\Delta x|^2).
 \end{align*}
 And then, by Jensen's inequality, we get 
 \begin{align*}
 	(q_{j+\frac32}^+)^2 + (q_{j-\frac12}^+)^2 &\ge (q_{j+\frac12}^+)^2 (2 + \alpha |\Delta x|^2),  
 \end{align*}
 or equivalently, \begin{equation}
     \label{eq: inequality for AB 2}
     \delta_x^2[(q_{j+\frac12}^+)^2] \ge \alpha (q_{j+\frac12}^+)^2.
 \end{equation} 
 Similarly, we recover $\delta_x^2[(q_{j-\frac12}^-)^2] \ge \alpha (q_{j-\frac12}^-)^2$. 
Upon combining Eq.~\eqref{eq: inequality for AB 1} with Eq.~\eqref{eq: wi} and adding and subtracting $G(p_j) w_j$, we get 
 \begin{equation*}
 \frac{d w_i}{dt} \geq w_i^2 - G(p_j) w_j + \delta_x^2[(q_{i+\frac12}^+)^2] + \delta_x^2[(q_{i-\frac12}^-)^2]-\alpha p_i w_i -\alpha(q_{i+\frac12}^+)^2 -\alpha(q_{i-\frac12}^-)^2,
 \end{equation*}
 which yields
  \begin{equation*}
 \frac{d w_i}{dt} \geq w_i^2 - G(p_j) w_j +  -\alpha p_i w_i,
 \end{equation*}
 thanks to Eq.~\eqref{eq: inequality for AB 2}.
 Finally, using the definition of $G$ and assuming without loss of generality that $w_j\leq0$, we obtain
 \begin{equation*}
     w_j \ge w_i^2,
 \end{equation*}
 which implies
  \begin{equation*}
 	w_j\ge -\frac{1}{t}.
 \end{equation*}
\noindent{$\bullet$ \textbf{Case} $\boldsymbol{\gamma\approx\infty}$.}

Now, we prove the AB estimate for $\gamma$ very large. We recall that 
\begin{equation*} 
    \frac{d}{dt}p_i=\gamma n_i^{\gamma -1}\left(\frac{n_{i+1/2}-n_i}{\Delta x}q_{i+1/2}+\frac{n_i-n_{i-1/2}}{\Delta x}q_{i-1/2}\right)+\gamma p_i \left(\delta_x^2p_i +G_i\right),
\end{equation*}
and we use the following definitions
 \begin{equation*}
     w_i=\delta_x^2 p_i = \frac{q_{i+1/2}-q_{i-1/2}}{\Delta x} + G(p_i),
\qquad \qquad
     q_{i+\frac 1 2}= \frac{n_{i+1}^\gamma - n_i^\gamma}{\Delta x}.
\end{equation*}
Computing the time derivative of $q_{i+1/2}$ we find 
\begin{equation*}
     \frac 1 \gamma \frac{d}{dt}q_{i+1/2} = \frac{1}{|\Delta x|^2} \left[ n_{i+1}^{\gamma-1} \left(n_{i+3/2} q_{i+3/2}-n_{i+1/2}q_{i+1/2}\right) - n_i^{\gamma-1}\left(n_{i+1/2}q_{i+1/2} - n_{i-1/2}q_{i-1/2}\right)\right].
\end{equation*}
Hence,
\begin{equation}\label{eq: dtwi}
  \begin{split}   \frac 1 \gamma \frac{d}{dt}w_i = &\frac{1}{|\Delta x|^3} \left[ n_{i+1}^{\gamma-1} \left(n_{i+3/2} q_{i+3/2}-n_{i+1/2}q_{i+1/2}\right) - n_i^{\gamma-1}\left(n_{i+1/2}q_{i+1/2} - n_{i-1/2}q_{i-1/2}\right) \right]   \\
     &+\frac{1}{|\Delta x|^3}\left[-n_{i}^{\gamma-1} \left(n_{i+1/2} q_{i+1/2}-n_{i-1/2}q_{i-1/2}\right) + n_{i-1}^{\gamma-1}\left(n_{i-1/2}q_{i-1/2} - n_{i-3/2}q_{i-3/2}\right)\right]\\
     &-\frac \alpha \gamma \prt*{\gamma n_i^{\gamma -1}\left(\frac{n_{i+1/2}-n_i}{\Delta x}q_{i+1/2}+\frac{n_i-n_{i-1/2}}{\Delta x}q_{i-1/2}\right)+\gamma p_i \left(\delta_x^2p_i +G_i\right)}.
     \end{split}
 \end{equation} 
Once again we define $\min_i w_i =: w_j$.
Let us notice that
\begin{equation*}
    \lim_{\gamma\to\infty}n_{j+1}^{\gamma-1} n_{j+2} = p_{j+1},
\end{equation*}
since
\begin{equation*}
    n_{j+1}^{\gamma-1} n_{j+2} = (p_{j+1})^{\frac{\gamma-1}{\gamma}} (p_{j+2})^{\frac 1 \gamma}.
\end{equation*}
Analogously, we also have
\begin{equation*}
    \lim_{\gamma\to\infty}n_{j}^{\gamma-1} n_{j+1/2} = p_{j} \qquad \lim_{\gamma\to\infty}n_{j}^{\gamma-1} n_{j-1/2} = p_{j}.
\end{equation*}
Thus, from Eq.~\eqref{eq: dtwi}, we recover
\begin{align*}
    \frac 1\gamma  \frac{d}{dt}w_j = & \lim_{\gamma\to\infty}\left\{\frac{1}{(\Delta x)^3} \left[ n_{j+1}^{\gamma-1} \left(n_{j+3/2} q_{j+3/2}-n_{j+1/2}q_{j+1/2}\right) - n_j^{\gamma-1}\left(n_{j+1/2}q_{j+1/2} - n_{j-1/2}q_{j-1/2}\right) \right] \right. \\[0.3em]
     & \:\: \left.+\frac{1}{(\Delta x)^3}\left[-n_{j}^{\gamma-1} \left(n_{j+1/2} q_{j+1/2}-n_{j-1/2}q_{j-1/2}\right) + n_{j-1}^{\gamma-1}\left(n_{j-1/2}q_{j-1/2} - n_{j-3/2}q_{j-3/2}\right)\right]\right.\\[0.3em]
 &\:\: \left. - \alpha p_j w_j -  \frac{\alpha}{\Delta x}  n_i^{\gamma -1}\prt*{(n_{i+1/2}-n_i) q_{i+1/2}+ (n_i-n_{i-1/2}) q_{i-1/2}} \right\} \\[0.3em]
     = &  \ \frac{1}{(\Delta x)^3} \left[ p_{j+1} \left( q_{j+3/2}-q_{j+1/2}\right) - p_j\left(q_{j+1/2} - q_{j-1/2}\right) \right]  \\[0.3em]
   & \:\: +\frac{1}{(\Delta x)^3}\left[-p_{j} \left( q_{j+1/2}-q_{j-1/2}\right) + p_{j-1}\left(q_{j-1/2} - q_{j-3/2}\right)\right] - \alpha p_j w_j\\[0.3em]
     \geq& \ \dfrac{p_{j+1} w_{j+1} - 2 p_j w_j + p_{j-1} w_{j-1}}{(\Delta x)^2} 
     \\[0.3em]
     \ge & \ w_j^2,
\end{align*} 
where we assumed again $w_j\leq 0.$
Hence
\begin{equation*}
      \frac{d}{dt}w_j \ge \gamma w_j^2,
\end{equation*}
thus the result is proven.
 
\end{proof}
\section{The fully discrete implicit scheme}\label{sec: implicit}
Now we consider the fully discrete implicit scheme and show that all the properties for the semi-discrete scheme hold for the fully discrete scheme if the time step $\Delta t$ is small enough. 

Similar to Section~\ref{sec: semi-discrete scheme}, we only consider the one dimensional problem and the scheme for the multidimensional problem is straightforward. 
In space, we use the same notations as in Section~\ref{sec: semi-discrete scheme}.  
We denote $N_i^k$ to be the numerical approximation of $n(t_k, x_i)$, where $t_k = k\Delta t$ and $x_i = i\Delta x$, $k\ge 0$, $i\in I$.
Then 
$P_i^k := \left(N_i^k\right)^\gamma$
is the numerical approximation of $p(t_k, x_i)$ and the fully implicit scheme can be written as
\begin{equation} \label{scheme_implicit}
	\delta_t N_i^k = \frac{N_{i+\frac12}^{k+1} Q_{i+\frac12}^{k+1} - N_{i-\frac12}^{k+1} Q_{i-\frac12}^{k+1}}{\Delta x} + N_i^{k+1} G_i^{k+1}, 
\end{equation}
where 
\begin{equation*}
	\delta_t N_i^k = \frac{N_i^{k+1}-N_i^k}{\Delta t}, \quad
	Q_{i+\frac12}^k = \frac{P_{i+1}^k - P_i^k}{\Delta x}, \quad 
	G_i^k = G(P_i^k) \le G(0), 
\end{equation*}
and 
\begin{equation*}
    N_{i+1/2}^k =\begin{cases}
    N_i^k, \quad &\text{ if } Q_{i+1/2}^k<0,\\
    N_{i+1}^k, \quad &\text{ if } Q_{i+1/2}^k>0.
    \end{cases}
\end{equation*}
For simplicity, we introduce 
\begin{equation}
\label{notation:A}
A(U,V) = VQ_+(U,V) - UQ_-(U,V), \quad \text{ for } U, V \ge0,
\end{equation}
where $Q(U,V) = (V^\gamma - U^\gamma)/\Delta x$ and 
\begin{equation*}
Q_+(U,V) = \max\{Q(U,V), 0\}, \quad Q_-(U,V) = \max\{-Q(U,V), 0\}.
\end{equation*}
A direct computation shows that
\begin{align*}
	&\partial_1 A(U,V)= -\gamma H(U,V) U^{\gamma-1} - Q_-(U,V) \le 0,\nonumber\\
	&\partial_2 A(U,V)  = \gamma H(U,V)V^{\gamma-1} + Q_+(U,V) \ge 0,
\end{align*}
where 
\begin{equation*}
	H(U,V) = \begin{cases}
    U, \quad &\text{ if } Q(U,V)<0,\\
    V, \quad &\text{ if } Q(U,V)>0.
    \end{cases}
\end{equation*}

With the notations defined above, Scheme~\eqref{scheme_implicit} can be reformulated as 
\begin{equation}\label{scheme_implicit_2}
	(1-\Delta t G_i^{k+1})N_i^{k+1} - \nu\left(A_{i+\frac12}^{k+1} - A_{i-\frac12}^{k+1}\right) = N_i^k,
\end{equation}
where $\nu = \Delta t/\Delta x$ and 
\begin{equation*}
	A_{i+\frac12}^{k+1} = A(N_i^{k+1}, N_{i+1}^{k+1}) = N_{i+\frac12}^{k+1} Q_{i+\frac12}^{k+1}.
\end{equation*}
\begin{thm}[A priori estimates] Let $T>0$ and $n_H:=p_H^{1/\gamma}$, $\Delta t < 1/G(0)$ and $k(T) = \lfloor T/\Delta t\rfloor$. Then, there exists a unique solution $N_i^k$ satisfying
\begin{itemize}
    \item[(i)]  $0\leq N_i^k\leq n_H, \; 0\leq P_i^k\leq p_H$, $\forall t >0, \forall i,$ and $\forall n$,
    \item[(ii)] $\Delta x \sum_i N_i^k \leq C(T),  \Delta x\sum_i P_i^k \leq C(T), $
    \item[(iii)] let $M_i^k$ be a non-negative solution satisfying Eq.~\eqref{scheme_implicit_2}, then $\Delta x \sum_i |M_i^k - N_i^k| \leq C(T)$,
    \item[(iv)]   $\Delta x  \sum_i |N_{i+1}^k - N_i^k| \leq C(T)$, 
    \item[(v)] $\Delta x \sum_i |\delta_t N_i^k| \leq C(T),$ $\Delta x \sum_i |\delta_t P_i^k|  \leq C(T),$
    \item[(vi)] $\Delta t \Delta x \sum_{j=0}^k \sum_i |Q_{i+\frac 1 2}^j|^2  \leq C(T)$,
\end{itemize}
for some positive constant $C(T)$ depending on $T$ and independent of $\gamma$.
\end{thm}

\begin{proof} 

\textbf{Solvability and $L^\infty$ estimate.}
When $\Delta t < 1/G(0)$ and $0\le N_i^k \le p_H^{\frac{1}{\gamma}}$ for all $i$, we claim that  there exists a unique solution $N_i^{k+1}$ satisfying $0\le N_i^{k+1} \le p_H^{\frac{1}{\gamma}}$.

The proof relies on the the existence of sub- and supersolutions. 
When $\bar{N}_i = p_H^{\frac{1}{\gamma}}$ for all $i$, we have $G(\bar{N}_i^\gamma)<0$ and $A(\bar{N}_i,\bar{N}_{i+1}) = 0$, which implies that 
\begin{equation*}
	(1-\Delta t G(\bar{N}_i^\gamma))\bar{N}_i - \nu \left(A(\bar{N}_i,\bar{N}_{i+1}) - A(\bar{N}_{i-1},\bar{N}_{i})\right) \ge N_i^k
\end{equation*} 
and thus 
$
\bar{N}_i = p_H^{\frac{1}{\gamma}}
$
 is a supersolution.
Similarly, we can prove that
$
\bar{N}_i = 0
$
 is a subsolution.
Then following the proof in \cite{AlmeidaBubbaPerthamePouchol}, we can prove the existence and uniqueness of the solution. The detailed proof can be found in Appendix \ref{appendix_existence}. 
  \\
\textbf{$\boldsymbol{L^1}$ estimate.}
Summing up Eq.~\eqref{scheme_implicit} over $i$, we have
\begin{equation*}
    \Delta x\sum_i N_i^{k+1} - \Delta x\sum_i N_i^k  = \Delta t \Delta x\sum_i N_i^{k+1} G_i^{k+1} \le G(0)  \Delta t\Delta x\sum_i N_i^{k+1} .
\end{equation*}
As a result, when $\Delta t \le \alpha/G(0)$ with $\alpha<1$, we have 
 \begin{equation*}
    \Delta x\sum_i N_i^k  \le \frac{1}{(1-\Delta t G(0))^k} \Delta x\sum_i N_i^{0} \le \frac{1}{(1-\alpha)^{\frac{G(0)T}{\alpha}}}\Delta x \sum_i N_i^{0} ,
\end{equation*}
where $T = k\Delta t$.  
Further, we have $\sum_i P_i^k \le p_H^{\gamma-1}\sum_i N_i^k \le C(T)$.
\\
\textbf{$\boldsymbol{L^1}$-contraction.} 
Denote $M_i^k$ to be another non-negative solution satisfying Eq.~\eqref{scheme_implicit_2}, or more specifically
\begin{equation*}
	(1-\Delta t G_{M,i}^{k+1})M_i^{k+1} - \nu\left(A_{M,i+\frac12}^{k+1} - A_{M,i-\frac12}^{k+1}\right) = M_i^k,
\end{equation*}
where 
\begin{equation*}
G_{M,i}^k = G(P_{M,i}^k) \text{ with } P_{M,i}^k = (M_i^k)^\gamma, \,A_{M,i+\frac12}^{k+1} =A_{i+\frac12}(M_i^{k+1}, M_{i+1}^{k+1}).
\end{equation*}
Subtracting the equation for $N_i^k$ from the equation for $M_i^k$, we get 
\begin{equation*}
	I_1 - \nu\left(A_{M,i+\frac12}^{k+1} - A_{N,i+\frac12}^{k+1}\right) + \nu\left(A_{M,i-\frac12}^{k+1} - A_{N,i-\frac12}^{k+1}\right) = M_i^k - N_i^k,
\end{equation*}
where the term $I_1$ is defined as 
\begin{align*}
	I_1 &= \left[(1-\Delta t G_{M,i}^{k+1})M_i^{k+1} - (1-\Delta t G_{N,i}^{k+1})N_i^{k+1}\right]\\ 
	&=(1-\Delta t G_{M,i}^{k+1})(M_i^{k+1}-N_i^{k+1}) - \Delta t(G_{M,i}^{k+1}-G_{N,i}^{k+1})N_i^{k+1}\\
	&=(1-\Delta t G_{M,i}^{k+1})(M_i^{k+1}-N_i^{k+1}) - \Delta t G'(P_{\eta,i}^{k+1})N_i^{k+1} (P_{M,i}^{k+1}-P_{N,i}^{k+1})
\end{align*}
where $P_{\eta,i}^{k+1} = (\eta_i^{k+1})^{\gamma}$ with $\eta_i^{k+1}$ being some non-negative number between $M_i^{k+1}$ and $N_i^{k+1}$. 
Noticing that $G'(\cdot)\le0$ and the fact that $P_{M,i}^{k+1}-P_{N,i}^{k+1}$ shares the same sign with $M_i^{k+1}-N_i^{k+1}$, we have that
\begin{align*}
	I_1\sign(M_i^{k+1}-N_i^{k+1}) &\ge (1-\Delta t G(0)) |M_i^{k+1}-N_i^{k+1}|+\Delta t \min_p |G'(p)| N_i^{k+1} \left|P_{M,i}^{k+1}-P_{N,i}^{k+1}\right|.
\end{align*}
In fact, we can further prove that
\begin{equation} \label{sign:I_1}
\begin{split}
	I_1\sign(M_i^{k+1}-N_i^{k+1}) &\ge (1-\Delta t G(0)) |M_i^{k+1}-N_i^{k+1}|\\[0.2em]
	&\qquad +\Delta t \min_p |G'(p)| \max\{M_i^{k+1}, N_i^{k+1}\} \left|P_{M,i}^{k+1}-P_{N,i}^{k+1}\right|\\[0.2em]
	& \ge(1-\Delta t G(0)) |M_i^{k+1}-N_i^{k+1}|.
	\end{split}
\end{equation}
By the mean value theorem, we have 
\begin{align*}
	A_{M,i+\frac12}^{k+1} - A_{N,i+\frac12}^{k+1} 
	& =   \alpha_i^{k+1} \left(M_i^{k+1}-N_i^{k+1}\right) 
		+ \beta_{i+1}^{k+1} \left(M_{i+1}^{k+1}-N_{i+1}^{k+1}\right) 
\end{align*}
where $\alpha_i^{k+1}\le0$ and $\beta_i^{k+1}\ge0$ are defined as 
\begin{align*}
	&\alpha_i^{k+1} :=\partial_1 A(\xi_i^{k+1}, M_{i+1}^{k+1})= \frac{A(M_i^{k+1}, M_{i+1}^{k+1}) - A(N_i^{k+1}, M_{i+1}^{k+1})}{M_i^{k+1}-N_i^{k+1}},\\
	&\beta_i^{k+1}:=\partial_2 A(N_i^{k+1}, \eta_{i+1}^{k+1})= \frac{A(N_i^{k+1}, M_{i+1}^{k+1}) - A(N_i^{k+1}, N_{i+1}^{k+1})}{M_{i+1}^{k+1}-N_{i+1}^{k+1}},
\end{align*}
for some $\xi_i^{k+1}$, $\eta_i^{k+1}$ between $M_{i}^{k+1}$ and $N_{i}^{k+1}$. 
As a result, 
\begin{equation}\label{sign:A_right}
	\left(A_{M,i+\frac12}^{k+1} - A_{N,i+\frac12}^{k+1}\right)\sign(M_i^{k+1}-N_i^{k+1})\le \alpha_i^{k+1} \left|M_i^{k+1}-N_i^{k+1}\right|
		+ \beta_{i+1}^{k+1} \left|M_{i+1}^{k+1}-N_{i+1}^{k+1}\right|.
\end{equation}
Similarly, we can prove that 
\begin{equation}\label{sign:A_left}
	\left(A_{M,i-\frac12}^{k+1} - A_{N,i-\frac12}^{k+1}\right)\sign(M_i^{k+1}-N_i^{k+1})\ge \alpha_{i-1}^{k+1} \left|M_{i-1}^{k+1}-N_{i-1}^{k+1}\right|
		+ \beta_i^{k+1} \left|M_i^{k+1}-N_i^{k+1}\right|.
\end{equation}
Combining Eqs.~(\ref{sign:I_1}, \ref{sign:A_right}, \ref{sign:A_left}), we finally get 
\begin{align*} 
	&\left(1-\Delta t G(0)-\nu\alpha_i^{k+1}+\nu\beta_i^{k+1}\right)\left|M_i^{k+1}-N_i^{k+1}\right| -\nu\beta_{i+1}^{k+1}\left|M_{i+1}^{k+1}-N_{i+1}^{k+1}\right| \\
	&+\nu \alpha_{i-1}^{k+1}\left|M_{i-1}^{k+1}-N_{i-1}^{k+1}\right|
	\le \left|M_i^k-N_i^k\right|. 
\end{align*}
Summing over $i$, we have
\begin{equation*}
	\left(1-\Delta t G(0)\right)\sum_i |M_i^{k+1}-N_i^{k+1}| \le \sum_i |M_i^k-N_i^k|.
\end{equation*}
which indicates that, when  $\Delta t < 1/G(0)$,
\begin{equation}\label{L1-contraction}
	 \Delta x\sum_i \left|M_i^k-N_i^k\right| \le \frac{1}{(1-\Delta t G(0))^k} \Delta x\sum_i \left|M_i^{0} -N_i^{0}\right|  \le C(T),
\end{equation}
since we assumed that $ \Delta x\sum_i \left|M_i^{0} -N_i^{0}\right| \le C$.
\\
\textbf{$\boldsymbol{BV}$ estimate.} 
When $\Delta t < 1/G(0)$, by taking $M_i^k=N_{i+1}^k$ in Eq.~\eqref{L1-contraction}, we get that,
\begin{equation*}
  \Delta x	\sum_i |N_{i+1}^k-N_i^k| \le \frac{1}{(1-\Delta t G(0))^k} \Delta x \sum_i |N_{i+1}^{0} -N_i^{0}|  \le C(T).
\end{equation*}
\\
\textbf{Estimate on time derivative.} 
The boundedness of the discrete time derivative of the density comes directly from the $L^1$-contraction \eqref{L1-contraction}.
Assuming $\Delta t < 1/G(0)$ and taking $M_i^k=N_i^{k+1}$ in Eq.~\eqref{L1-contraction}, we have that 
\begin{equation}\label{dtN_bound}
	\Delta x \sum_i |\delta_t N_i^k|  \le \frac{1}{(1-\Delta t G(0))^k} \Delta x \sum_i |\delta_t N_i^{0}| \le C(T).
\end{equation}
Analogous to the semi-discrete case, we can prove an estimate of the discrete time derivative of the pressure.
Denoting $k(T) = \lfloor T/\Delta t\rfloor$, where $\lfloor x\rfloor$ is the largest integer that is less or equal than $x$, 
then we are able to prove that
\begin{equation}\label{proof:dtP}
	\Delta t \Delta x\sum_{n=1}^{k(T)} \sum_i \left|\delta_t P_{N,i}^k\right|  \leq C(T).
\end{equation}
The proof is similar to the semi-discrete case. 
To begin with, we have that 
\begin{equation*}
	|\delta_t P_{N,i}^k| = |\delta_t P_{N,i}^k| \mathds{1}_{\{\max\{N_i^k,N_i^{k+1}\} \le\frac12\}}
		+ |\delta_t P_{N,i}^{k+1}| \mathds{1}_{\{\max\{N_i^k,N_i^{k+1}\}>\frac12\}}.
\end{equation*}
The first term is uniformly bounded in $\gamma$ thanks to Eq.~\eqref{dtN_bound} and 
\begin{equation}\label{proof:dtP_1}
	|\delta_t P_{N,i}^k| \le \gamma \max\{(N_i^k)^{\gamma-1},(N_i^{k+1})^{\gamma-1}\} |\delta_t N_i^k|\le \frac{\gamma}{2^{\gamma-1}}|\delta_t N_i^k|. 
\end{equation} 
To give an estimate of the second term, we recall the first inequality in Eq.~\eqref{sign:I_1}, \ie
\begin{align*}
I_1\sign(M_i^{k+1}-N_i^{k+1}) \ge &(1-\Delta t G(0)) |M_i^{k+1}-N_i^{k+1}|\\
&+\Delta t \min_p |G'(p)| \max\{M_i^{k+1}, N_i^{k+1}\} \left|P_{M,i}^{k+1}-P_{N,i}^{k+1}\right|.
\end{align*}
And then following a similar procedure as before, we have that 
\begin{align*}
 \left(1-\Delta t G(0)-\nu\alpha_i^{k+1}+\nu\beta_i^{k+1}\right)&\left|M_i^{k+1}-N_i^{k+1}\right|
\\[0.3em]
&+ \Delta t \min_p |G'(p)| \max\{M_i^{k+1}, N_i^{k+1}\}\left|P_{M,i}^{k+1}-P_{N,i}^{k+1}\right|
\\[0.3em]
&-\nu\beta_{i+1}^{k+1}\left|M_{i+1}^{k+1}-N_{i+1}^{k+1}\right|+\nu \alpha_{i-1}^{k+1}\left|M_{i-1}^{k+1}-N_{i-1}^{k+1}\right|\\[0.3em]
&	\le \left|M_i^k-N_i^k\right|.
\end{align*}
Now taking $M_i^k=N_i^{k+1}$, dividing both sides by $\Delta t$ and summing over $i$ and $n=0,1,\dots$, we proved that 
\begin{align*}
	&\min_p |G'(p)|  \Delta t \Delta x \sum_{n=1}^{k(T)}\sum_i \max\{N_i^k,N_i^{k+1}\} \left|\delta_t P_{N,i}^k\right|
	\\
	\le &\Delta x  \sum_i\left|\delta_t N_i^{0}\right|  -\Delta x \sum_i\left|\delta_t N_i^{k(T)}\right|
	+ G(0) \Delta t  \Delta x \sum_{n=1}^{k(T)}\sum_i |\delta_t N_i^k|\le C(T),
\end{align*}
which further implies that
\begin{equation}\label{proof:dtP_2}
\begin{split}
  \Delta t  \Delta x \sum_{n=1}^{k(T)}\sum_i |\delta_t P_{N,i}^k| &\mathds{1}_{\{\max\{N_i^k,N_i^{k+1}\}>\frac12\}} 
 \\
 &\le 2  \Delta t  \Delta x \sum_{n=1}^{k(T)}\sum_i \max\{N_i^k,N_i^{k+1}\} |\delta_t P_{N,i}^k| \\
 &\le C(T).
\end{split}
\end{equation}
The conclusion \eqref{proof:dtP} is then obvious by combining Eq.~\eqref{proof:dtP_1} and Eq.~\eqref{proof:dtP_2}.
\\
\textbf{$\boldsymbol{L^2}$ estimate on the pressure gradient.}
Rewriting Eq.~\eqref{scheme_implicit} to be 
\begin{equation*}
\delta_t N_i^k = 
\frac{N_{i+\frac12}^{k+1}  - N_{i}^{k+1}}{\Delta x} Q_{i+\frac12}^{k+1} + \frac{N_{i}^{k+1}  - N_{i-\frac12}^{k+1}}{\Delta x} Q_{i-\frac12}^{k+1} + N_i^{k+1} (\delta_x^2 P_i^{k+1} + G_i^{k+1})
\end{equation*}
and multiplying both sides by $\gamma(N_i^{k+1})^{\gamma-1}$, we get 
\begin{align*}
	\gamma(N_i^{k+1})^{\gamma-1}\delta_t N_i^k 
	\le \left|Q_{i+\frac12}^{k+1} \right|_+^2 + \left|Q_{i-\frac12}^{k+1} \right|_-^2+ \gamma P_i^{k+1} (\delta_x^2 P_i^{k+1} + G_i^{k+1}).
\end{align*}
Noticing that $\delta_t P_i^k \le \gamma(N_i^{k+1})^{\gamma-1}\delta_t N_i^{k}$ due to the convexity, 
we proved 
\begin{equation}\label{ineq:dtP}
	\delta_t P_i^k \le \left|Q_{i+\frac12}^{k+1} \right|_+^2 + \left|Q_{i-\frac12}^{k+1} \right|_-^2+ \gamma P_i^{k+1} (\delta_x^2 P_i^{k+1} + G_i^{k+1}).
\end{equation}
Summing Eq.~\eqref{ineq:dtP} over all $i$, we have 
\begin{align*}
	\delta_t \sum_i P_i^k &\le \sum_i\left|Q_{i+\frac12}^{k+1} \right|_+^2 + \sum_i\left|Q_{i-\frac12}^{k+1} \right|_-^2+ \sum_i\gamma P_i^{k+1} (\delta_x^2 P_i^{k+1} + G_i^{k+1})\\
	& = (1-\gamma)\sum_i |Q_{i+\frac12}^{k+1}|^2  + \gamma \sum_i P_i^{k+1}G_i^{k+1}\\
	& \le (1-\gamma)\sum_i |Q_{i+\frac12}^{k+1}|^2  + \gamma G(0)\sum_i P_i^{k+1}.
\end{align*}
Then summing over $n=0,1,2,\dots$ and dividing both sides by $\gamma-1$, we get 
\begin{equation*}
	\Delta t \Delta x\sum_{j=0}^k\sum_i |Q_{i+\frac12}^{j}|^2  \le \dfrac{\Delta x\sum_i P_i^0 - \Delta x\sum_i P_i^k}{\gamma - 1} 
		+ \frac{\gamma}{\gamma-1}G(0) \Delta t \Delta x\sum_{j=0}^k\sum_i P_i^{j} \le C(T).
\end{equation*}

\end{proof}

\section{Numerical simulations}\label{sec: numerical simulations}
Now we present some numerical results on Eq.~\eqref{eq: n} and for some extensions of the model including the effect of a nutrient. In particular, we are interested in the performance of the implicit scheme~\eqref{scheme_implicit} for large values of $\gamma$, hence confirming the AP property of the scheme.

\subsection{Accuracy test: the Barenblatt solution}
At first, we consider the simplest example in order to test the accuracy of the scheme as $\gamma$ increases.
Let us take the standard porous medium equation in dimension 1, \ie \ Eq.~\eqref{eq: n} with trivial reaction terms
\begin{equation*}
    \partialt n = \frac{\partial^2 n^{\gamma+1}}{\partial x^2}, 
\end{equation*}
where for sake of simplicity we take $p=\frac{\gamma+1}{\gamma}n^\gamma$.
We take as initial data the delayed Barenblatt solution
\begin{equation}\label{eq: barenblatt}
    n(x,0) = \frac{1}{t_0^\beta} \left| C-\beta\frac{\gamma}{2(\gamma+1)}\frac{x^2}{t_0^{2\beta}}\right|_+^{\frac 1 \gamma},
\end{equation}
with $t_0=0.01$, $\beta=1/(\gamma+2)$ and $C$ a positive constant to be chosen later.

We compare the numerical solution of the scheme to the Barenblatt profile for $\gamma=3$, $\gamma=12$, $\gamma=40$. We compute the $L^1$-error for $\Delta x = 1/2^k$, with $k=4,5,6,7,8$ and $\Delta t= 0.01 \Delta x$. 

We choose $[-5,5]$ to be the spatial computational domain and $T=0.1$ as final time. Upon defining $N_x=10/\Delta x$, $N_t=0.1/ \Delta t$, the error is given by
\begin{equation}\label{eq: error}
    err_1 = \sum_{i=1}^{N_x}\sum_{j=1}^{N_t} |N_i^k - n(i\Delta x, j \Delta t)| \Delta x \Delta t.
\end{equation}
In the formula of the exact solution, Eq.~\eqref{eq: barenblatt}, we choose $C=1$ for $\gamma=3$ and $C=0.1$ for $\gamma=12,40$.
In Figure~\ref{fig:numericalvsbarenblatt}, the plots of both the analytical solution and the numerical solution are displayed. We notice that as $\gamma$ increases, the moving boundary becomes sharper and sharper and this affects the accuracy of the scheme as can be seen in Figure~\ref{fig:error}, where on the left we plot the error \eqref{eq: error} with respect to $\Delta x$, and for different values of $\gamma$, and on the right we display the error along the time variable, namely
\begin{equation*}
    err_1(j \Delta t) =\sum_{i=1}^{N_x} |N_i^j - n(i\Delta x, j \Delta t)| \Delta x .
\end{equation*}
\begin{figure}[hbt!]
	\centering
	\begin{subfigure}[b]{0.48\textwidth}
		\includegraphics[width=\textwidth]{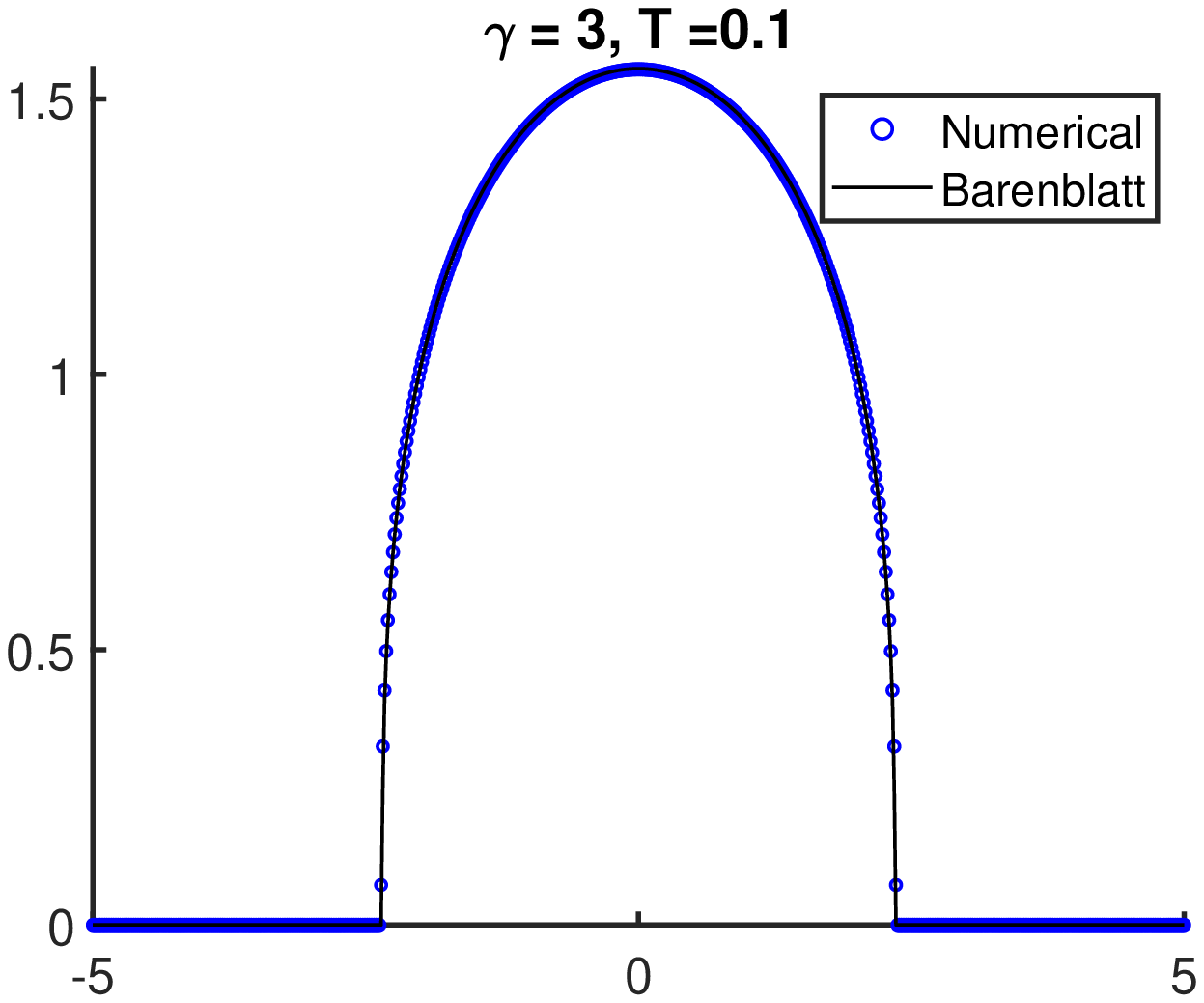}
	\end{subfigure}
	\begin{subfigure}[b]{0.48\textwidth}
		\includegraphics[width=\textwidth]{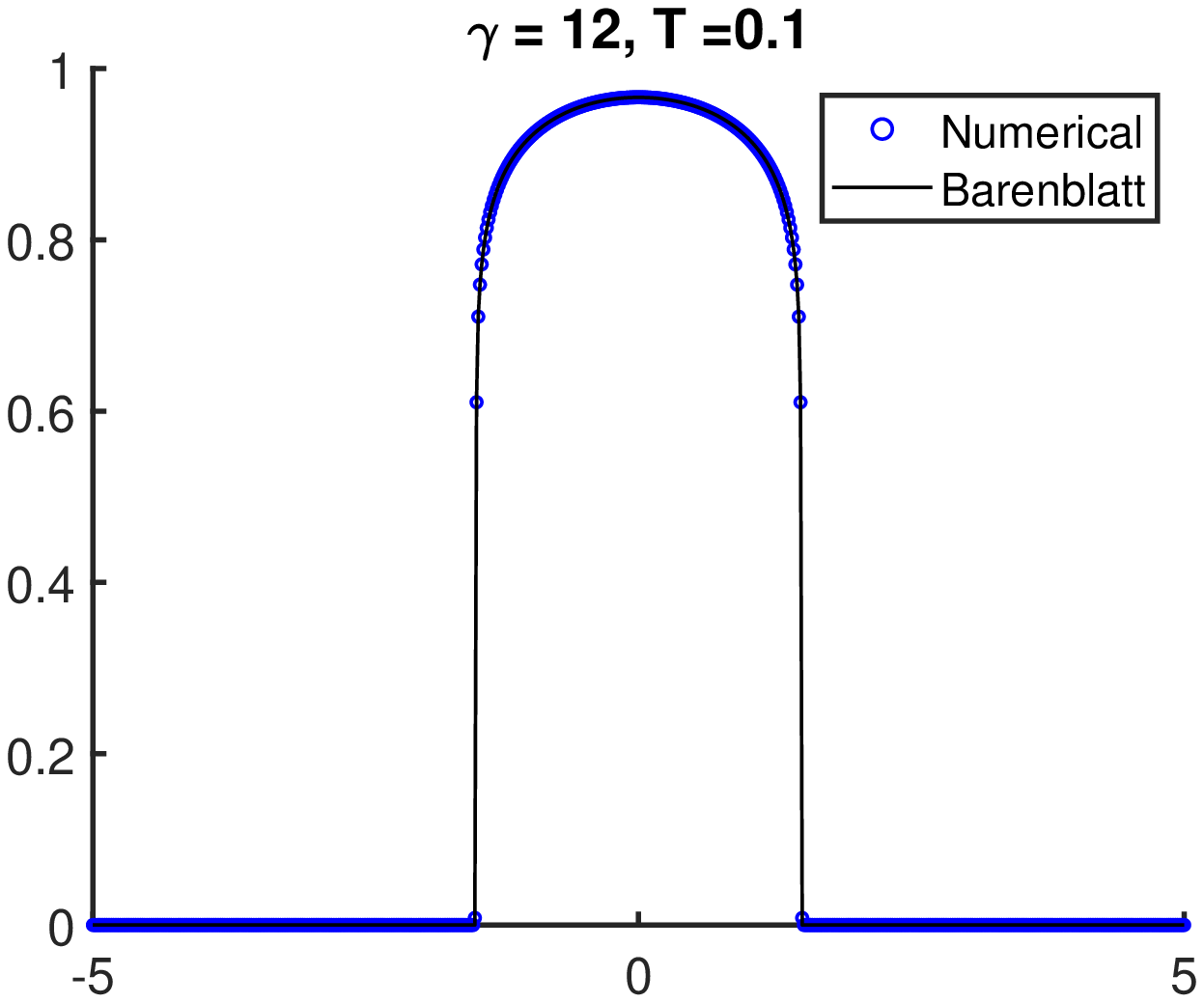}
	\end{subfigure}
	\caption{Porous Medium Equation in 1D: we compare the analytical solution and the numerical solution for $\gamma=3$ (left) and $\gamma=12$ (right), with $\Delta x=1/64$ and  $\Delta t = 0.01 \Delta x$. }\label{fig:numericalvsbarenblatt}
\end{figure}
\begin{figure}[hbt!]
	\centering
	\begin{subfigure}[b]{0.48\textwidth}
		\includegraphics[width=\textwidth]{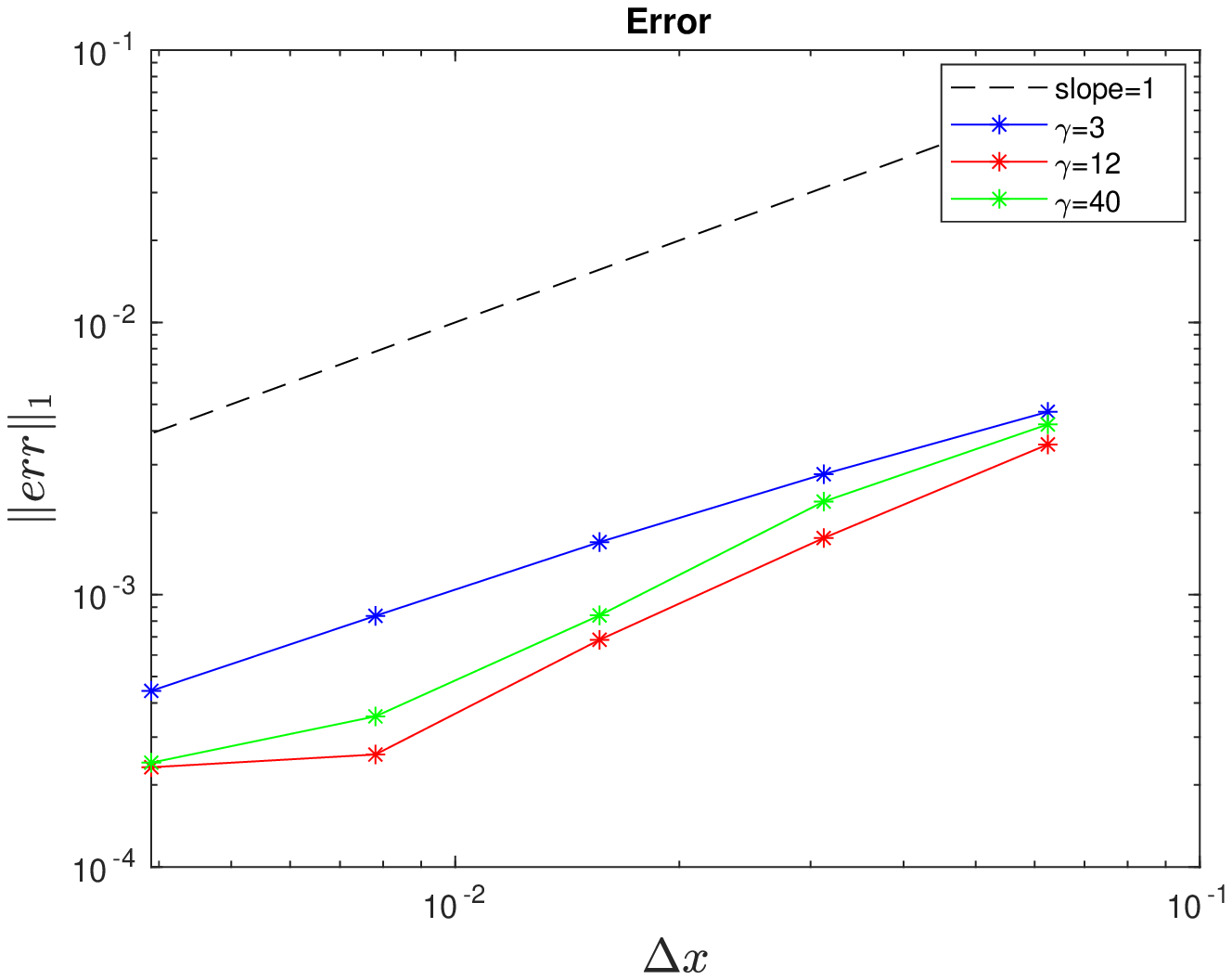}
	\end{subfigure}
	\begin{subfigure}[b]{0.48\textwidth}
		\includegraphics[width=\textwidth]{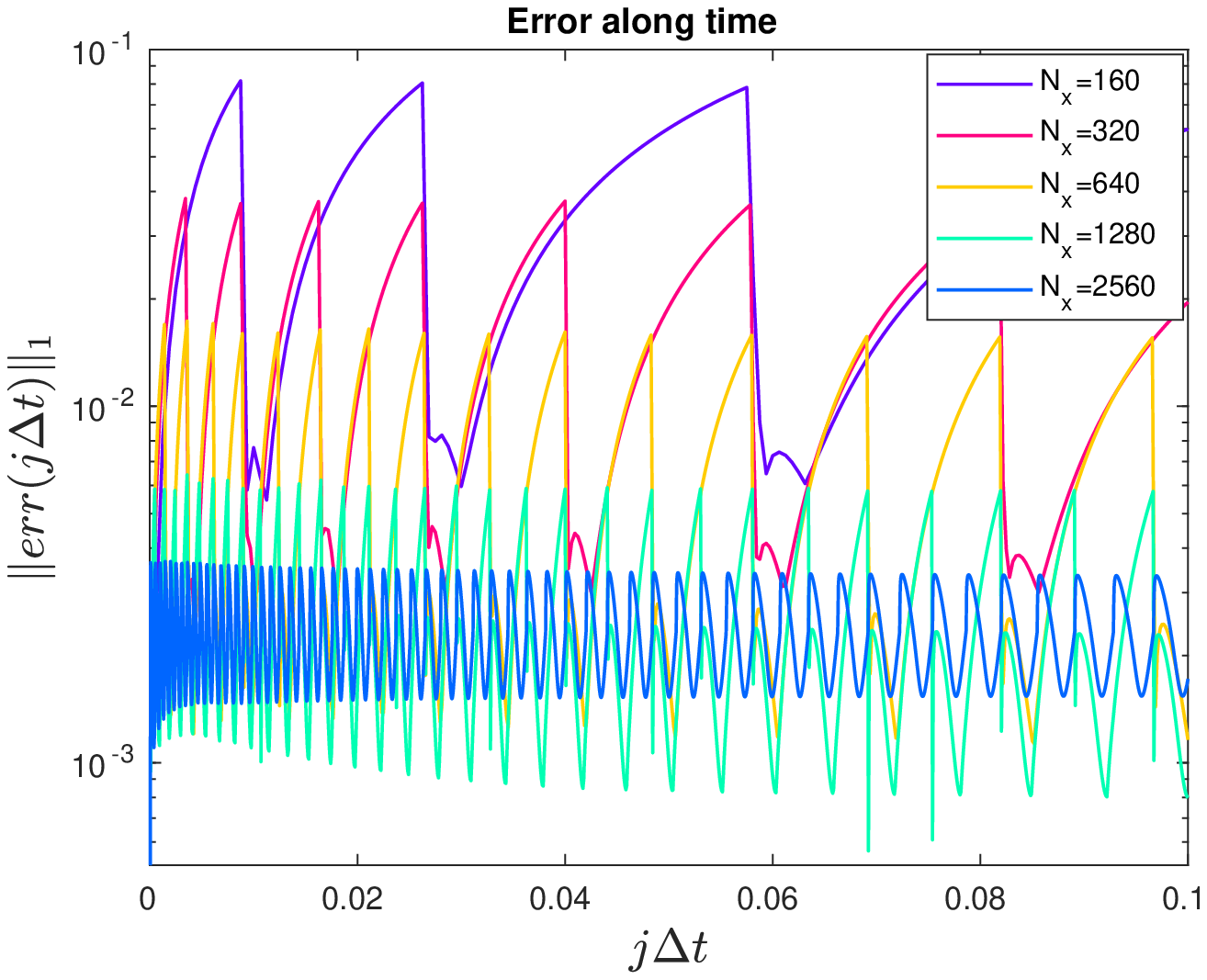}
	\end{subfigure}
	\caption{Porous Medium Equation in 1D: Left: plot of the error with respect to $\Delta x$ for different values of $\gamma$. Right: plot of the error along time for $\gamma=12$}\label{fig:error}.
\end{figure}
The oscillations of the error along time confirm the effect of the free boundary on the accuracy. 

\subsection{1D model with nutrient: \textit{in vitro} and \textit{in vivo}}
Including the effect of a nutrient (\textit{e.g.} oxygen) into the model, the density equation~\eqref{eq: n} is coupled with an equation for the nutrient concentration $c(x,t)$, to obtain the system
\begin{equation*}
    \begin{dcases}
       \partialt n - \div(n \nabla p) = n G(p,c),\\
       \tau\partialt c - \Delta c + H(n,c) = 0,
       \end{dcases}
\end{equation*}
where $H$ denotes the nutrient consumption and $\tau$ is a time scaling parameter. Since the nutrient diffuses much faster than the tumor invasion, it is usual to take $\tau=0$.
The consumption term $H$ can take different forms, depending on which stage of tumor growth we put under investigation. 

For instance, if one considers an \textit{in vitro} setting, which means that the tumor is developing surrounded by an homogeneous liquid, then the level of nutrient is assumed to be constant outside the region occupied by the tumor, while inside it is consumed linearly, with a rate $\psi(n)$ depending on the tumor cell population density.
The model reads
\begin{equation}
    \label{eq: sysnutri_vitro}
    \tag{in vitro}
    \begin{dcases}
        - \Delta c + \psi(n) c = 0, &\text{ in } \{n> 0\}, \\[0.3em]
          c = c_B, &\text{ in } \R^d \setminus \{n> 0\},
          \end{dcases}
\end{equation}
The consumption rate $\psi(n)$ is always non-negative and vanishes for $n=0$.

A second kind of models are the \textit{in vivo} models, which include the effect of the blood vessels that deliver the nutrient supply. During the early stages of tumor growth, the vasculature is present only outside the tumor region (\textit{avascular phase}), and the equation reads
\begin{equation}
    \label{eq: nutri_vivo_avas}
    \tag{in vivo}
    - \Delta c + \psi(n) c = (c_B -c) \mathds{1}_{\{n=0\}}.
\end{equation}
On the other hand, if the tumor is already in its \textit{vascular phase}, we have
\begin{equation}
    \label{eq: nutri_vivo_vasc}
    \tag{in vivo: vascular}
    - \Delta c + \psi(n) c = (c_B -c) K(p),
\end{equation}
where $K$ is the nutrient release rate which depends on the pressure. In particular, we assume it to decrease with respect to the pressure to describe the shrinking effect of the mechanical stress generated by the cells on the vessels, which may cause the reduction of nutrients delivery, \textit{cf.} \cite{MMACCL}. 
We refer the reader to \cite{PTV} for an extensive study of the Hele-Shaw model in both the \textit{in vitro} and \textit{in vivo} cases.

From now on, we assume that the growth term $G$ depends only on the nutrient concentration, forgetting the effect of the pressure. Then, passing to the \textit{incompressible limit} $\gamma \rightarrow \infty$, we obtain the limit problem
\begin{equation*} 
    \begin{dcases}
       \partialt{n_\infty} - \div(n_{\infty} \nabla p_{\infty}) = n_{\infty} G(c_{\infty}),\\[0.3em]
        - \Delta c_{\infty} + H(n_{\infty},c_{\infty}) = 0,
       \end{dcases}
\end{equation*}
and since it holds $p_{\infty}(1- n_{\infty})=0, $
the density is constantly equal to $1$ in the set $\{p_{\infty}>0\}$.

As shown in \cite{DP}, one can also pass to the limit in the equation for the pressure, which leads to the Hele-Shaw problem
\begin{equation*}
    \begin{dcases}
-\Delta p_{\infty}=G(c_{\infty}), &\text{ in } \Omega(t),\\[0.3em]
 p_{\infty}=0, &\text{ on } \partial\Omega(t),
\end{dcases}
\end{equation*}
where $\Omega(t):=\{x \ | \ p_{\infty}(x,t)>0\}$. 
\subsubsection{In vitro model: comparison with the exact solution of the Hele-Shaw problem}
We consider the model~\eqref{eq: sysnutri_vitro} in 1D with linear growth, \ie \ $G(c)=c$, and $\psi(n)=n$, namely
\begin{equation}\label{eq: 1D vitro}
\begin{cases}
\partial_t n -\partial_x (n \partial_x p)= n c, &\\
-\partial_{xx} c + n c=0, \qquad &\text{ in }\{n> 0\},\\
c=c_B, \qquad &\text{ in } \R^d \setminus \{n> 0\}
.
\end{cases}
\end{equation}
We take as initial density $n(x,0)$ the characteristic function of the interval $[-R_0,R_0]$, with $R_0>0$. Then, passing to the incompressible limit, the density remains always a patch, with support $[-R(t),R(t)]$. Therefore, we have
\begin{equation}\label{eq: ninfty vitro}
    n_\infty = \mathds{1}_{[-R(t),R(t)]}.
\end{equation}
Thus, as computed in \cite{LTWZ19}, the explicit solution is
\begin{equation*}
c_\infty=\begin{cases}
\dfrac{c_B \cosh(x)}{\cosh(R(t))}, &\text{ for } x\in [-R(t),R(t)],\\[0.3em]
c_B , &\text{ for }  x\notin [-R(t),R(t)], 
\end{cases}
\end{equation*}
and
\begin{equation}\label{eq: pinfty vitro}
p_\infty=\begin{cases}
-\dfrac{c_B \cosh(x)}{\cosh(R(t))}+c_B, &\text{ for } x\in [-R(t),R(t)],\\[0.3em]
0 , &\text{ for }  x\notin [-R(t),R(t)].
\end{cases}
\end{equation}
The velocity of the front is
$$R'(t)=c_B \tanh(R(t)).$$
We perform numerical simulations using our scheme for system \eqref{eq: 1D vitro} for $\gamma=80$ and compare the results to the exact solution \eqref{eq: ninfty vitro}-\eqref{eq: pinfty vitro}. We use the computational domain $[-5,5]$ and choose as initial data 
\begin{equation}\label{eq: initial data simulations}
 n(x,0)= \left(p_\infty(x,0)\right)^{\frac{1}{\gamma}},
\end{equation}
with $p_\infty$ defined by \eqref{eq: pinfty vitro}.
We also set $c_B=1$, $R(0)=1$, $\Delta x=0.025$ and $\Delta t=10^{-6}$, \cf Fig.~\ref{fig: in VITRO 1D}.
\begin{figure}[hbt!]
	\centering
	\begin{subfigure}[b]{0.48\textwidth}
		\includegraphics[width=\textwidth]{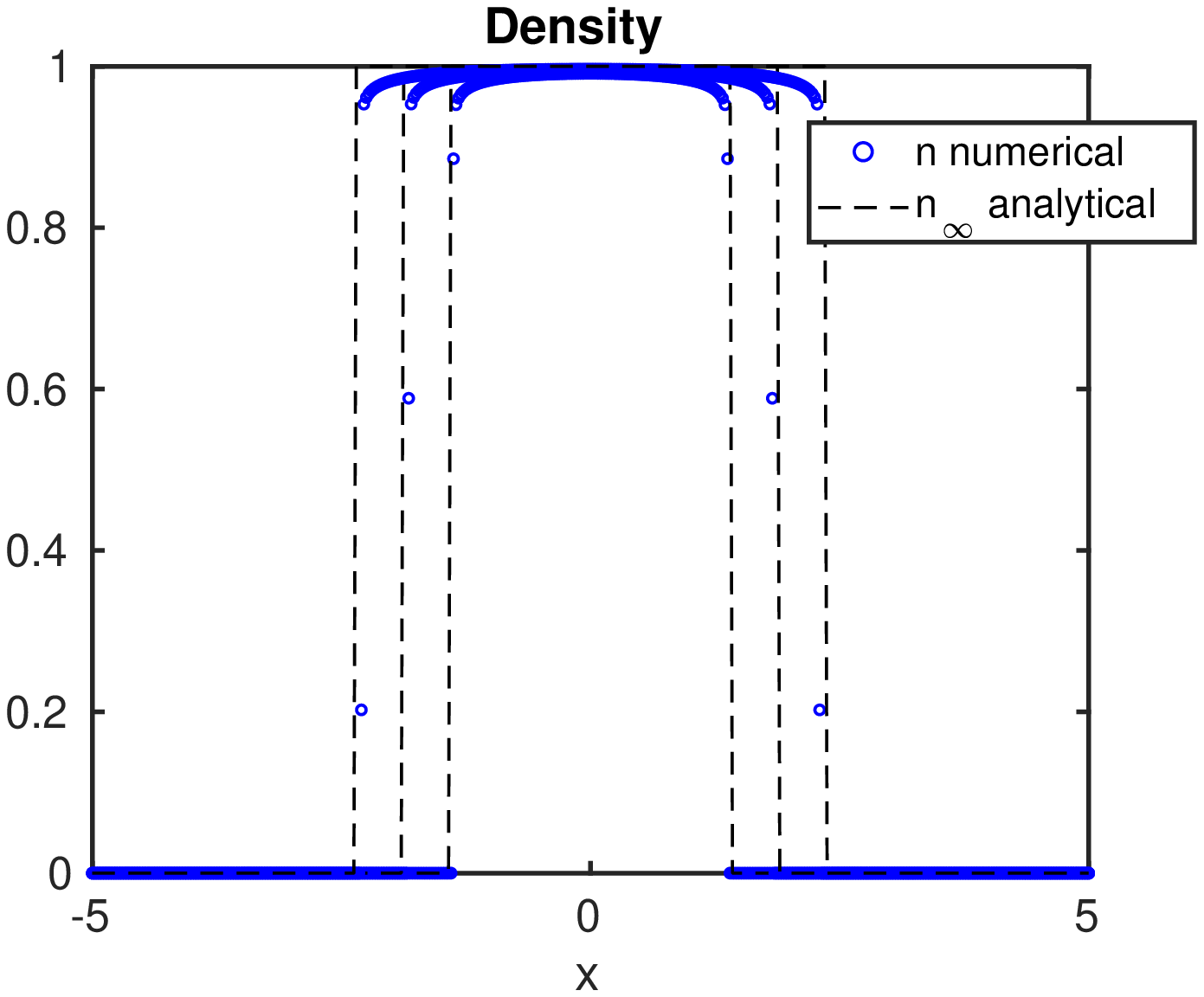}
	\end{subfigure}
	\begin{subfigure}[b]{0.48\textwidth}
		\includegraphics[width=\textwidth]{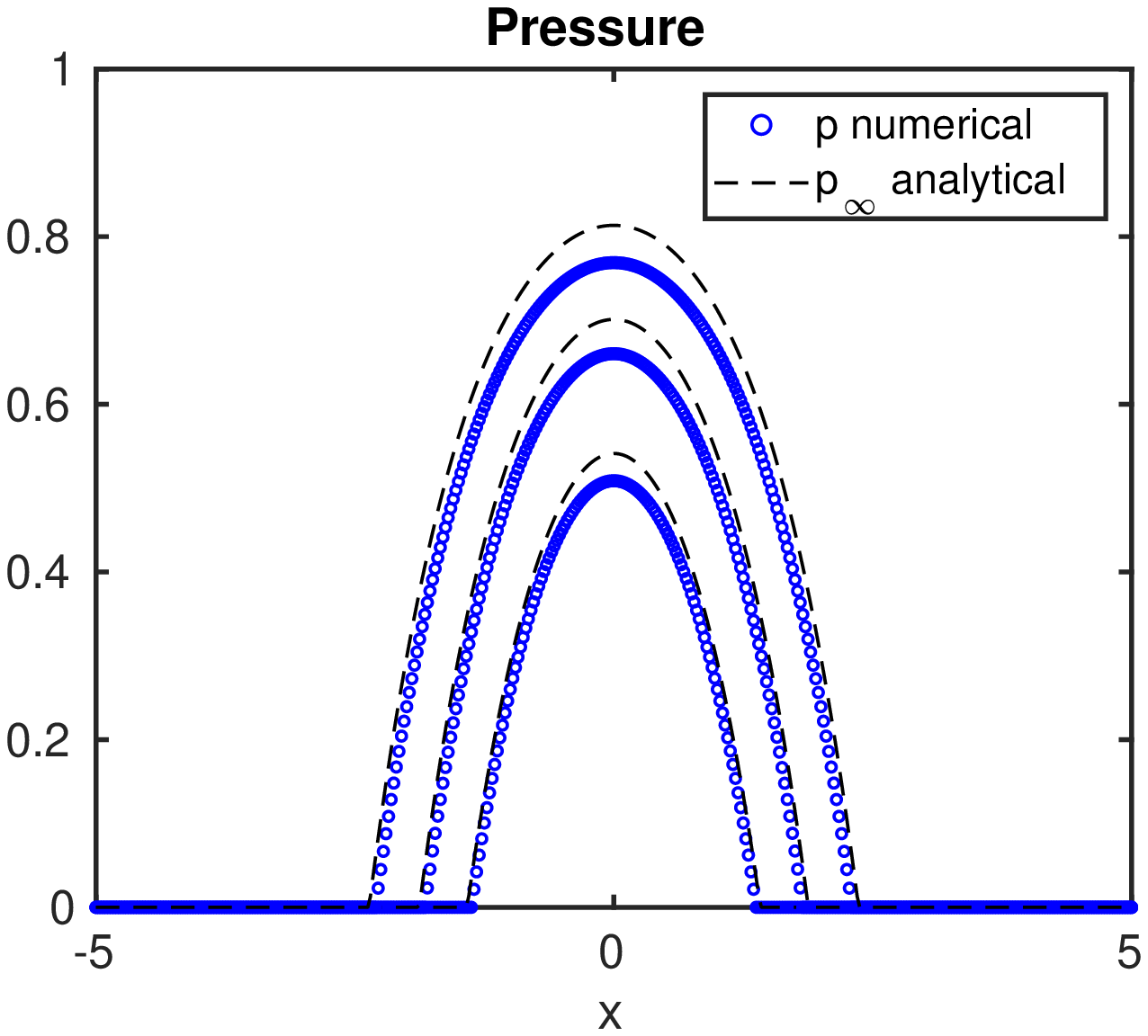}
	\end{subfigure}
	\caption{\textit{In vitro} model in 1D: comparison between the numerical solution and the analytical solution at different times, t=0.5, t=1, t=1.5, with $\gamma=80$, $\Delta x=0.025$ and $\Delta t=10^{-6}$.}
	\label{fig: in VITRO 1D}
\end{figure}
\subsubsection{In vivo model: comparison with the exact solution}
Using again a characteristic function as initial data, in the limit $\gamma \rightarrow \infty$ the model \eqref{eq: nutri_vivo_avas} reads
\begin{equation*}
-\partial_{xx}c_\infty +c_\infty = (c_B - c_\infty)\mathds{1}_{\{n=0\}},
\end{equation*}
with ${\{n=0\}}=\R \setminus [-R(t), R(t)]$.
Thus, the explicit solution is given by
\begin{equation*}
c_\infty=\begin{cases}
\dfrac{c_B}{e^{R(t)}}\cosh(R(t)), &\text{ for } x\in [-R(t),R(t)],\\[0.3em]
c_B - c_B \sinh(R(t))e^{-|x|} , &\text{ for }  x\notin [-R(t),R(t)], 
\end{cases}
\end{equation*}
\cf \cite{LTWZ19}. The limit pressure is 
\begin{equation}\label{eq: pinfty vivo}
p_\infty=\begin{cases}
-\dfrac{c_B G_0}{e^{R(t)}}\cosh(x)+ \dfrac{c_B G_0}{e^{R(t)}}\cosh(R(t)), &\text{ for } x\in [-R(t),R(t)],\\[0.3em]
0 , &\text{ for }  x\notin [-R(t),R(t)],
\end{cases}
\end{equation}
with a front invasion speed given by
$$R'(t)=c_B G_0 \frac{\sinh{R(t)}}{e^{R(t)}}.$$
As for the previous case, we perform numerical simulations using our scheme for the system \eqref{eq: nutri_vivo_avas} with $\gamma=80$ and compare the results to the exact solution. As before we choose \eqref{eq: initial data simulations} as initial data where the pressure is defined by \eqref{eq: pinfty vivo} and we set $c_B=1$, $R(0)=1$, $\Delta x=0.025$ and $\Delta t=10^{-6}$, \cf Fig.~\ref{fig: in vivo 1D}. As in \cite{LTWZ}, we notice that the scheme is more accurate for the \textit{in vivo} model than for the \textit{in vitro}.

\begin{figure}[hbt!]
	\centering
	\begin{subfigure}[b]{0.48\textwidth}
		\includegraphics[width=\textwidth]{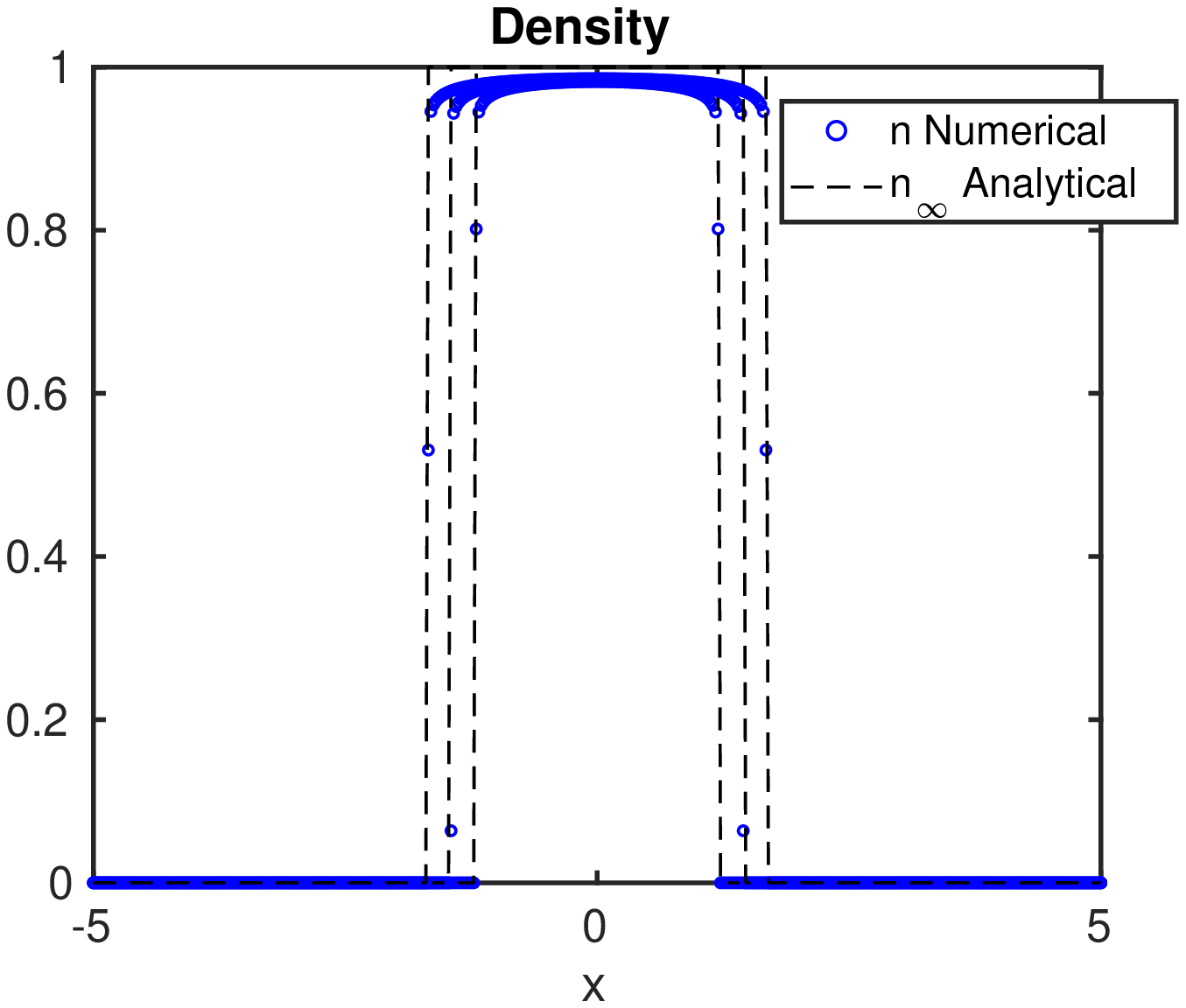}
	\end{subfigure}
	\begin{subfigure}[b]{0.48\textwidth}
		\includegraphics[width=\textwidth]{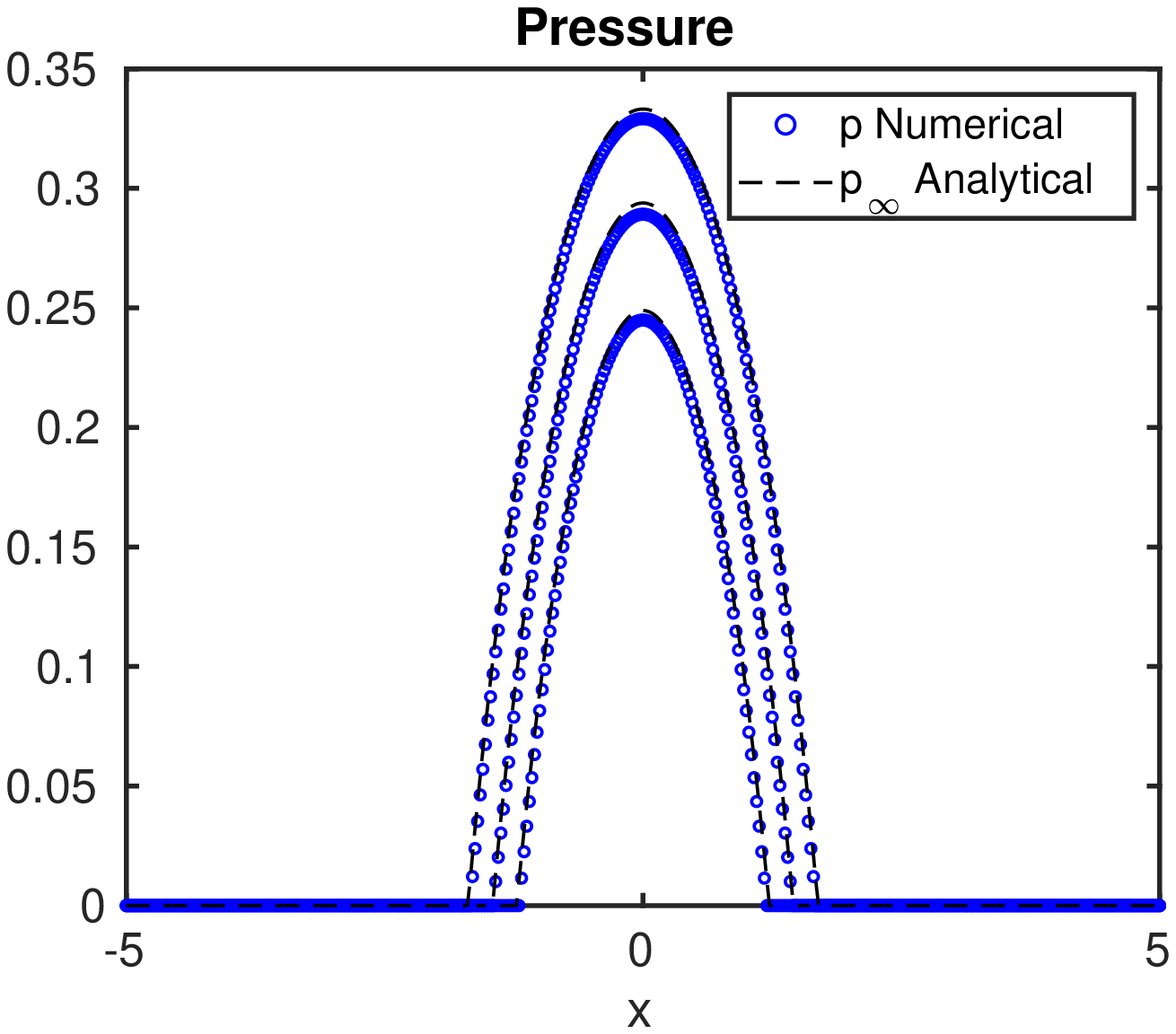}
	\end{subfigure}
	\caption{\textit{In vivo} model in 1D: comparison between the numerical solution and the analytical solution at different times, t=0.5, t=1, t=1.5, with $\gamma=80$, $\Delta x=0.025$ and $\Delta t=10^{-6}$.}\label{fig: in vitro 1D}
	\label{fig: in vivo 1D}
\end{figure}
\subsection{Two-species model: proliferating and necrotic cells}
We consider a model in including a second species of cells. Indeed, at the early stages of its growth, the tumor mass develop a necrotic core of dead cells, which is surrounded by a rim of quiescent or proliferating cells. The model reads
\begin{equation}\label{eq: necrotic}
\begin{dcases}
    \partialt{n_P} - \partialx{}\prt*{n_P \partialx p} = n_P G(c),\\[0.2em]
    \partialt{n_D} - \partialx{}\prt*{n_D \partialx p} = n_P |G(c)|_-,
    \end{dcases}
\end{equation}
where $n_P$ and $n_D$ represent the cell densities of proliferating and necrotic (dead) cells. The total population density and the pressure are, respectively, $n=n_P + n_D, \; p=n^\gamma$.

Since in this case the growth rate $G=G(c)$ can be negative, the proliferating cells die and turn into necrotic with the same rate. In particular, we assure there exist a positive constant $\bar{c}$ such that $G(c)<0$ if $c<\bar{c}$, to indicate that the cells die because of the lack of nutrients.

We use the scheme \eqref{scheme_implicit} for both the equations on $n_P$ and $n_D$ and we test it for both \eqref{eq: sysnutri_vitro} and \eqref{eq: nutri_vivo_avas}. 
We take as computational domain $[-6,6]$, and we set $c_B=1$,
\begin{equation*}
    G(c)=\begin{cases}
        12 &\text{ if } c<0.4,\\[0.2em]
        -15 &\text{ if } c\geq 0.4,
    \end{cases}
\end{equation*}
and as initial data 
\begin{equation*}
    n_P^0= \mathds{1}_{[-1,1]}, \qquad n^0_D=0.
\end{equation*}
The numerical simulations for the \textit{in vitro} and \textit{in vivo} environments are displayed along time in Fig.~\ref{fig: in vitro necrotic 1D} and Fig.~\ref{fig: in vivo necrotic 1D}, respectively.
\begin{figure}[hbt!]
	\centering
	\begin{subfigure}[b]{0.3\textwidth}
		\includegraphics[width=\textwidth]{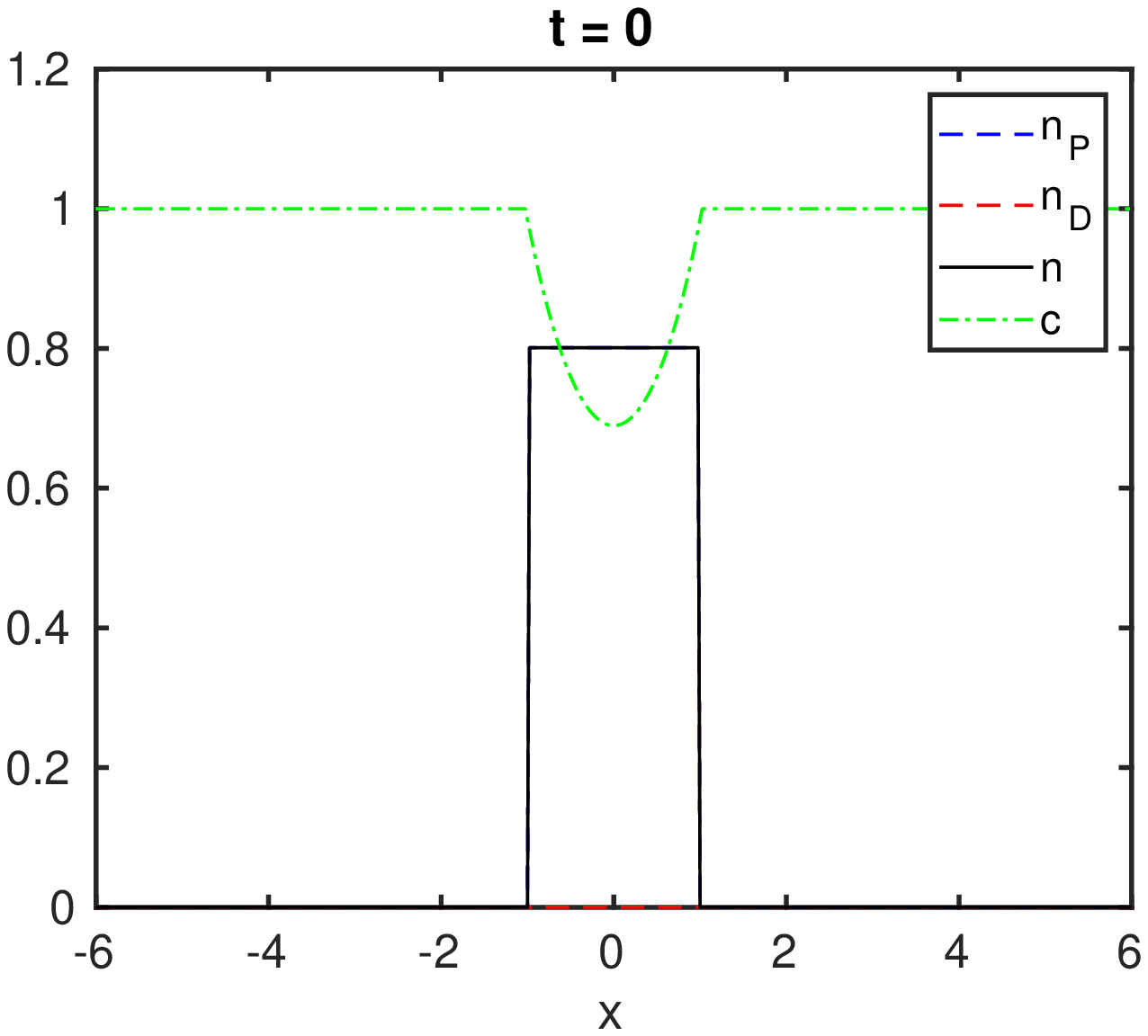}
	\end{subfigure}
	\begin{subfigure}[b]{0.3\textwidth}
		\includegraphics[width=\textwidth]{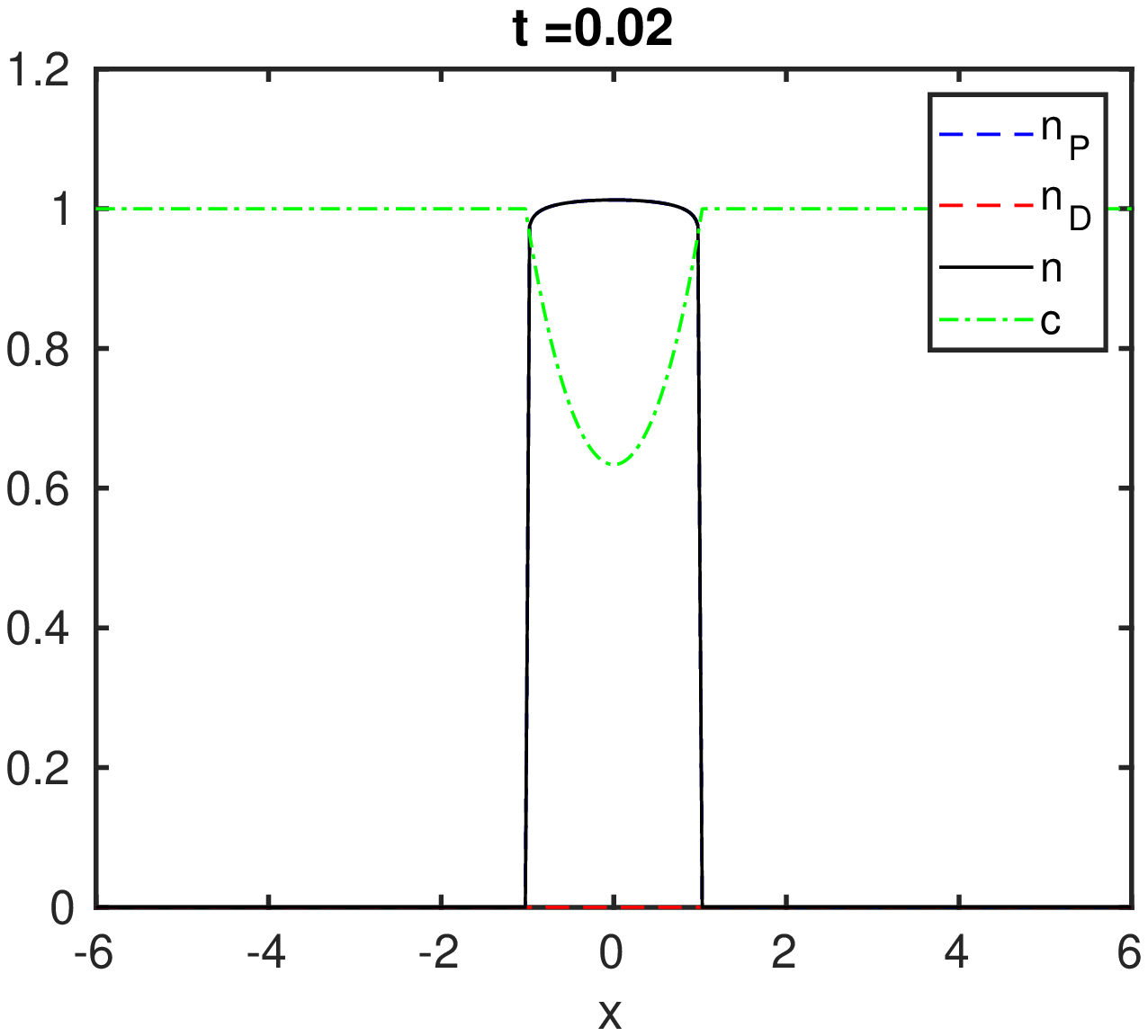}
	\end{subfigure}
	\begin{subfigure}[b]{0.3\textwidth}
		\includegraphics[width=\textwidth]{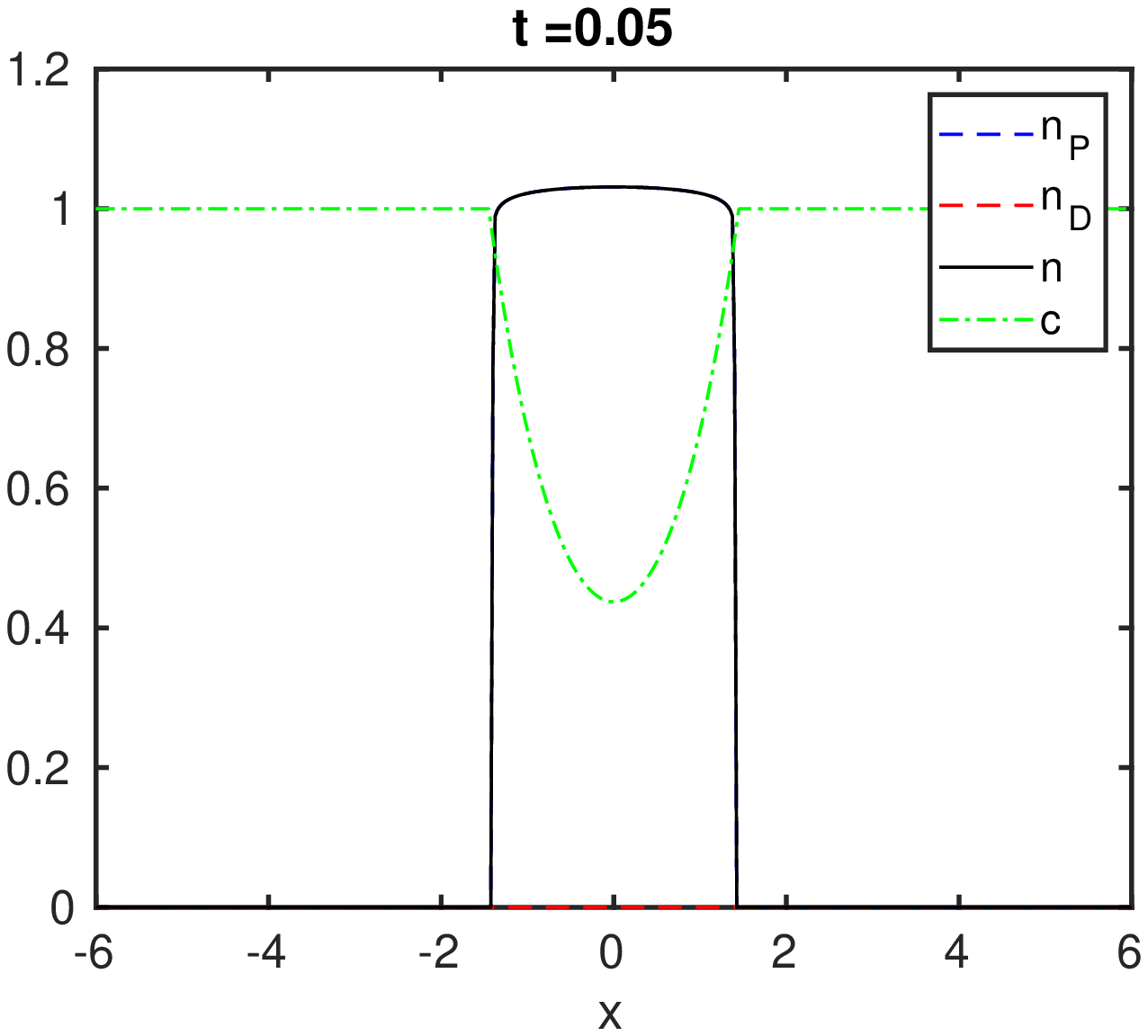}
	\end{subfigure}\\
	\begin{subfigure}[b]{0.3\textwidth}
		\includegraphics[width=\textwidth]{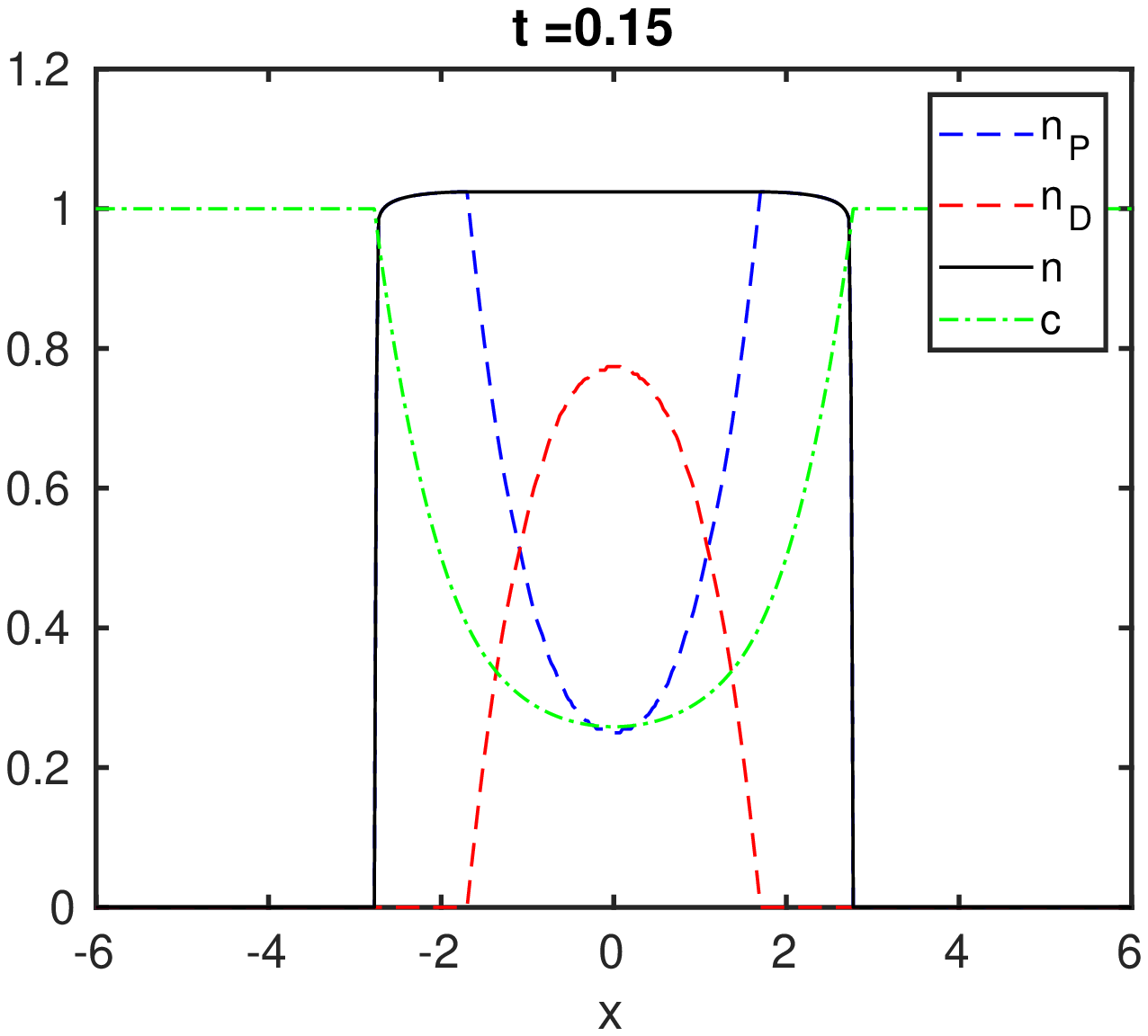}
	\end{subfigure}
	\begin{subfigure}[b]{0.3\textwidth}
		\includegraphics[width=\textwidth]{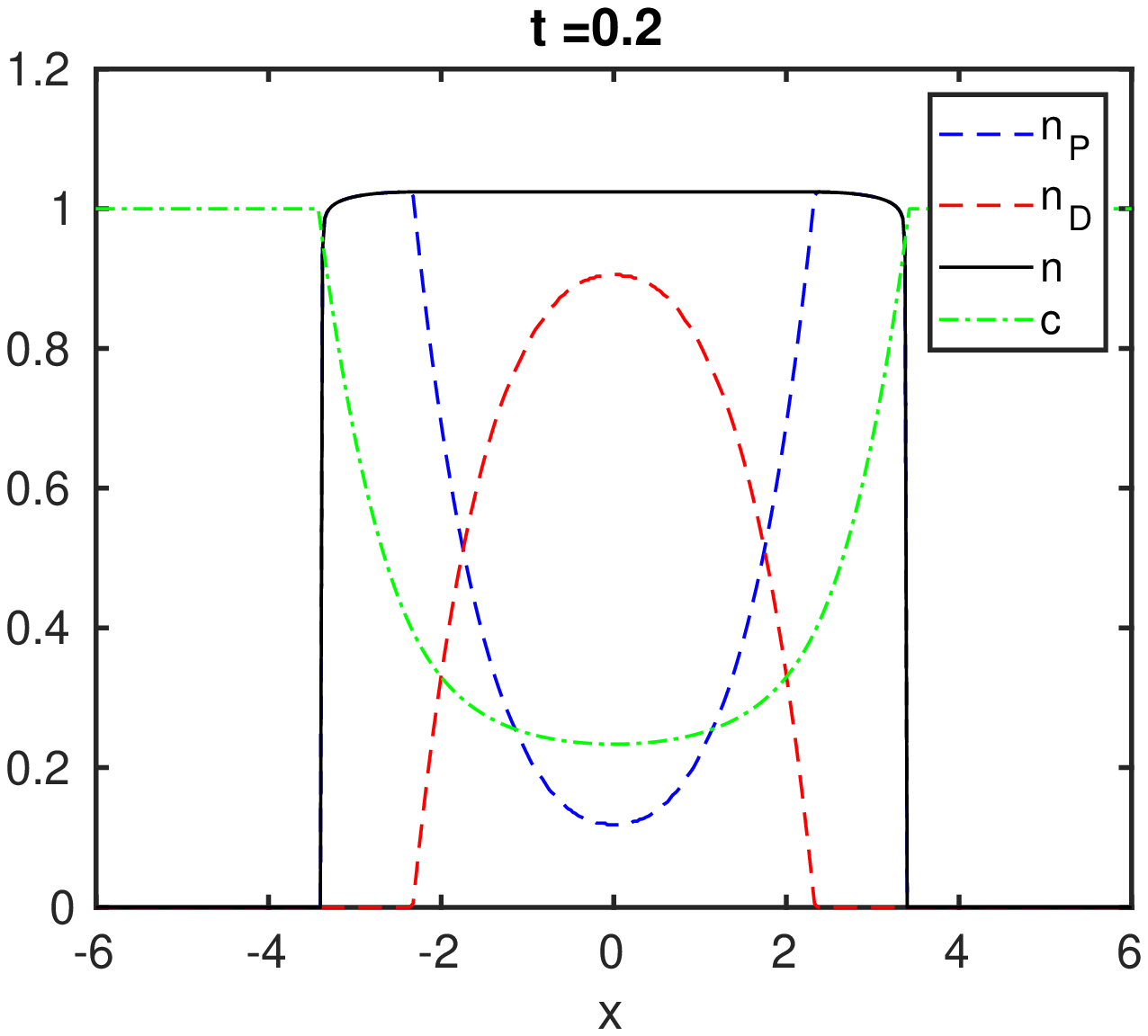}
	\end{subfigure}
	\begin{subfigure}[b]{0.3\textwidth}
		\includegraphics[width=\textwidth]{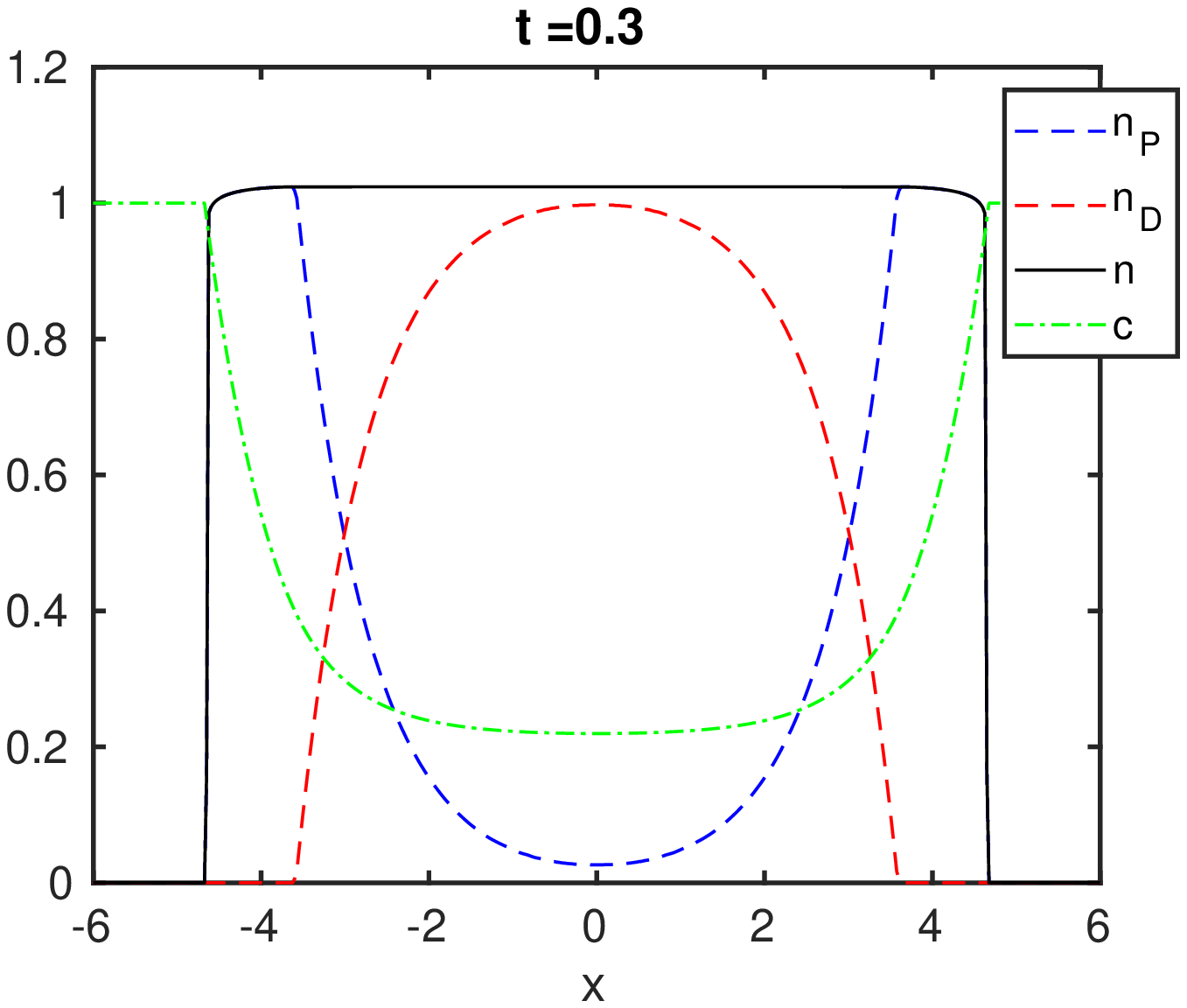}
	\end{subfigure}
	\caption{\textit{In vitro} two-species model in 1D: plot of $n_P, n_D, n, c$ with $\gamma=80$, $\Delta x=0.025$ and $\Delta t=10^{-4}$.} 
	\label{fig: in vitro necrotic 1D}
\end{figure}
\begin{figure}[hbt!]
	\centering
	\begin{subfigure}[b]{0.3\textwidth}
		\includegraphics[width=\textwidth]{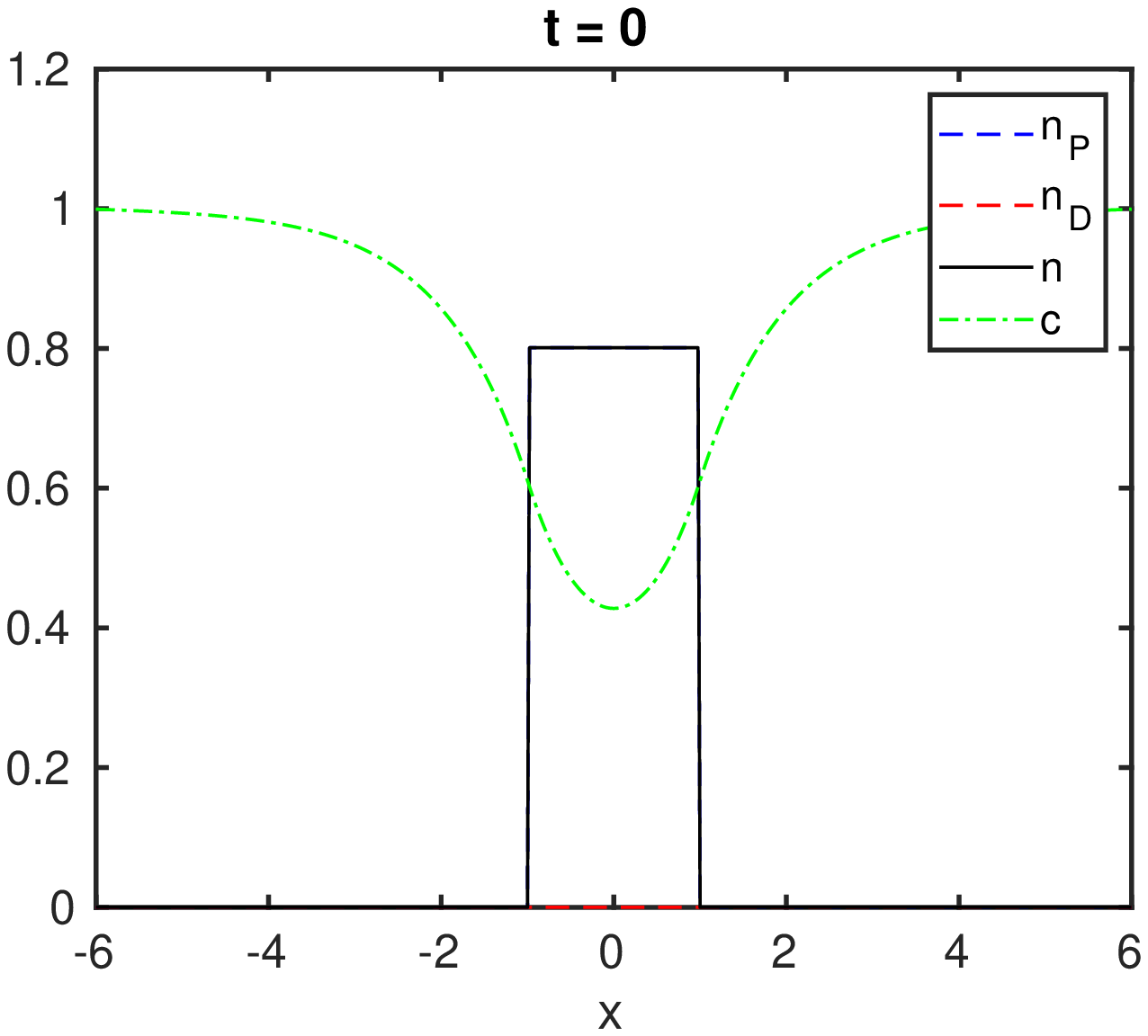}
	\end{subfigure}
	\begin{subfigure}[b]{0.3\textwidth}
		\includegraphics[width=\textwidth]{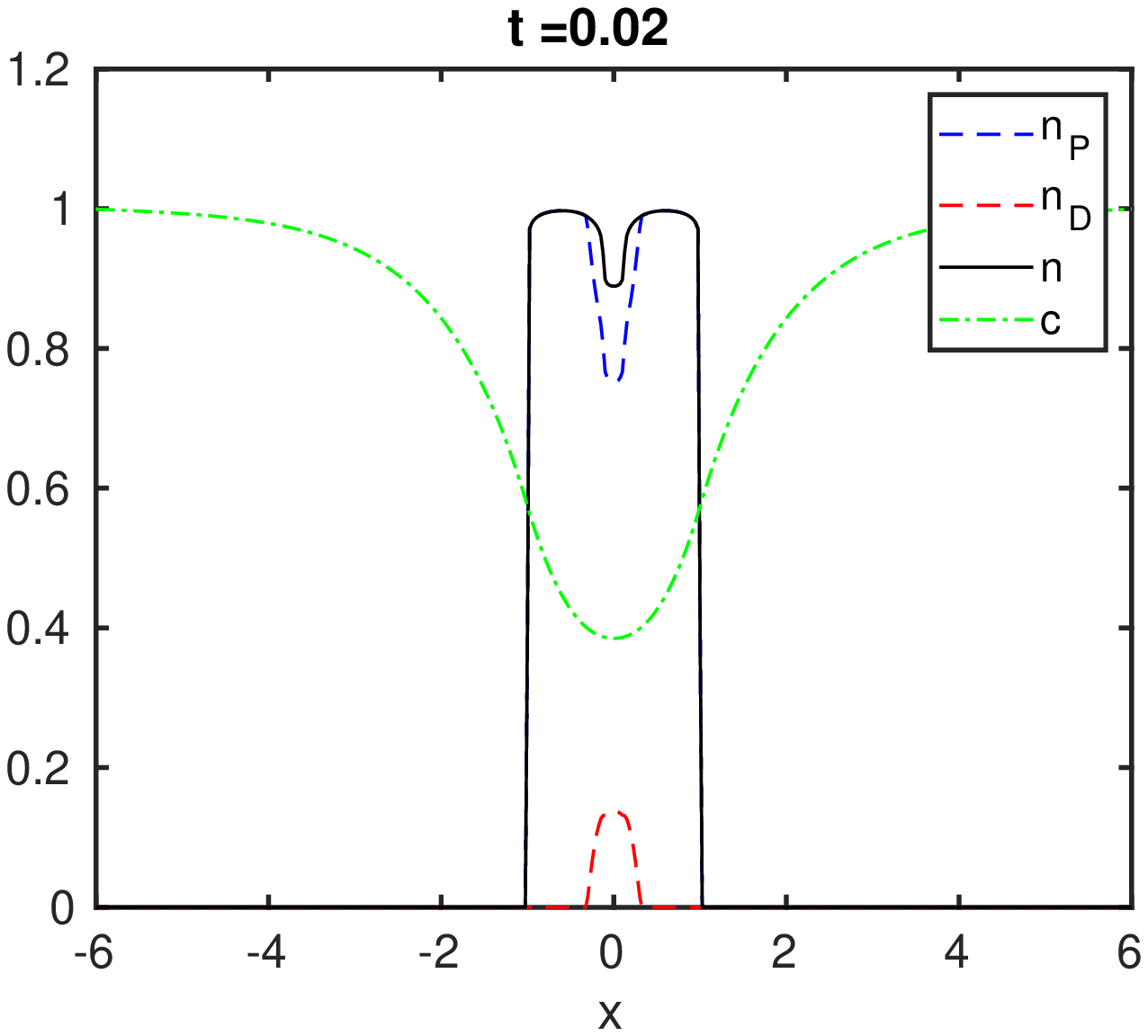}
	\end{subfigure}
	\begin{subfigure}[b]{0.3\textwidth}
		\includegraphics[width=\textwidth]{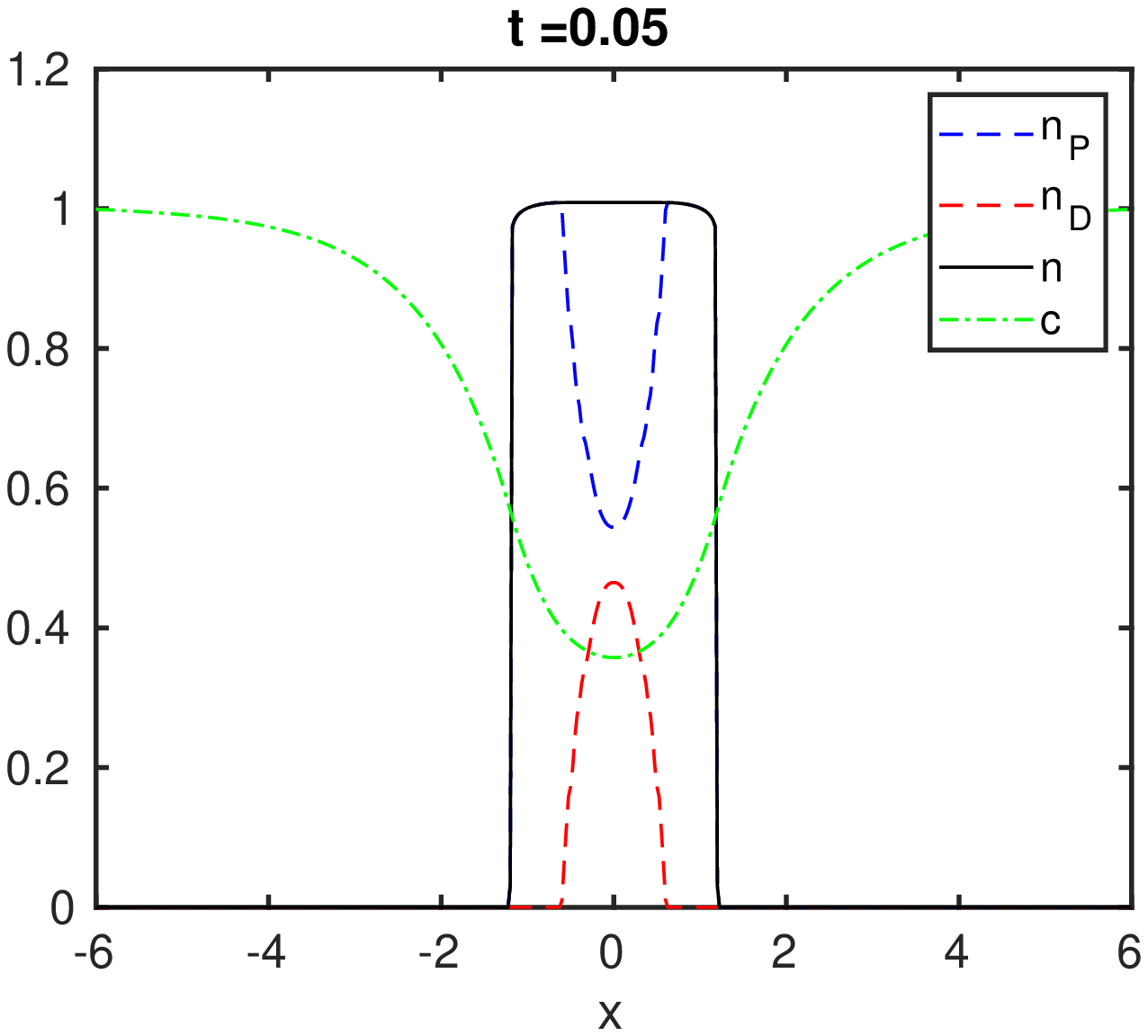}
	\end{subfigure}\\
	\begin{subfigure}[b]{0.3\textwidth}
		\includegraphics[width=\textwidth]{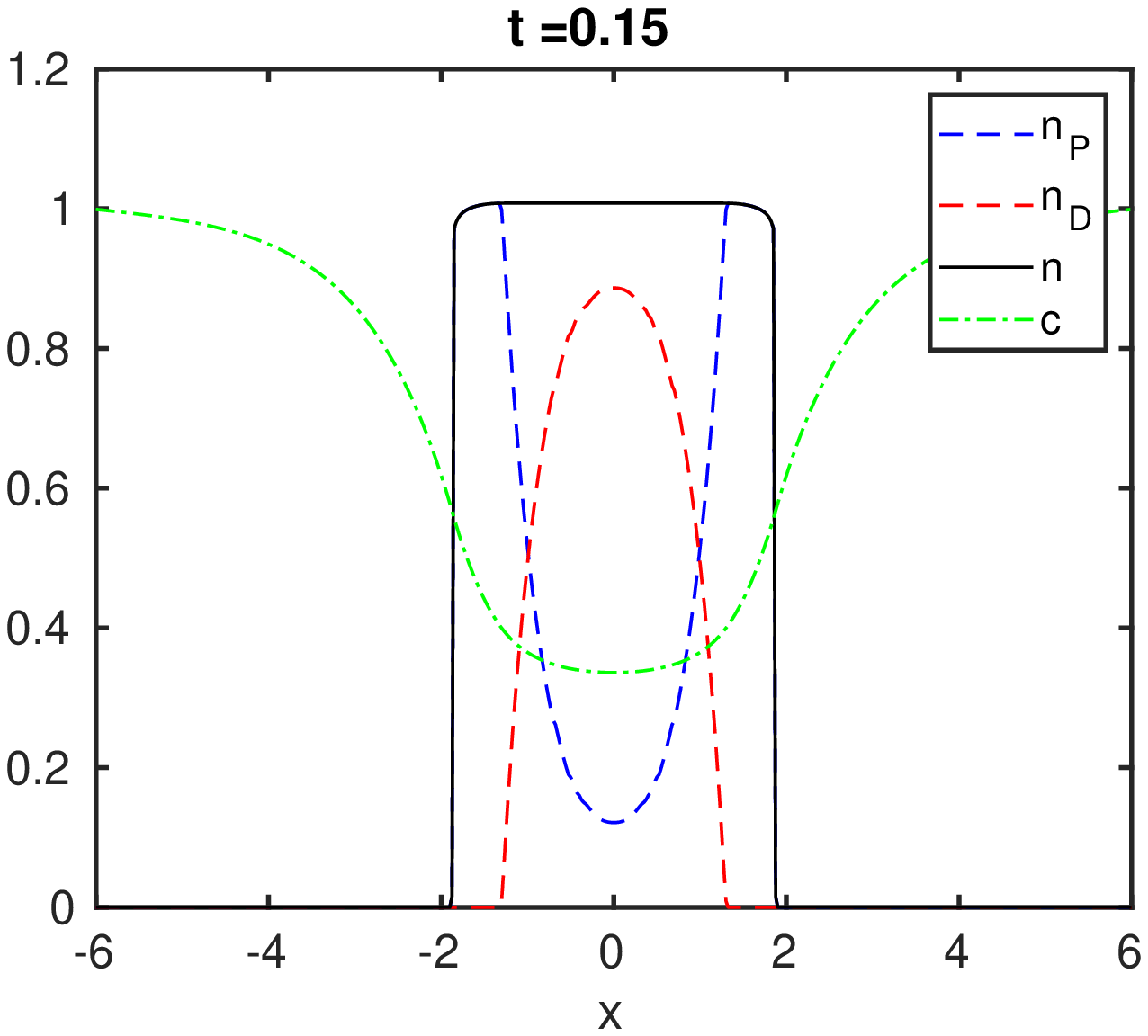}
	\end{subfigure} 
	\begin{subfigure}[b]{0.3\textwidth}
		\includegraphics[width=\textwidth]{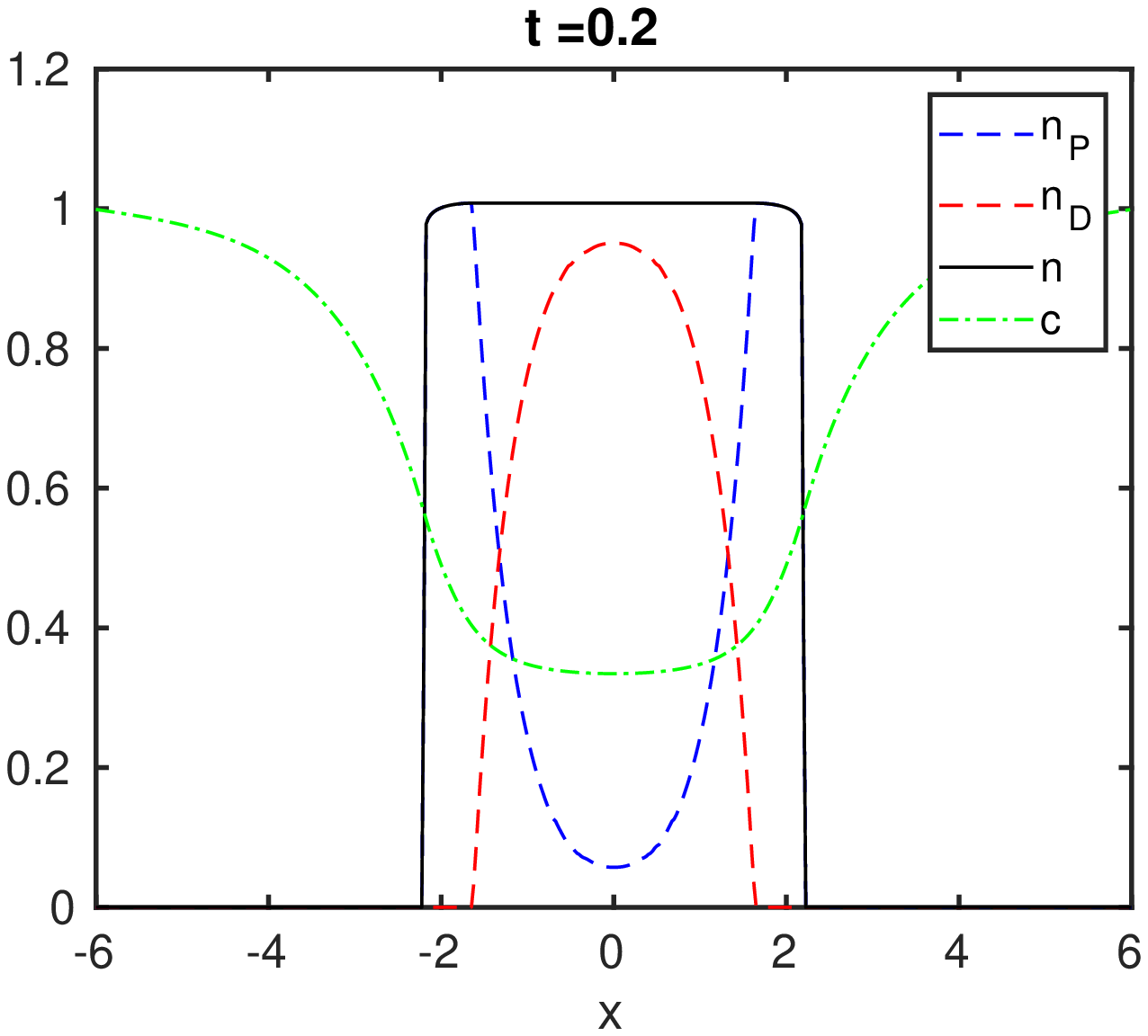}
	\end{subfigure}
	\begin{subfigure}[b]{0.3\textwidth}
		\includegraphics[width=\textwidth]{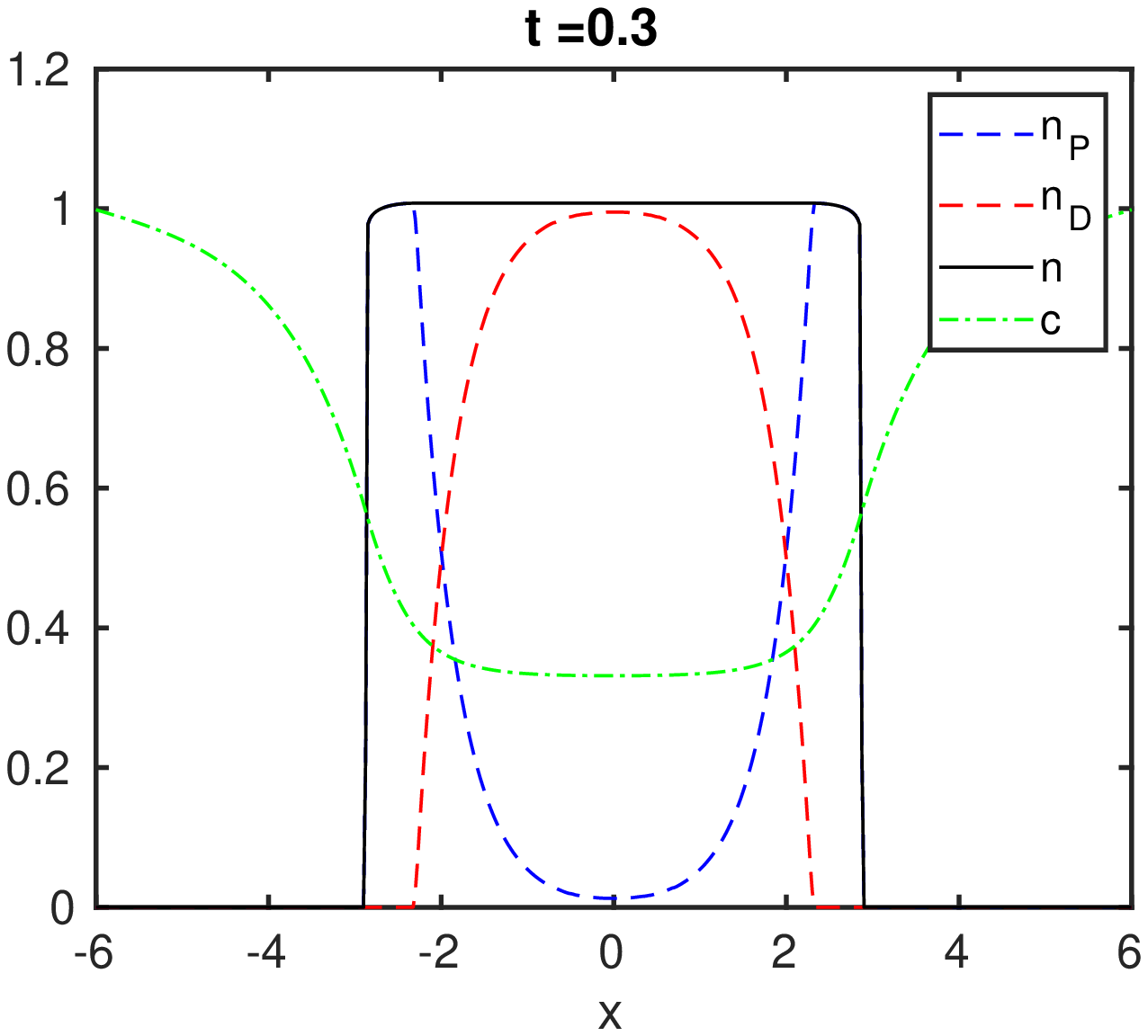}
	\end{subfigure}
	\caption{\textit{In vivo} two-species model in 1D: plot of $n_P, n_D, n, c$ with $\gamma=80$, $\Delta x=0.025$ and $\Delta t=10^{-4}$.} 
	\label{fig: in vivo necrotic 1D}
\end{figure}

\subsection{2D model: the focusing problem}
The focusing solution of the porous medium equation is the solution of Eq. \eqref{eq: n} with an initial data whose support is contained outside of a compact set. At finite time the empty bubble closes up and the topological change of the support generates a singularity of the pressure gradient. In \cite{DP}, the authors show that the pressure gradient is uniformly bounded with respect to $\gamma$ in $L^4(\R^d \times (0,T))$. Then, they prove the sharpness of this uniform bound using the focusing solution as counterexample.
\begin{figure}[hbt!]  
	\centering
		\begin{subfigure}[b]{0.32\textwidth}
		\includegraphics[width=\textwidth]{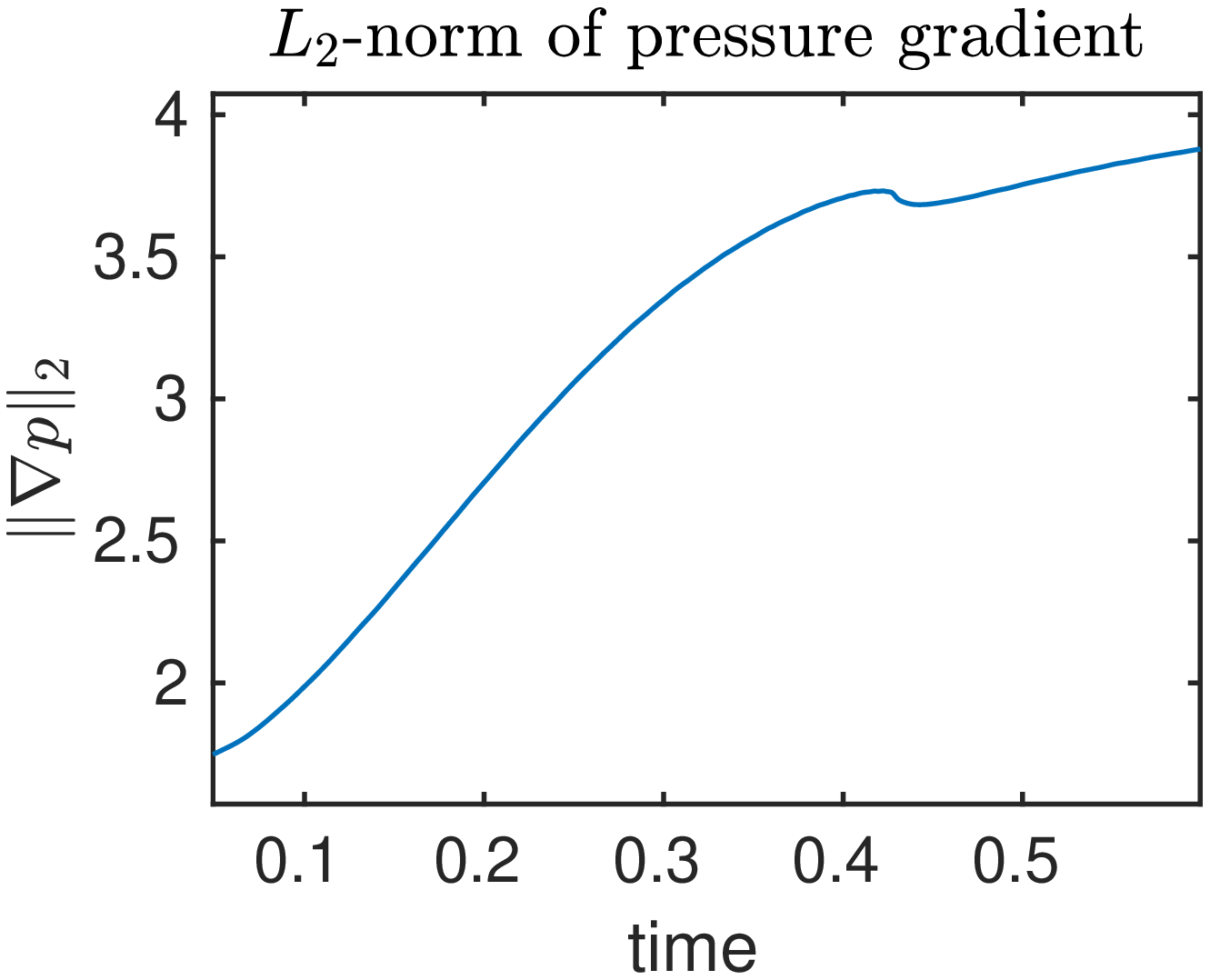}
	\end{subfigure}
	\begin{subfigure}[b]{0.32\textwidth}
		\includegraphics[width=\textwidth]{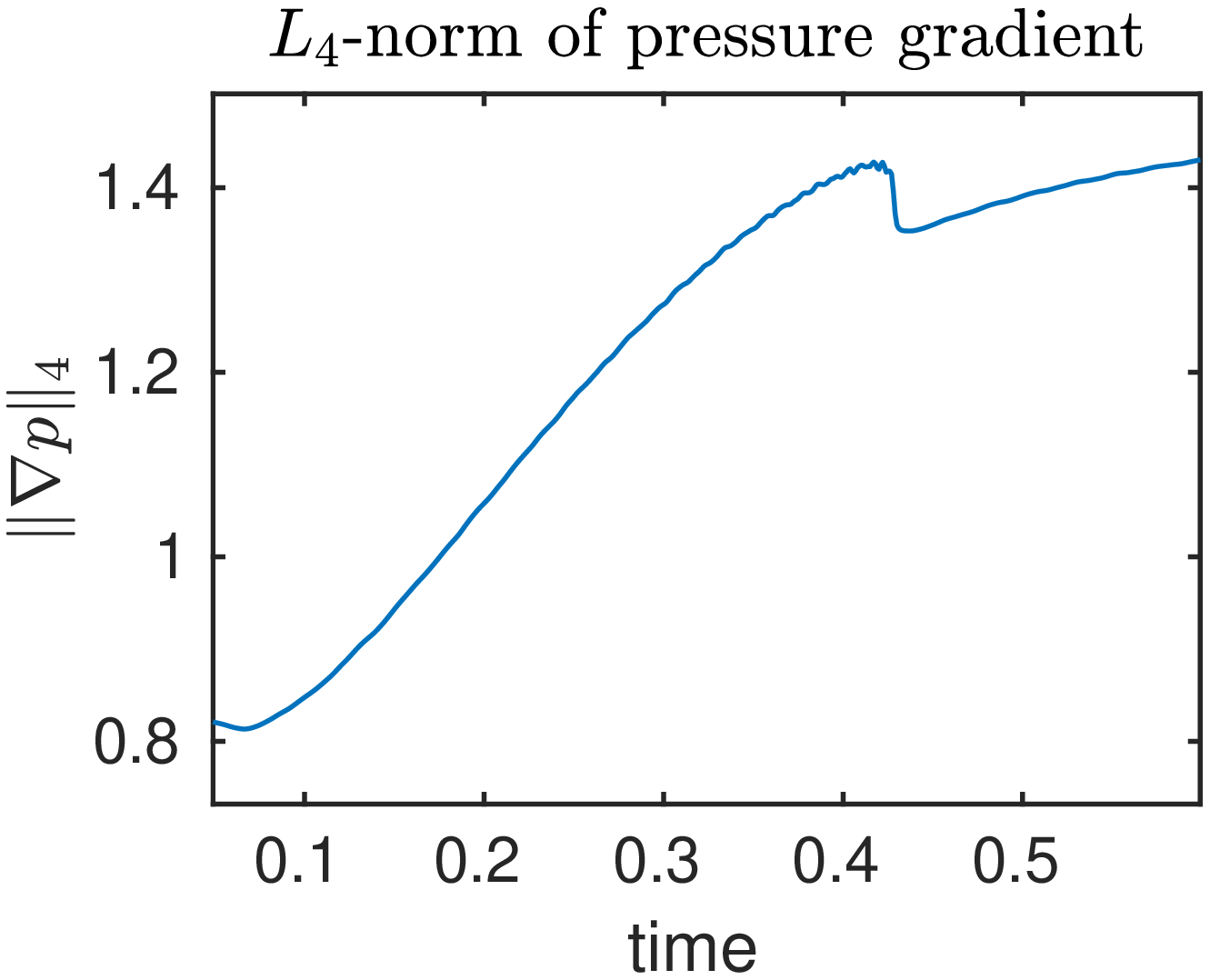}
	\end{subfigure}
	\begin{subfigure}[b]{0.32\textwidth}
		\includegraphics[width=\textwidth]{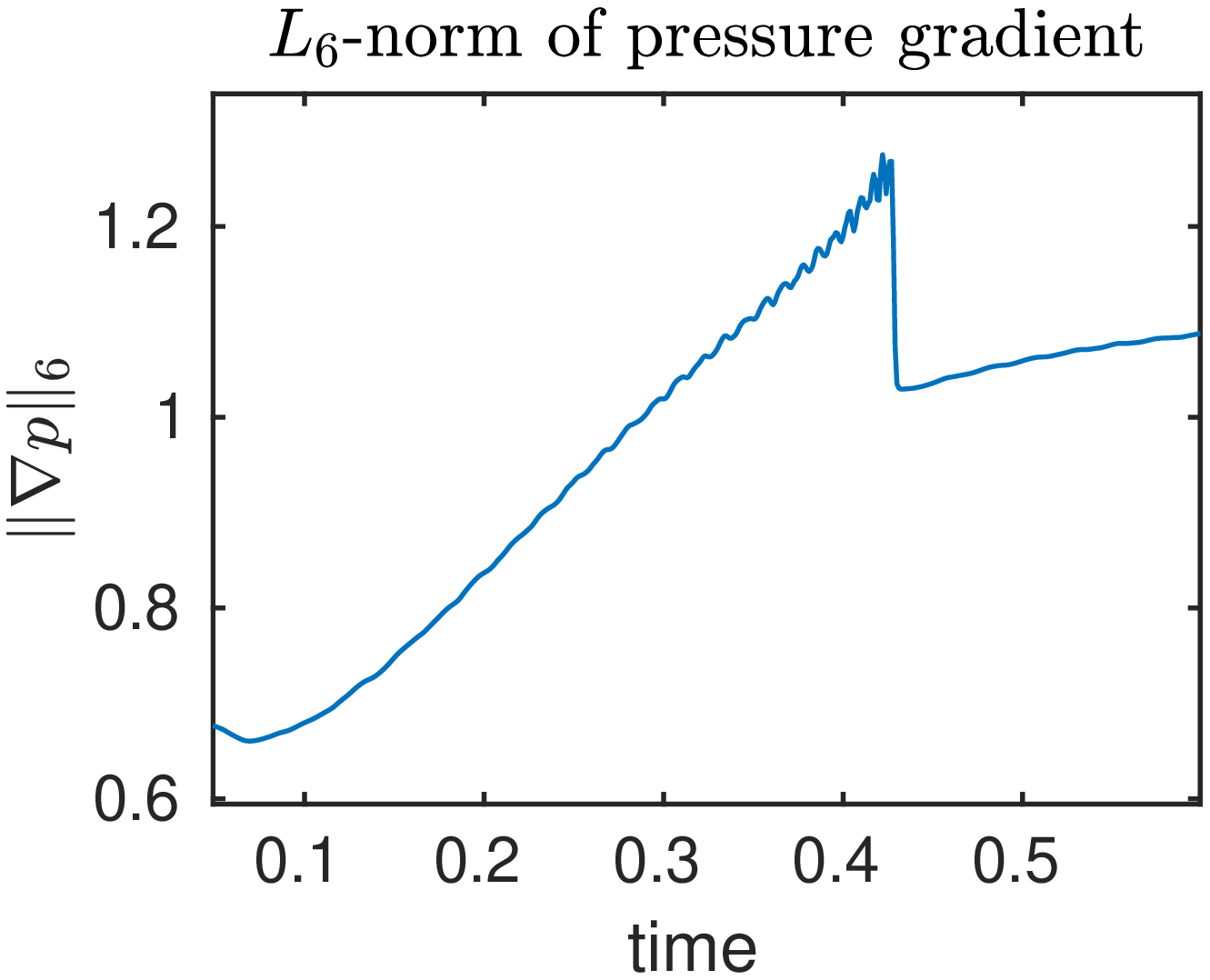}
	\end{subfigure}\\[0.7em]
		\begin{subfigure}[b]{0.32\textwidth}
		\includegraphics[width=\textwidth]{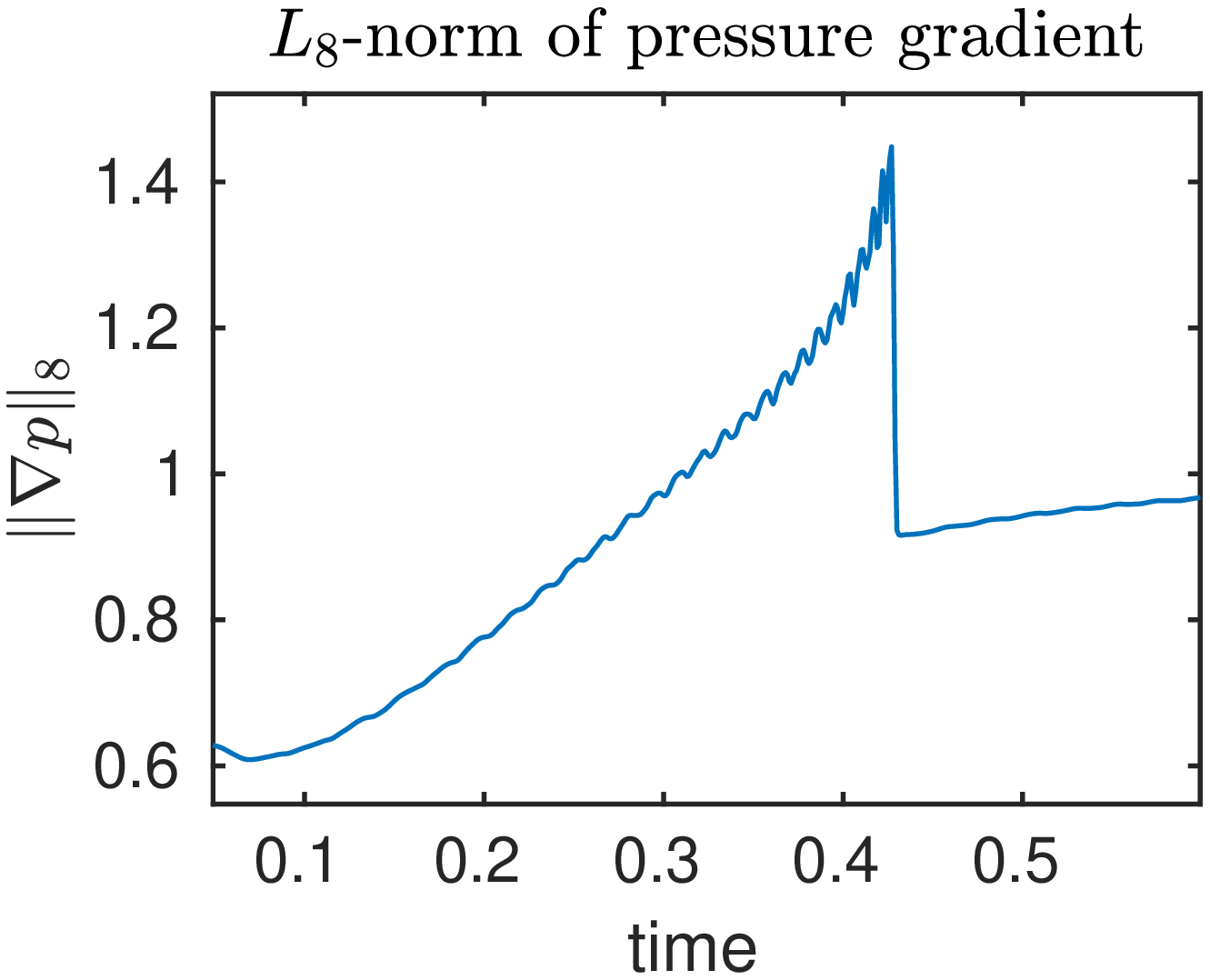}
	\end{subfigure}
	\begin{subfigure}[b]{0.32\textwidth}
		\includegraphics[width=\textwidth]{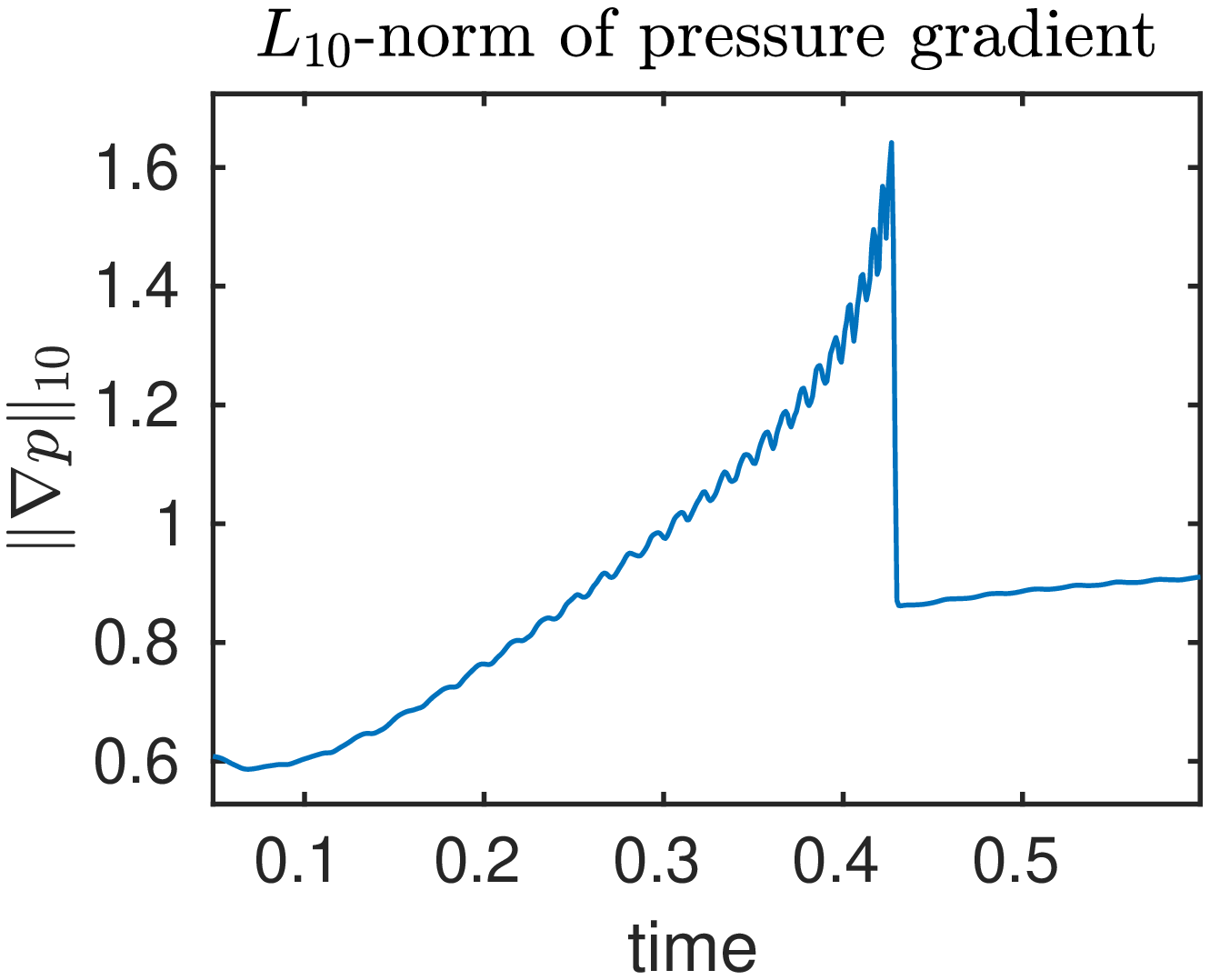}
	\end{subfigure}
	\begin{subfigure}[b]{0.32\textwidth}
		\includegraphics[width=\textwidth]{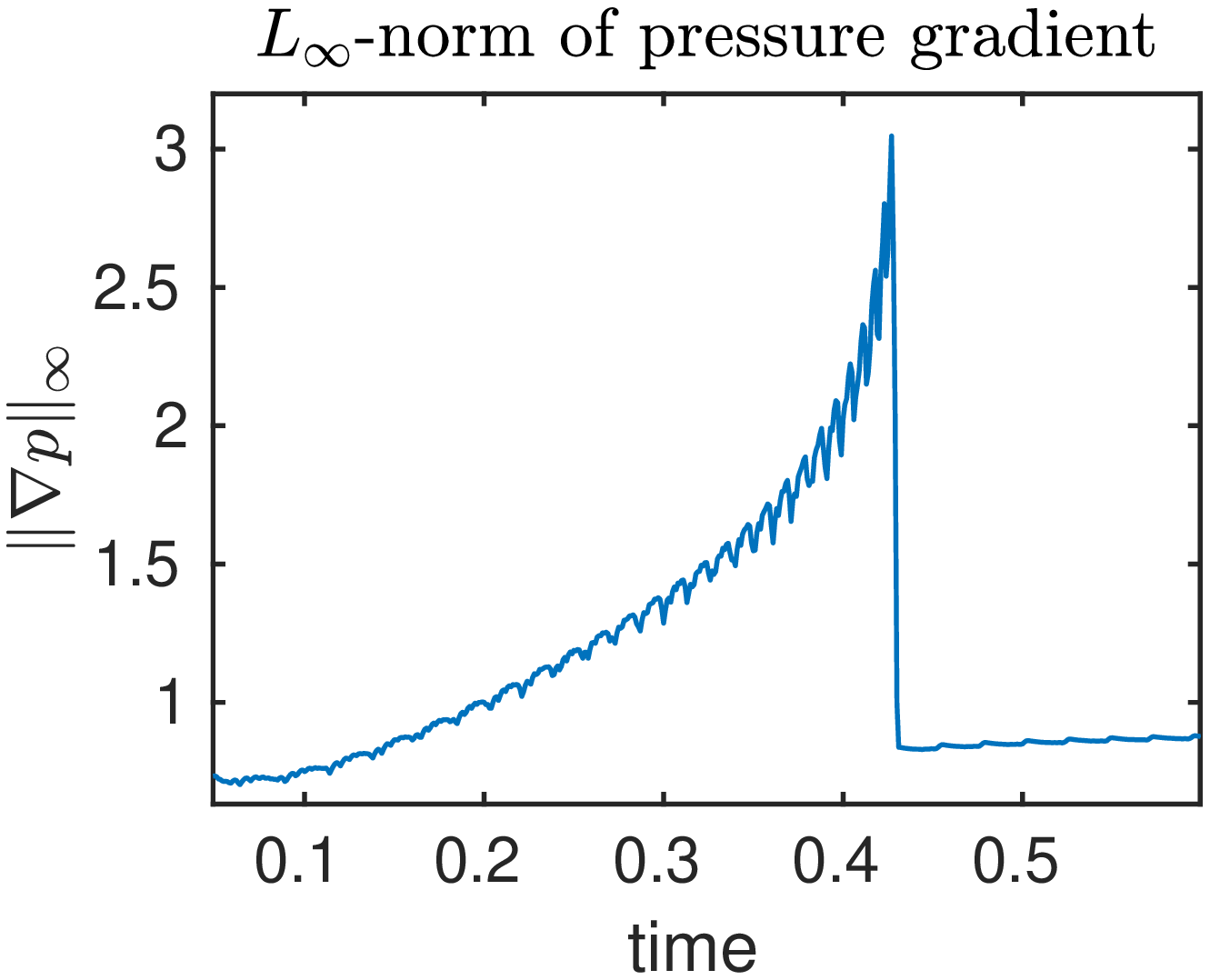}
	\end{subfigure}
	\caption{\textit{Focusing solution: pressure gradient norms.} Plot of the pressure gradient norms along time, from the left ot the right, from the top down, $L^2, L^4, L^6, L^8, L^{10}, L^\infty$-norm, with $\gamma=10$, $\Delta x=0.02$, $\Delta t=0.001$, $p_H=1$ and initial internal radius $0.6$.} \label{fig: gradients gamma 10 ph05}
 \end{figure}

\begin{figure}[hbt!]
	\centering
		\begin{subfigure}[b]{0.4\textwidth}
		\includegraphics[width=\textwidth]{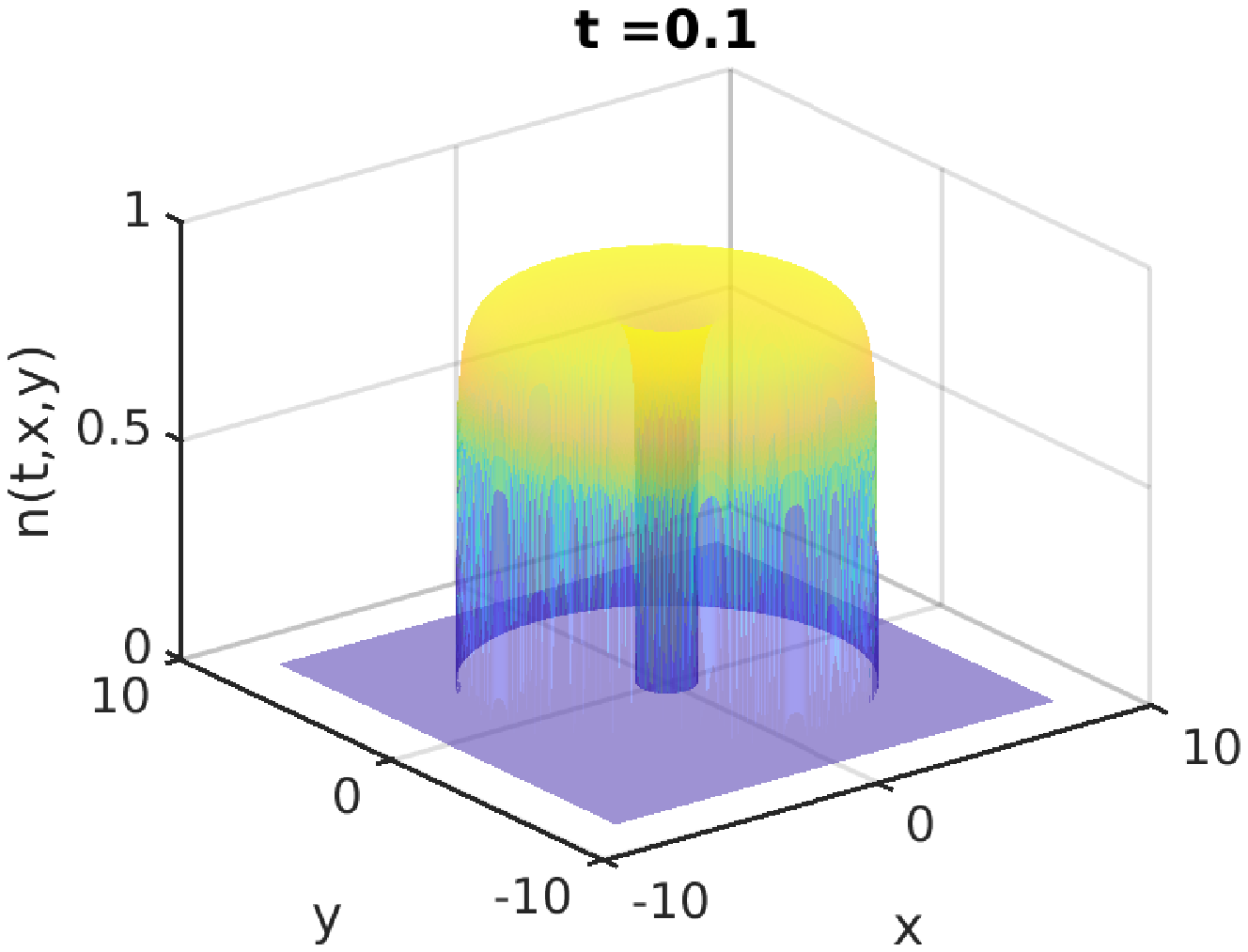}
	\end{subfigure}
		\begin{subfigure}[b]{0.4\textwidth}
		\includegraphics[width=\textwidth]{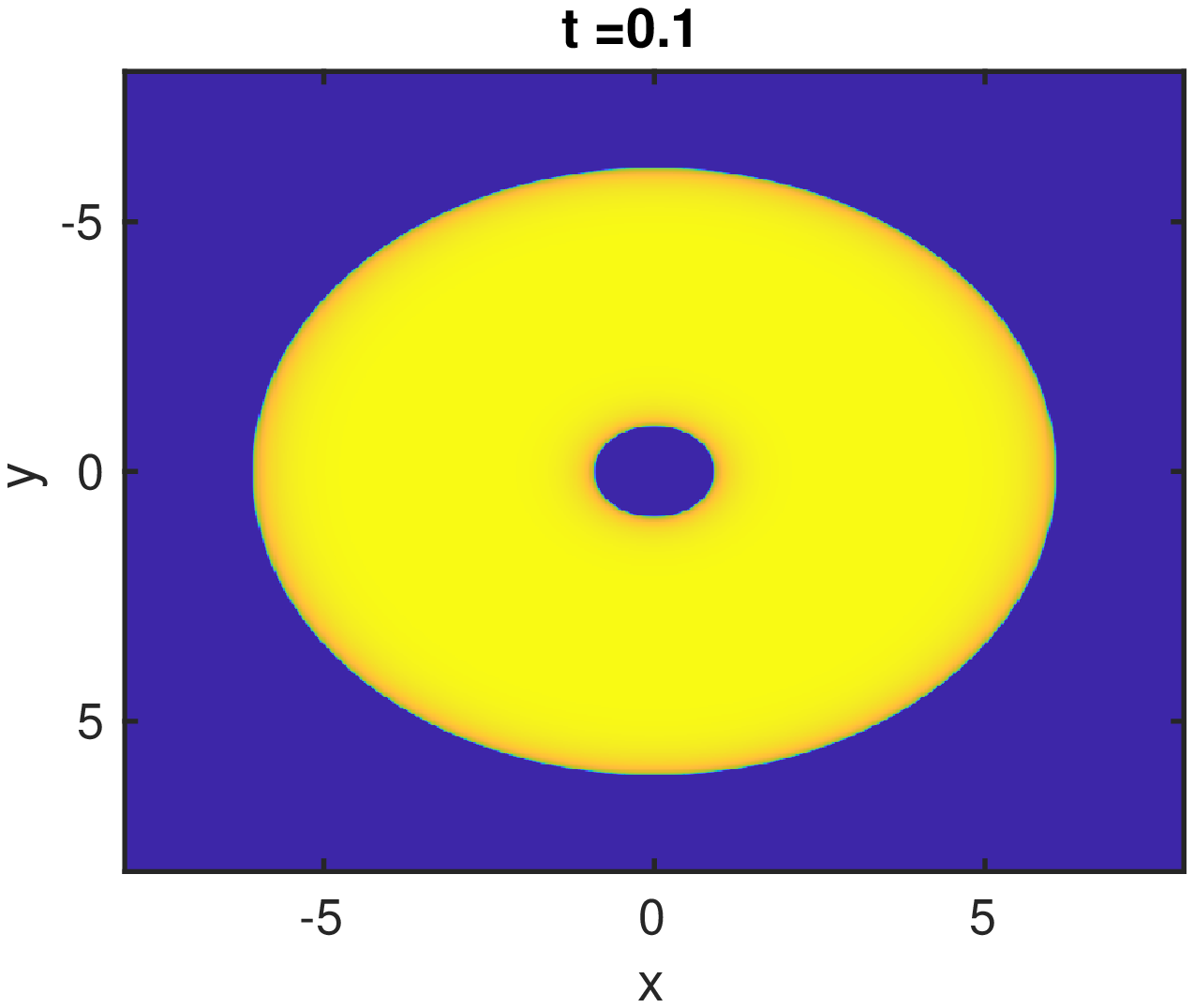}
	\end{subfigure}\\
	\begin{subfigure}[b]{0.4\textwidth}
		\includegraphics[width=\textwidth]{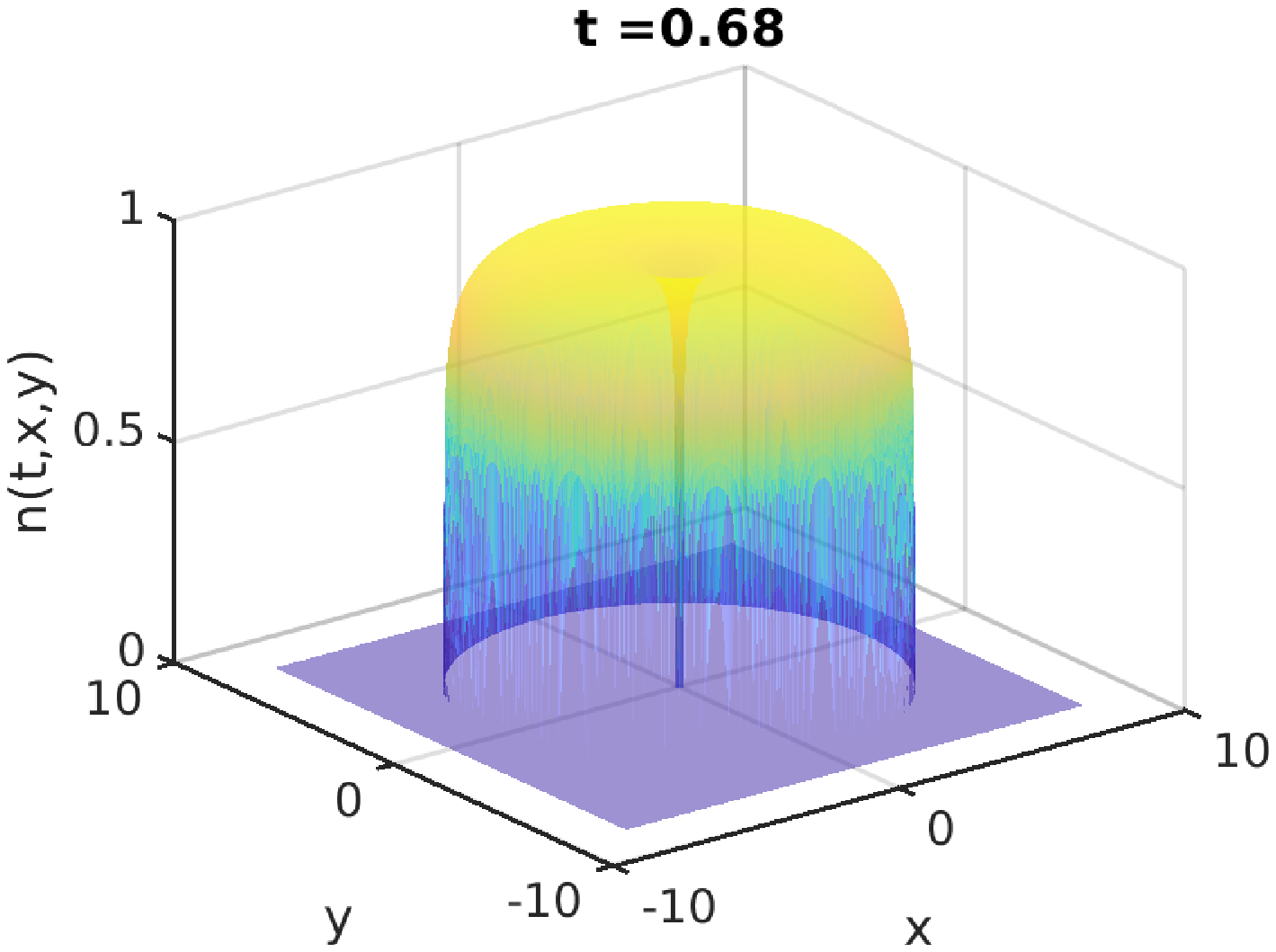}
	\end{subfigure}
		\begin{subfigure}[b]{0.4\textwidth}
		\includegraphics[width=\textwidth]{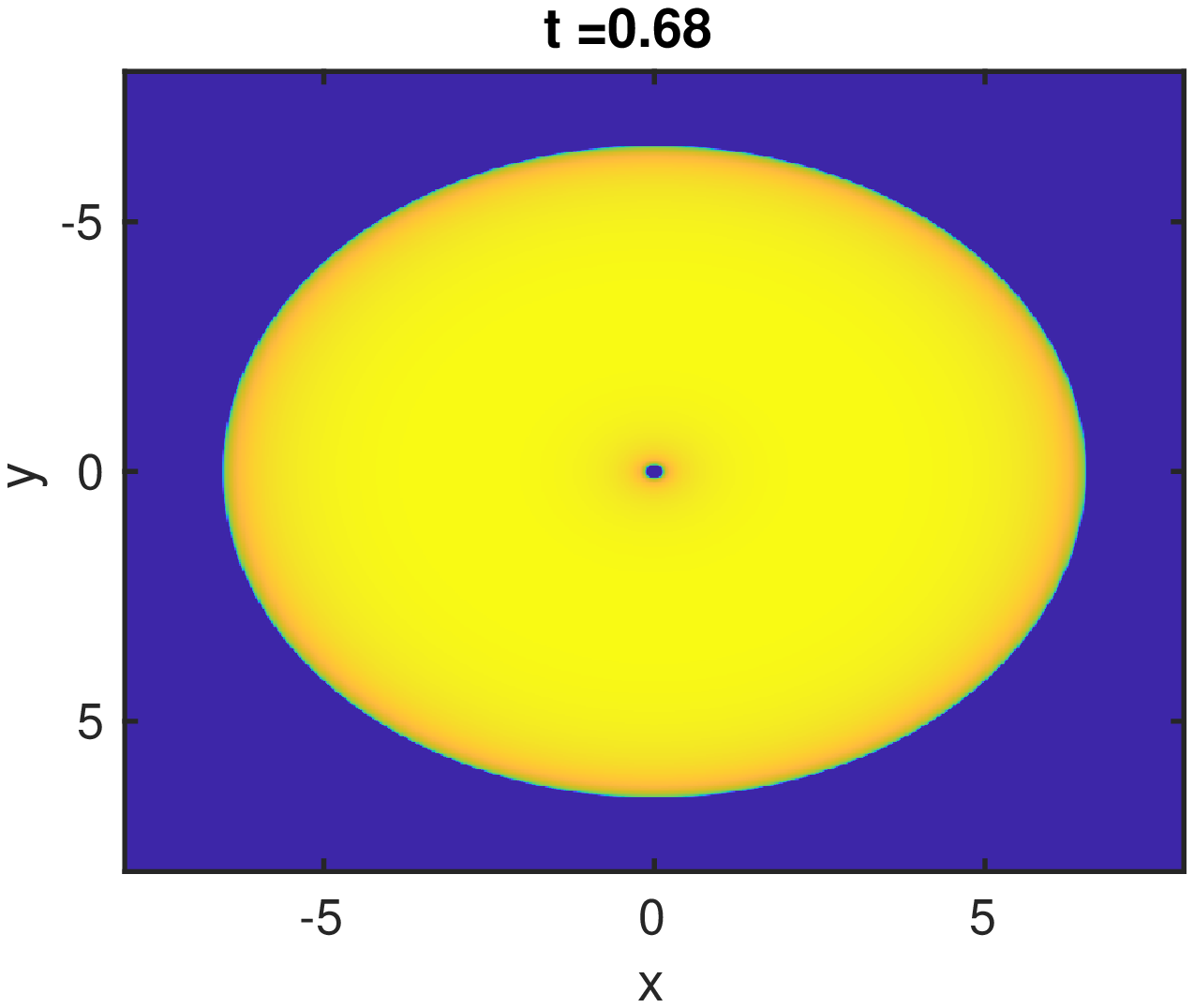}
	\end{subfigure}\\
	\begin{subfigure}[b]{0.4\textwidth}
		\includegraphics[width=\textwidth]{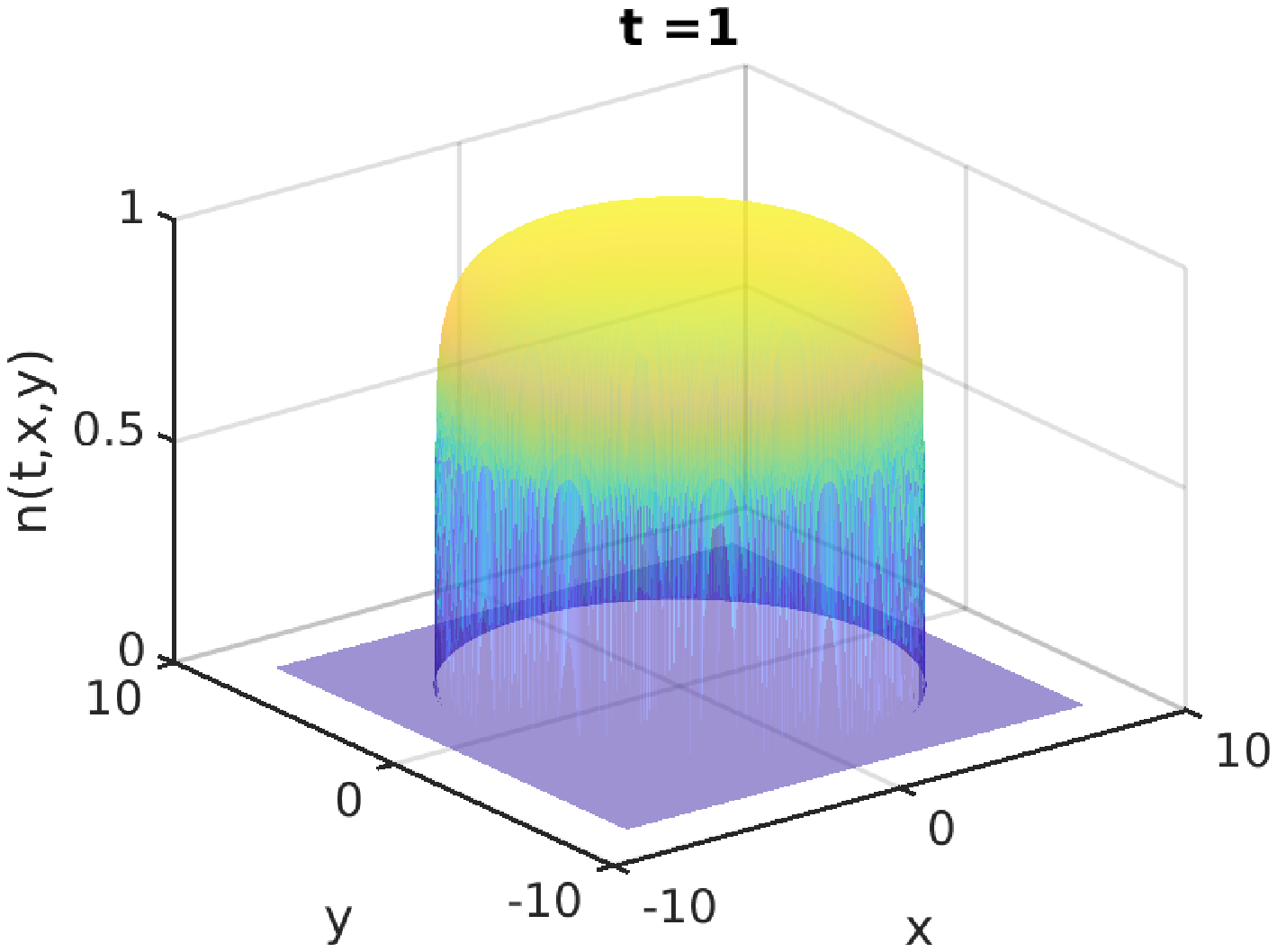}
	\end{subfigure}
		\begin{subfigure}[b]{0.4\textwidth}
		\includegraphics[width=\textwidth]{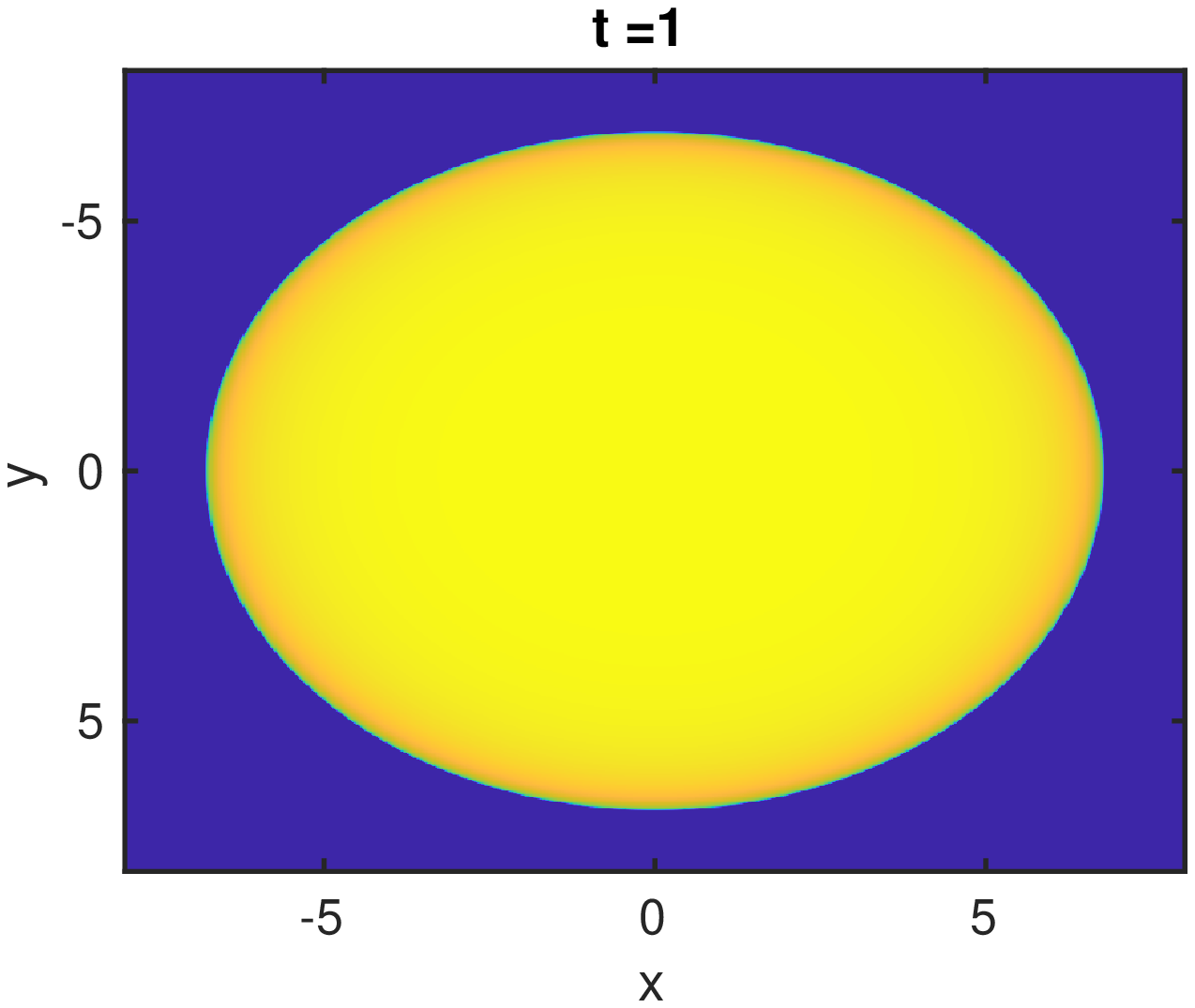}
	\end{subfigure}
	\caption{\textit{Focusing solution (density).} Numerical solution of the focusing problem with $\gamma=10$, $\Delta x=0.02$, initial internal radius $1$.} 
	\label{fig: focusing solution}
 \end{figure} 
 
\begin{figure}[hbt!]
	\centering
		\begin{subfigure}[b]{0.4\textwidth}
		\includegraphics[width=\textwidth]{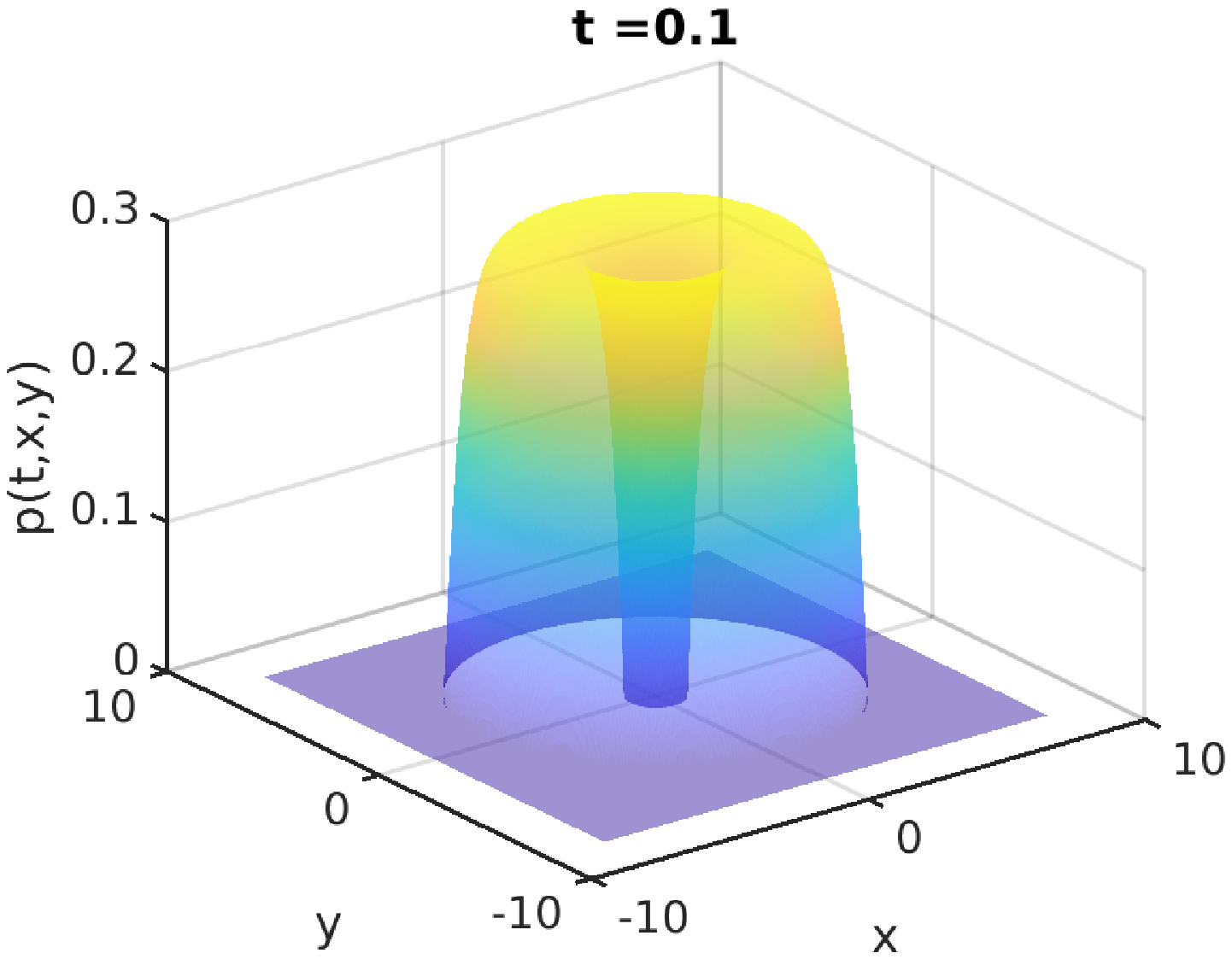}
	\end{subfigure}
		\begin{subfigure}[b]{0.4\textwidth}
		\includegraphics[width=\textwidth]{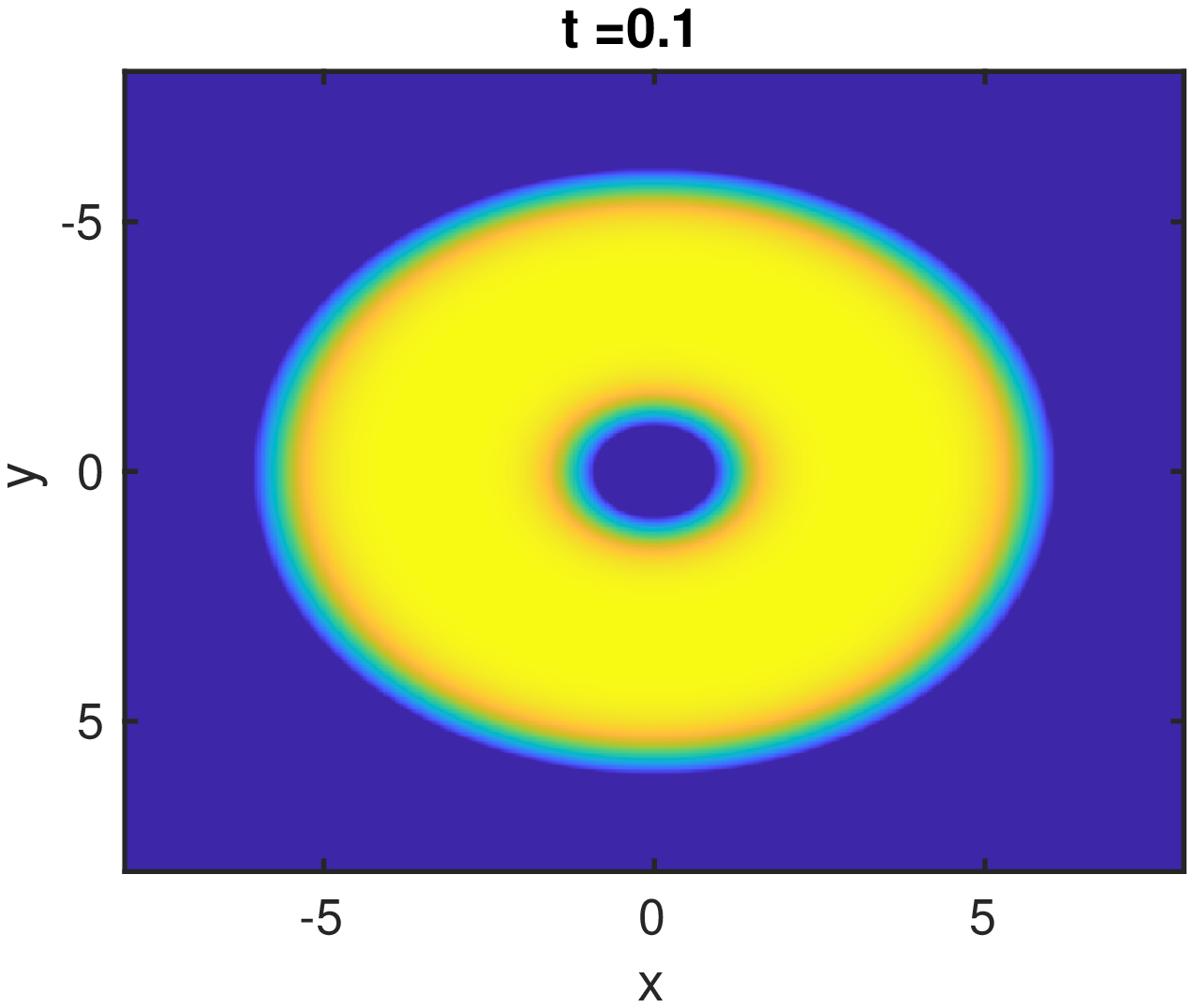}
	\end{subfigure}\\
	\begin{subfigure}[b]{0.4\textwidth}
		\includegraphics[width=\textwidth]{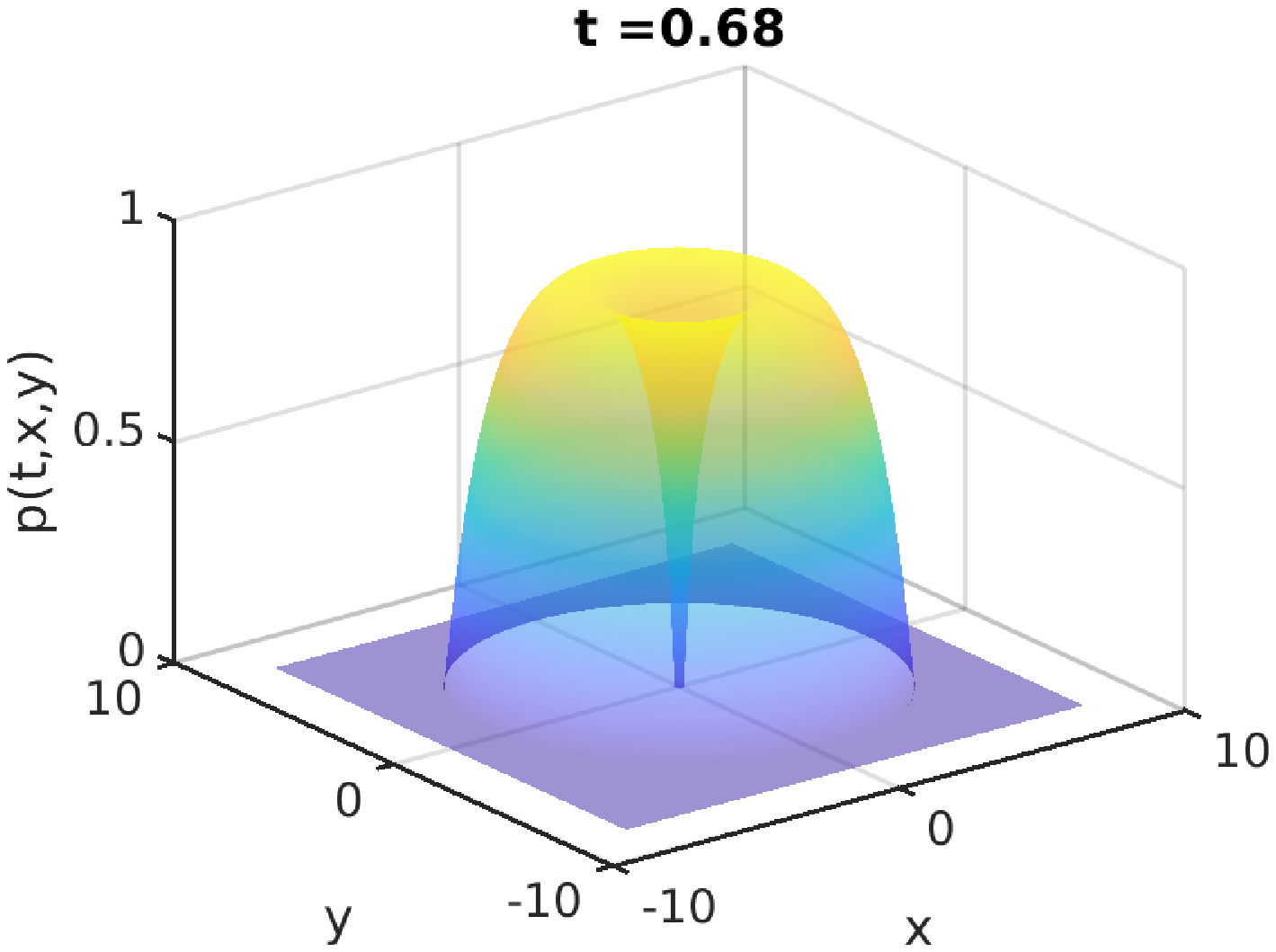}
	\end{subfigure}
		\begin{subfigure}[b]{0.4\textwidth}
		\includegraphics[width=\textwidth]{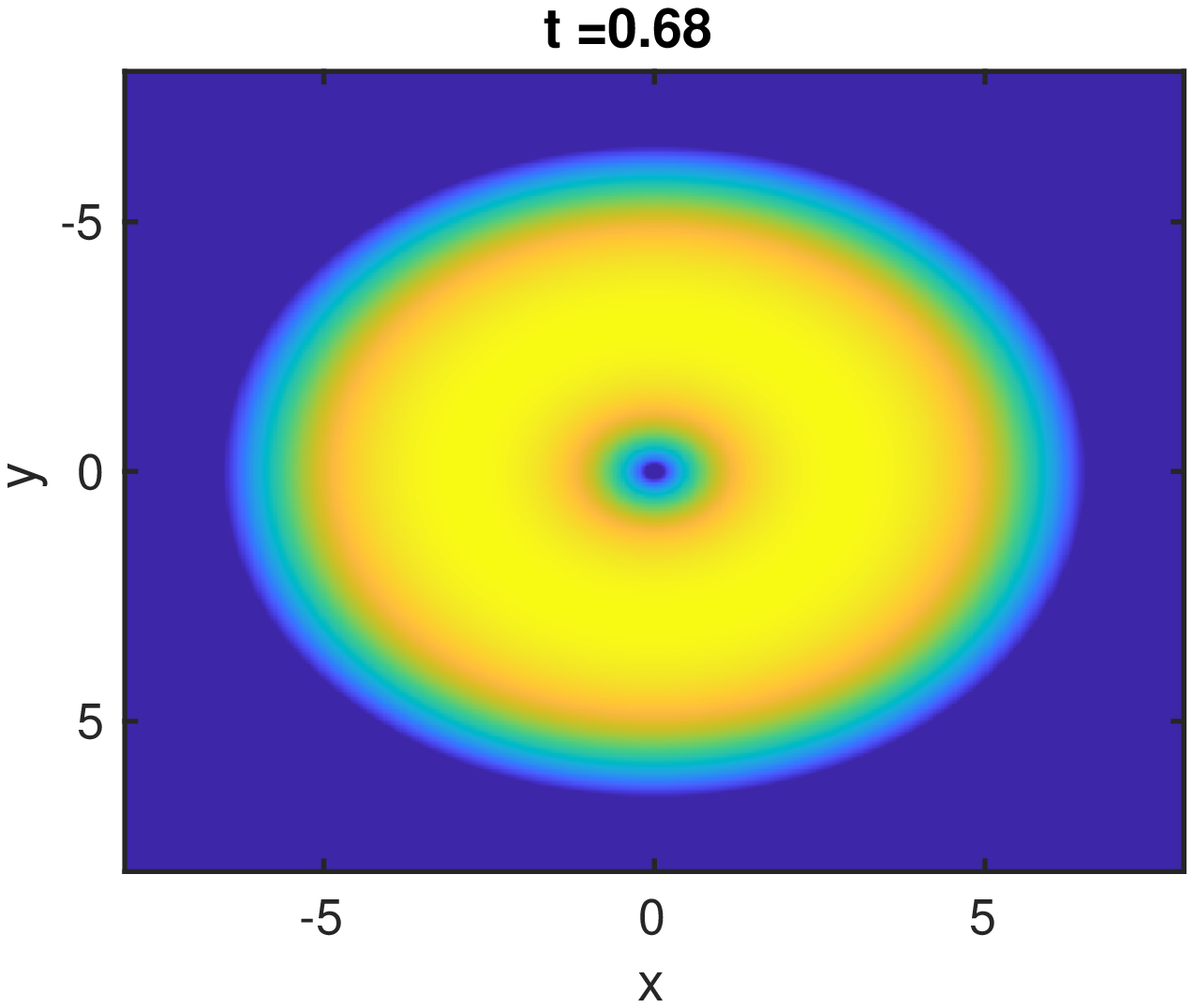}
	\end{subfigure}\\
	\begin{subfigure}[b]{0.4\textwidth}
		\includegraphics[width=\textwidth]{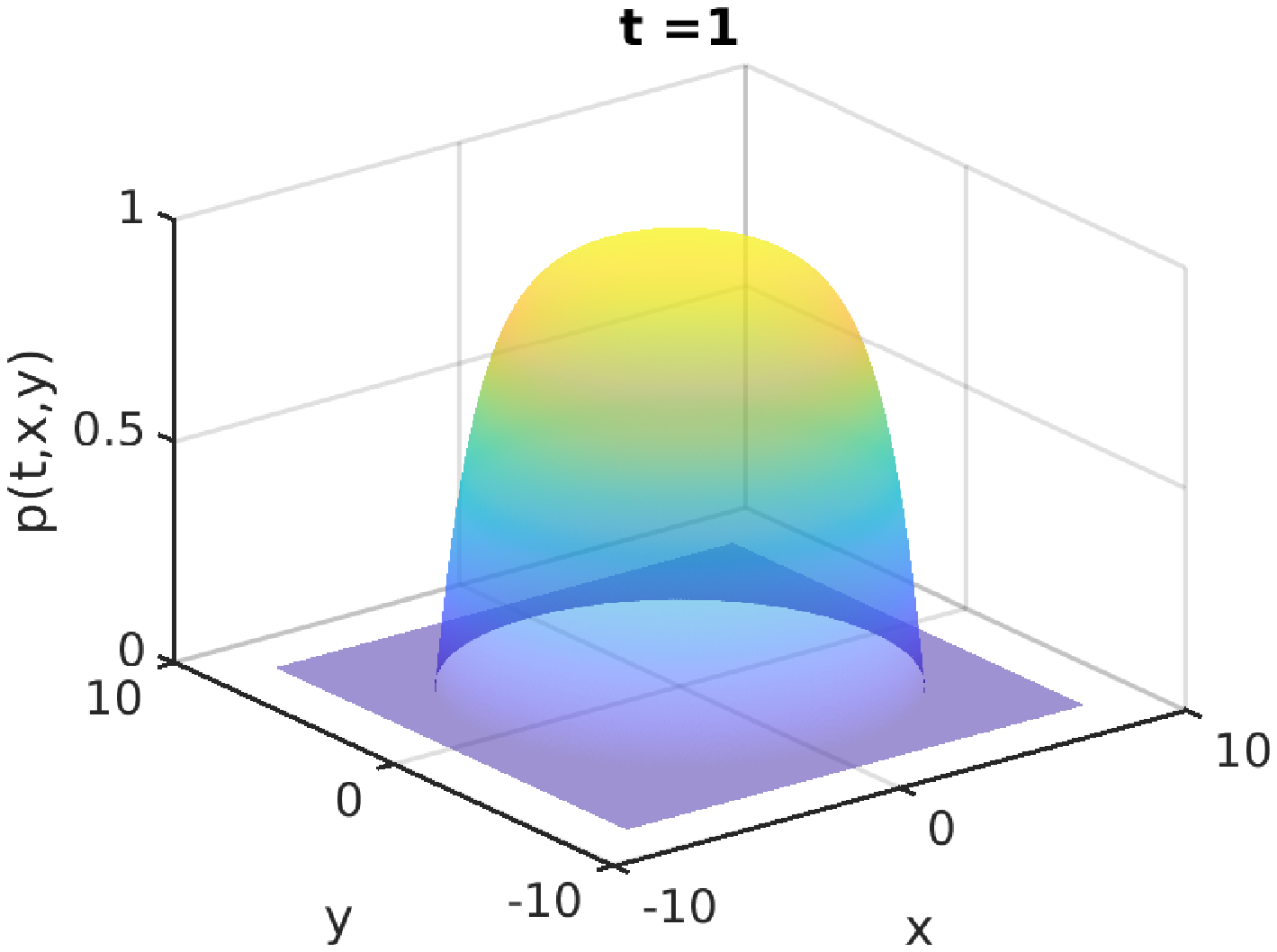}
	\end{subfigure}
		\begin{subfigure}[b]{0.4\textwidth}
		\includegraphics[width=\textwidth]{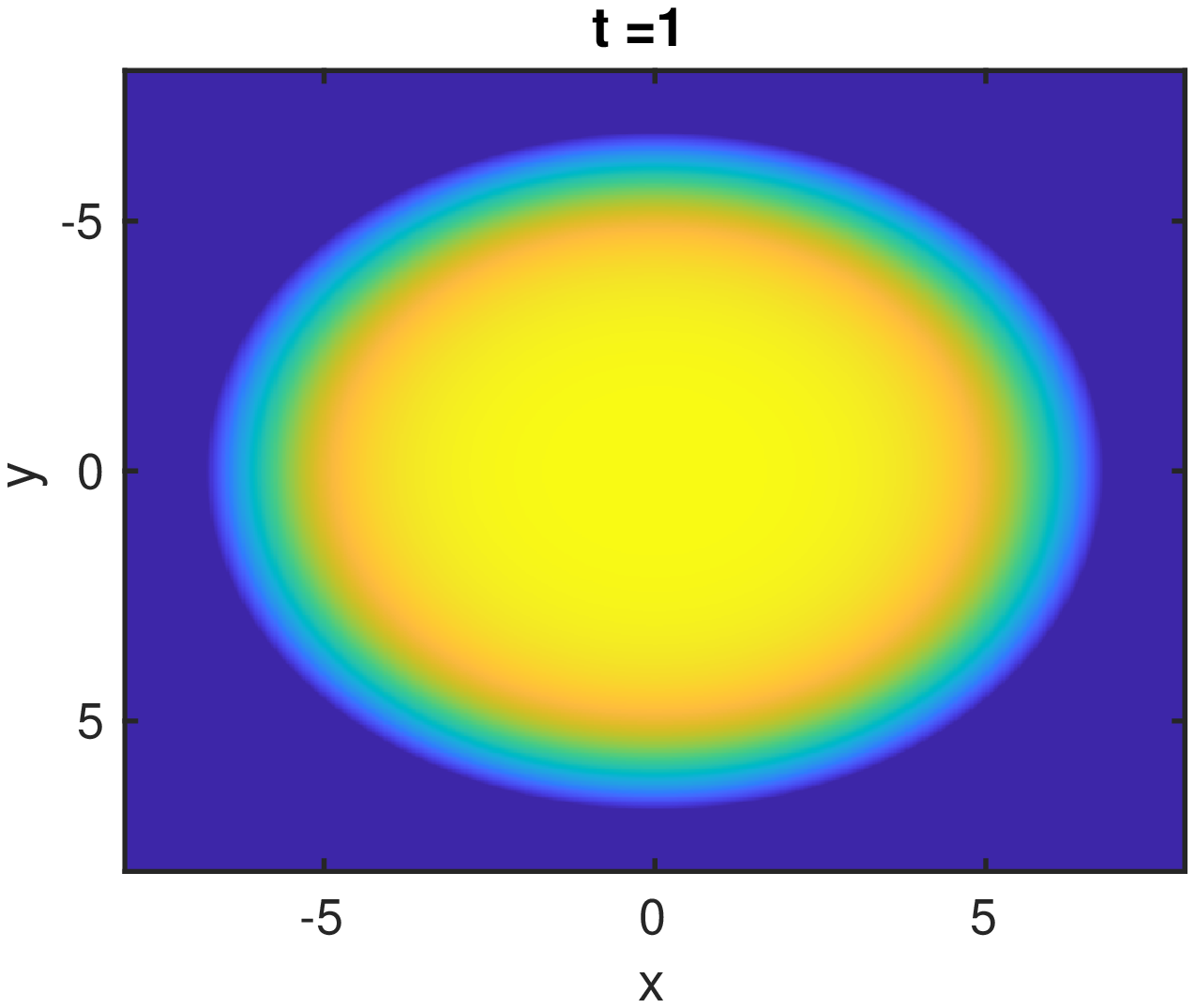}
	\end{subfigure}
	\caption{\textit{Focusing solution (pressure).} Numerical solution of the focusing problem with $\gamma=10$, $\Delta x=0.02$, initial internal radius $1$.} 
	\label{fig: focusing solution pressure}
 \end{figure} 
The Hele-Shaw problem in a spherical shell is defined by the following system
\begin{equation}\label{eq: HS focusing}
    \begin{dcases}
    -\Delta p_\infty = G(p_\infty) &\text{ in } \Omega(t),\\
   \quad  V \quad \,=-\partial_\nu p_\infty &\text { on } \partial \Omega(t),
    \end{dcases}
    \end{equation}
 where $\nu$ and $V$ denote the outward normal and the normal velocity of the free boundary, with 
    \begin{equation*}
   \Omega(t)=\{x; \ R_1(t) \leq|x|\leq R_2(t)\}.   
    \end{equation*}
In \cite{DP} the authors compute the asymptotic behaviour of the $L^p$-norms in space and time of the gradient of a radial solution, choosing for the sake of simplicity a constant reaction term and external radius $R_2(t)=R_2$ fixed. They show that the $L^p$-norms are \textit{uniformly} bounded (with respect to $\gamma$) if and only if $p\leq4$, which confirms that the uniform $L^4$ bound of the PME solution gradient is optimal.

We use our fully discrete scheme \eqref{scheme_implicit} in 2D to verify this interesting behaviour. We approximate the solution of system \eqref{eq: HS focusing}, taking $\gamma=10$, which is a value that well approximate the behaviour of the solution as $\gamma\rightarrow\infty$.

We take as computational domain $[-8,8]\times[-8,8]$ and $G(p)=1-p$. The initial data is given by
\begin{equation}\label{eq: initial data focu}
    n(x,y) = \begin{cases}
    0.8\qquad &\text{if} \quad 0.6< \sqrt{x^2+y^2} <6,\\
    0 \qquad &\text{otherwise}.
\end{cases}
\end{equation} 
The plots of the $L_x^q$-norms of $\grad p(t)$, with $q=2,4,6,8$, are displayed along time in Fig.~\ref{fig: gradients gamma 10 ph05}. We notice that at the focusing time, which is around $t=0.428$, the norms with exponent larger than $4$ develop a singularity.
We also present 3D plots of the solution and its pressure as time evolves, \cf Fig.~\ref{fig: focusing solution} and Fig.~\ref{fig: focusing solution pressure}. In order to better show the behaviour and the shape of the focusing solution, we choose to take a larger initial internal radius. Hence, we take it to be equal to 1 rather than 0.6 in Eq. \eqref{eq: initial data focu}.

\section{Conclusions}
We studied the properties of an upwind finite difference scheme for a mechanical model of tumor growth proving stability results which allowed us to infer the asymptotic preserving property of the scheme in the so-called incompressible limit.
We performed numerical simulations in order to investigate the sharpness of the $L^4$-uniform bound of the pressure gradient, using the focusing solution as limiting example.
 
The question of how to derive the Aronson-B\'enilan estimate for a fixed grid and $\gamma>1$ remains completely open and faces several technical difficulties, due to the stronger non-linearity of the equation. Moreover, as aforementioned, it could be of use in order to pass to the limit as $\Delta x\rightarrow 0$ in the semi-discrete scheme. Extending our approach on the Aronson-B\'enilan estimate to finite difference schemes for cross-reaction-diffusion systems of porous medium type could also represent a challenging problem.

\section*{Acknowledgements}
This project has received funding from the European Union's Horizon 2020 research and innovation program under the Marie Skłodowska-Curie (grant agreement No 754362).\\
The authors are grateful to Beno\^it Perthame for many valuable discussions during the preparation of this paper.
\begin{appendices}
\section{Proof of the solvability of \eqref{scheme_implicit_2} }\label{appendix_existence}
The following theorem, which is a generalization of \cite[Theorem~A.1]{AlmeidaBubbaPerthamePouchol} holds.
\begin{thm}
Denote $\bar{n}_i(t)$ and $\underline{n}_i(t)$ to be two solutions of the system of equations 
\begin{equation}\label{equation:n_proof}
	\frac{d n_i(t)}{dt} + (1-\alpha_i(t))n_i(t) - \nu \left[A(n_i(t),n_{i+1}(t)) - A(n_{i-1}(t),n_i(t)) \right]= N_i^k, \quad i\in I,
\end{equation}
where $A(n_i(t),n_{i+1}(t))$ is defined from \eqref{notation:A}, 
$\nu=\Delta t / \Delta x$, $\alpha_i(t)=\Delta tG(n_i^\gamma(t))$ and $\Delta t <1 / G(0)$,
with a super- and a sub-solution initial data, respectively, i.e. 
\begin{equation*}
    \bar{n}_i(0) = p_H^{\frac{1}{\gamma}}, \quad 
    \underline{n}_i(0) = 0.
\end{equation*}
Then we have \\
(i) $\bar{n}_i(t)$ and $\underline{n}_i(t)$ are nonnegative for all $t>0$ and $i\in I$.\\
(ii) $\bar{n}_i(t)$ and $\underline{n}_i(t)$  are  super- and  sub-solutions for all $t>0$ and $i\in I$.\\
(iii) $\bar{n}_i(t) \ge \underline{n}_i(t)$ for all $t>0$ and $i\in I$.\\
(iv) for any $i\in I$, both $\bar{n}_i(t)$ and $\underline{n}_i(t)$ converges to the same limit, which is the unique solution of  \eqref{scheme_implicit_2}.
\end{thm}
\begin{proof} 
(i) We prove the case of supersolution.
The proof for the case of subsolution is similar. Consider the moment $t^*$ when $\bar{n}_i(t)$ first reach 0 for some $i_0$, i.e. 
$\bar{n}_{i_0}(t^*)=0$ while $\bar{n}_i(t^*) \ge 0$ for all $i\neq i_0$, then 
$A(\bar{n}_{i_0}(t),\bar{n}_{i_0+1}(t))\ge0$, $A(\bar{n}_{i_0-1}(t),\bar{n}_{i_0}(t))\le0$ and thus 
\begin{equation*}
	\frac{d \bar{n}_{i_0}(t^*)}{dt}  = \nu \left[A(\bar{n}_{i_0}(t),\bar{n}_{i_0+1}(t)) - A(\bar{n}_{i_0-1}(t),\bar{n}_{i_0}(t)) \right] + N_{i_0}^k \ge0
\end{equation*}
 via the evolution equation \eqref{equation:n_proof}. 
 As a result, $\bar{n}_i(t)$ can't change signs and thus remain nonnegative for all $t\ge0$. 

(ii) Here we prove the case of subsolution. The proof for the case of supersolution is similar. 
Denote 
\begin{equation*}
	\underline{z}_i(t) = \frac{d\underline{n}_i(t)}{dt}, \quad 
	\underline\alpha_i(t)=\Delta tG(\underline{n}_i^\gamma(t)), \quad
	\underline{A}_{i+\frac12}(t) = A(\underline{n}_i(t), \underline{n}_{i+1}(t)),
\end{equation*}
then $\underline{z}_i(0)\ge0$ for all $i$ since $\underline{n}_i(0)$ is a subsolution. Differentiating \eqref{equation:n_proof}, we get 
\begin{align*}
	&\frac{d \underline{z}_i(t)}{dt} + (1-\underline{\alpha}_i(t))\underline{z}_i(t) -\underline{\alpha}_i'(t)\underline{n}_i(t) 
	- \nu \left[\partial_1\underline{A}_{i+\frac12} - \partial_2 \underline{A}_{i-\frac12} \right]\underline{z}_i(t)\\ 
	=& \nu \partial_2 \underline{A}_{i+\frac12}\underline{z}_{i+1}(t) - \nu \partial_1 \underline{A}_{i-\frac12}\underline{z}_{i-1}(t) .
\end{align*}
Noticing that $\underline{\alpha}_i'(t)=0$ when $\underline{z}_i(t)=0$, the function $\underline{z}_i(t)$ can't change signs following a similar argument as in (i), which implies that $\underline{z}_i(t)\ge0$ for all $t\ge0$.
Then combining with \eqref{equation:n_proof}, we have that
\begin{equation*}
	 (1-\underline\alpha_i(t))\underline{n}_i(t) - \nu \left[\underline{A}_{i+\frac12}(t) - \underline{A}_{i-\frac12}(t) \right] \le N_i^k, \text{ for all } t\ge0,
\end{equation*}
which shows that $\underline{n}_i(t)$ is always a subsolution. 

(iii) Denote $w_i(t) = \bar{n}_i(t) - \underline{n}_i(t)$, then initially we have $w_i(0)\ge0$ for all $i$. 
We wish to show that $w_i(t)\ge0$ for all $t\ge0$ and $i\in I$. 
For simplicity of notations, we denote 
\begin{equation*}
	\bar{\alpha}_i(t) = \Delta t G(\bar{n}_i^\gamma(t)), \quad
	\underline{\alpha}_i(t) = \Delta t G(\underline{n}_i^\gamma(t)).
\end{equation*}
Noticing \eqref{eq: assumptions G} and the fact that both $\bar{n}_i(t)$ and $\underline{n}_i(t)$ are nonnegative, 
when $\Delta t<1/G(0)$, we have $\bar{\alpha}_i(t) \le 1$ and $\underline{\alpha}_i(t) \le 1$.
A direct computation shows that 
\begin{align*}
	(1-\bar{\alpha}_i(t)) \bar{n}_i(t) - (1-\underline{\alpha}_i(t)) \underline{n}_i(t) 
	& =
	(1-\bar{\alpha}_i(t)) w_i(t) + (\underline{\alpha}_i(t)-\bar{\alpha}_i(t)) \underline{n}_i(t) \\
	& = (1-\bar{\alpha}_i(t) + \beta_i(t)) w_i(t),
\end{align*}
where $\beta_i(t) = -\Delta t G'(\eta_i^\gamma(t)) \gamma\eta_i^{\gamma-1}(t)\underline{n}_i(t)\ge 0$ for some nonnegative $\eta_i(t)$ between $\bar{n}_i(t)$ and $\underline{n}_i(t)$.
By \eqref{equation:n_proof} and the fact that  $\bar{n}_i(t)$ and $\underline{n}_i(t)$ are  super- and  subsolutions, we have
\begin{align*}
	(1-\bar{\alpha}_i(t) + \beta_i(t)) w_i(t) 
	&-\nu\left[A_{i+\frac12}(\bar{n}_i, \bar{n}_{i+1})-A_{i-\frac12}(\bar{n}_{i-1}, \bar{n}_i)\right]\\
	&+\nu\left[A_{i+\frac12}(\underline{n}_i, \underline{n}_{i+1})-A_{i-\frac12}(\underline{n}_{i-1}, \underline{n}_i)\right]
	\ge 0. 
\end{align*}
Combining the above inequality with the following expression 
\begin{equation*}
	A_{i+\frac12}(\bar{n}_i, \bar{n}_{i+1}) - A_{i+\frac12}(\underline{n}_i, \underline{n}_{i+1}) 
	= 
	\partial_1 A_{i+\frac12}(\xi_i, \bar{n}_{i+1}) w_i + \partial_2 A_{i+\frac12}(\underline{n}_i, \eta_{i+1}) w_{i+1},
\end{equation*}
where $\partial_1 A_{i+\frac12}(\xi_i, \bar{n}_{i+1})\le0$ with some $\xi_i$ between $\bar{n}_i$ and $\underline{n}_i$ and $\partial_2 A_{i+\frac12}(\underline{n}_i, \eta_{i+1})\ge0$ with some $\eta_{i+1}$ between $\bar{n}_{i+1}$ and $\underline{n}_{i+1}$,  we have
\begin{align*}
	&\left[1-\bar{\alpha}_i(t) + \beta_i(t) - \nu \left(\partial_1 A_{i+\frac12}(\xi_i, \bar{n}_{i+1}) - \partial_2 A_{i-\frac12}(\underline{n}_{i-1}, \eta_i)\right)\right] w_i(t) \\
	& - \nu \partial_2 A_{i+\frac12}(\underline{n}_i, \eta_{i+1}) w_{i+1}(t) 
	+ \nu \partial_1 A_{i-\frac12}(\xi_{i-1}, \underline{n}_i) w_{i-1}(t) \ge 0.
\end{align*} 
Multiplying both sides by $\mathds{1}_{\{w_i<0\}}$ and summing over $i$, we get
\begin{align*}
	&- \sum_i (1-\bar{\alpha}_i(t) + \beta_i(t))w_i^- 
	+ I_1 + I_2 \ge0,
\end{align*}
where $w_i^- = \max\{ -w_i, 0\}$ and 
\begin{align*}
&I_1 = \nu \sum_i  \partial_2 A_{i-\frac12}(\underline{n}_{i-1}, \eta_i) w_i(\mathds{1}_{\{w_i<0\}} - \mathds{1}_{\{w_{i-1}<0\}}),\\
&I_2 = -\nu \sum_i  \partial_1 A_{i+\frac12}(\xi_{i}, \underline{n}_{i+1}) w_i(\mathds{1}_{\{w_i<0\}} - \mathds{1}_{\{w_{i+1}<0\}}).
\end{align*}
It is worth noticing that 
\begin{equation*}
	w_i(\mathds{1}_{\{w_i<0\}} - \mathds{1}_{\{w_{i\pm1}<0\}}) \le0,
\end{equation*}
which implies that $I_1\le0$, $I_2\le0$ and further
\begin{equation}\label{proof:w_sign2}
	 \sum_i (1-\bar{\alpha}_i(t) + \beta_i(t))w_i^- \le 0.
\end{equation}
It is easy to see from \eqref{proof:w_sign2} that we must have 
$w_i^-(t)\equiv 0$, i.e. $w_i(t)\ge0$ for all $t>0$.

\textit{(iv)} The monotonicity of $\bar{n}_i(t)$ and $\underline{n}_i(t)$ indicates that there exist the limits 
\begin{equation*}
	\bar{N}_i = \lim_{t\to\infty} \bar{n}_i(t), \quad 
	\underline{N}_i = \lim_{t\to\infty} \underline{n}_i(t).
\end{equation*}
Denote $W_i = \bar{N}_i - \underline{N}_i$, we can show that 
\begin{align*}
	&\left[1-\Delta t G(\bar{N}_i^\gamma) + \beta_i(t) - \nu \left(\partial_1 A_{i+\frac12}(\xi_i, \bar{N}_{i+1}) - \partial_2 A_{i-\frac12}(\underline{N}_{i-1}, \eta_i)\right)\right] W_i \\
	& - \nu \partial_2 A_{i+\frac12}(\underline{N}_i, \eta_{i+1}) W_{i+1}+ \nu \partial_1 A_{i-\frac12}(\xi_{i-1}, \underline{N}_i) W_{i-1} = 0,
\end{align*}
for some $\xi_i$'s and $\eta_i$'s. 
Summing over all $i$, we have 
\begin{equation*}
	\sum_i \left[1-\Delta t G(\bar{N}_i^\gamma) + \beta_i(t) \right] W_i = 0.
\end{equation*}
Noticing that $W_i\ge0$  and $1-\Delta t G(\bar{N}_i^\gamma)+ \beta_i(t)>0$, we have $W_i = 0$ for all $i\in I$. 
In other words, for each $i$, there is a unique limit of $\bar{n}_i(t)$ and $\underline{n}_i(t)$ as $t\to\infty$, which is $N_i^{k+1}$, the unique solution of \eqref{scheme_implicit_2}. 

\end{proof}
\end{appendices}

\bibliographystyle{abbrv}
\bibliography{biblio}

\end{document}